\documentclass[twoside,11pt]{amsart}
\usepackage{amsmath,latexsym,amssymb,times,mathptm, enumerate, xypic}
\usepackage{verbatim}
\usepackage[font=small,format=plain,labelfont=bf,up,textfont=it,up]{caption}

\def\sqr#1#2{{\vcenter{\hrule height.#2pt
        \hbox{\vrule width.#2pt height#1pt \kern#1pt
                \vrule width.#2pt}
        \hrule height.#2pt}}}
\def\square{\mathchoice\sqr64\sqr64\sqr{4}3\sqr{3}3}
\def\QED{\hfill$\square$}
\def\Sym{{\rm Sym}}

\newtheorem{theorem}{Theorem}[section]

\newtheorem{corollary}[theorem]{Corollary}
\newtheorem{lemma}[theorem]{Lemma}

\newtheorem{observation}[theorem]{Observation}

 \theoremstyle{definition}

\newtheorem{definition}[theorem]{Definition}
\newtheorem{data}[theorem]{Data}

\newtheorem{Table}[theorem]{Table}

\newtheorem{remark}[theorem]{Remark}

\newtheorem*{note}{Note}
\newtheorem{remarks}[theorem]{Remarks}
\newtheorem{example}[theorem]{Example}

\newtheorem*{proof1}{Proof of Theorem \ref{claudia}}
\newtheorem*{proof2}{Proof of Theorem \ref{sextic-table}}
\def\lto{\longrightarrow}
\def\p{{\mathfrak p}}

\newcommand{\f}[1]{\ensuremath{\mathfrak{#1}}}

\def\m{{\mathfrak m}}
\def\p{{\mathfrak p}}

\numberwithin{equation}{theorem}
\numberwithin{table}{theorem}

\def\BalH{{\rm BalH}}
\def\GL{{\rm GL}}
\def\DO{{\rm DO}}
\def\Bal{{\rm Bal}}
\def\syl{{\rm syl}}

\begin{document}
\baselineskip=16pt

\title[The bi-graded structure of
Symmetric Algebras with applications to Rees rings]{\bf The bi-graded
structure of Symmetric Algebras with applications to Rees rings}
\author[A. Kustin, C. Polini, and B. Ulrich]
{Andrew Kustin, Claudia Polini, \and Bernd Ulrich}

\address{Department of Mathematics, University of South Carolina,
Columbia, SC 29208} \email{kustin@math.sc.edu}

\address{Department of Mathematics, University of Notre Dame,
Notre Dame, IN 46556} \email{cpolini@nd.edu}

\address{Department of Mathematics, Purdue University,
West Lafayette, IN 47907} \email{ulrich@math.purdue.edu}

\thanks{AMS 2010 {\em Mathematics Subject Classification}.
Primary  13A30, 14H50; Secondary 14H20, 13D02,  14E05.}

\thanks{The first author was partially supported by the NSA and the Simons Foundation.
The second author was partially supported by the NSF and the NSA.
The last author was partially supported by the NSF and as a Simons Fellow.}


\begin{abstract} Consider a rational projective plane curve $\mathcal C$ parameterized by three homogeneous forms of the same degree in the polynomial ring 
$R=k[x,y]$ over a field $k$. The ideal $I$ generated by these forms is presented by a homogeneous $3 \times 2$ matrix $\varphi$ 
with column degrees $d_1 \leq d_2$. The Rees algebra $\mathcal R=R[It]$ of $I$ is the bi-homogeneous coordinate ring of the graph of the parameterization of $\mathcal C$; and accordingly, there is a dictionary that translates between the singularities of $\mathcal C$ and algebraic properties of the ring $\mathcal R$ and its defining ideal. 
Finding the defining equations of Rees rings is a classical problem in elimination theory that amounts to determining the kernel $\mathcal A$ of the natural map from the symmetric algebra $\Sym(I)$ onto $\mathcal R$. The ideal $\mathcal A_{\ge d_2-1}$, which is an approximation of $\mathcal A$, can be obtained using linkage. We exploit the bi-graded structure of $\Sym(I)$ in order to describe the structure of an improved approximation $\mathcal A_{\ge d_1-1}$ when $d_1<d_2$ and $\varphi$ has a generalized zero in its first column. (The latter condition is equivalent to assuming that $\mathcal C$ has a singularity of multiplicity $d_2$.) In particular, we give the bi-degrees of a minimal bi-homogeneous generating set for this  ideal. When $2=d_1<d_2$  and $\varphi$ has a generalized zero in its first column, then we record explicit generators for $\mathcal A$. When $d_1=d_2$, we provide a translation between the bi-degrees of a bi-homogeneous minimal generating set for $\mathcal A_{d_1-2}$ and the number of singularities of multiplicity $d_1$ that are on or infinitely near $\mathcal C$. We conclude with a table that translates between the bi-degrees of a bi-homogeneous minimal generating set for $\mathcal A$ and the configuration of  singularities of $\mathcal C$ when the curve $\mathcal C$ has degree six.  
\end{abstract}

\maketitle

\section{Introduction.}\label{Intro}

\begin{center}
{\sc    Table of Contents }
\end{center}

\begin{enumerate}
\item[\ref{Intro}.] Introduction. 
\item[\ref{DPC}]Duality, perfect pairing, and consequences. 
\begin{enumerate}
\item[\ref{DPC}.A.] The abstract duality relating  $\mathcal A$ and $\Sym(I)$. 
\item[\ref{DPC}.B.]   The torsionfreeness and reflexivity of the $S$-module ${\rm Sym}(I)_i$ and how these properties are related to the geometry of the corresponding curve. 
\item[\ref{DPC}.C.]  The duality is given by multiplication.  
\item[\ref{DPC}.D.]   Explicit $S$-module generators for $\mathcal A_i$, when $i$ is large.
\end{enumerate}

\item[\ref{3}.]The case of a generalized zero in the first column of $\varphi$.
\item[\ref{MorleyForms}.]{Morley forms}. 
\item[\ref{2=d1}.]{Explicit generators for $\mathcal A$ when $d_1=2$}. 
\item[\ref{d1=d2}.]{The case of $d_1=d_2$}. 
\item[\ref{sextic}.]{An Application: Sextic curves}. 
\end{enumerate}

\bigskip

Our basic setting is as follows: Let $k$ be an algebraically
closed field, $R=k[x,y]$ a polynomial ring in two variables,
and $I$ an ideal of $R$ minimally generated by homogeneous forms $h_1, h_2, h_3$ of
the same degree $d > 0$. Extracting a common divisor we may
harmlessly assume that $I$ has height two. We will keep these
assumptions throughout the introduction, though many  of our results are stated and proved
 in greater generality.

On the one hand, the homogeneous forms $h_1, h_2, h_3$ define a morphism
\begin{equation}\label{gc} \xymatrix{
\eta: \ {\mathbb P}_k^1  \ar[rr]^{[h_1:h_2:h_3]} & & {\mathbb P}_k^2
} \end{equation}
whose image is a curve $\mathcal C$. After reparametrizing we may
assume that the map $\eta$ is birational onto its image or, equivalently,
that the curve $\mathcal C$ has degree $d$.

On the other hand, associated to $h_1,h_2,h_3$ is a syzygy matrix $\varphi$
that gives rise to a homogeneous free resolution of the ideal $I$,
$$ 0\longrightarrow R(-d-d_1) \oplus R(-d-d_2) \overset{\varphi}{\longrightarrow} R(-d)^3
\longrightarrow I \longrightarrow 0 \ . $$
Here $\varphi$ is a 3 by 2 matrix with homogeneous entries in $R$,
of degree $d_1$ in the first column and of degree $d_2$ in
the second column. We may assume that $d_1 \leq d_2$. Notice that
$d= d_1 + d_2$ by the Hilbert-Burch Theorem.

The two aspects, the curve $\mathcal C$ parametrized by the forms $h_1,h_2,h_3$
and the syzygy matrix $\varphi$ of these forms, are mediated by the Rees algebra
${\mathcal R}$ of $I$.
The Rees algebra is defined as the subalgebra $R[It]=R[h_1t,h_2t,h_3t]$ of the polynomial ring $R[t]$.
It becomes a standard bi-graded $k$-algebra if one sets ${\rm deg} \ x = {\rm deg} \ y =(1,0)$ and
${\rm deg} \ t =(-d,1)$, which gives ${\rm deg}\, h_it =(0,1)$. The
bi-homogeneous spectrum of $\mathcal R$ is the
graph $\Gamma \subset {\mathbb P}_k^1 \times {\mathbb P}_k^2$ of the
morphism $\eta=[h_1:h_2:h_3]$. Projecting to the second factor of ${\mathbb P}_k^1 \times {\mathbb P}_k^2$
one obtains a surjection $\Gamma \twoheadrightarrow \mathcal C$, which corresponds to
an inclusion of coordinate rings $${\mathcal R} \hookleftarrow k[h_1t,h_2t,h_3t] \ . $$
Thus the coordinate ring $A({\mathcal C})$ of the curve $\mathcal C$ can be recovered as
a direct summand of the Rees algebra ${\mathcal R}$, namely
$$ A(\mathcal C) = \bigoplus_i {\mathcal R}_{(0,i)} \ . $$
The same holds for the ideal $I$,
$$ I \simeq It = \bigoplus_i {\mathcal R}_{(i,1)} \ . $$

Finally, the inclusion $\Gamma \subset {\mathbb P}_k^1 \times {\mathbb P}_k^2$
corresponds to a homogeneous epimorphism $\mathcal R \twoheadleftarrow B$, where $B=k[x,y,T_1,T_2,T_3]$
is a bi-graded polynomial ring with ${\rm deg} \ x = {\rm deg} \ y =(1,0)$ and
${\rm deg} \ T_i =(0,1)$, and the variables $T_i$ are mapped to $h_it$. The kernel of this
epimorphism is a bi-homogeneous ideal $\mathcal J$ of $B$, the `defining ideal' of the Rees algebra $\mathcal R$.
Now the syzygy module of $I$ can be recovered as well, $$ {\rm syz}(I) \simeq \bigoplus_i {\mathcal J}_{(i,1)} \ . $$

Thus, the philosophy underlying this work can be summarized as follows: One wishes to study local properties
of the rational plane curve $\mathcal C$, such as the types of its singularities, by means of the syzygies
of $I$, since linear relations among polynomials are easier to handle than polynomial relations.
The mediator is the Rees algebra, which in turn carries more information than the coordinate ring $A(\mathcal C)$
of the curve,
just like the graph of a map reveals more than the image of the map. One may therefore hope that
even relatively simple numerical data associated to this algebra, such as the (first) bi-graded Betti numbers, say a
great deal about the curve.
The syzygies of $I$ appear in the defining
ideal $\mathcal J$, which leads one to study defining ideals of Rees algebras. Finding such ideals or, equivalently,
describing Rees rings explicitly in terms of generators and relations, is a fundamental problem in elimination
theory, that has occupied commutative algebraists, algebraic geometers, and, more recently, applied mathematicians.
The problem is wide open, even for ideals of polynomial rings in two variables. 


Write $\mathfrak m =(x,y)$ for the homogeneous maximal ideal of $R=k[x,y]$ and $S$ for the polynomial
ring $k[T_1,T_2,T_3]$. Recall that $B=k[x,y,T_1,T_2,T_3]=R \otimes _k S$ and that $\mathcal R =B/ \mathcal J$.
To study the Rees algebra of an ideal one customarily maps the symmetric algebra onto it,
$$
0 \longrightarrow \mathcal A \longrightarrow {\rm Sym}(I) \longrightarrow \mathcal R \longrightarrow 0
\ .
$$
One readily sees that $\mathcal A = {\mathrm H}^0_{\mathfrak m}({\rm Sym}(I)) = 0 :_{{\rm Sym}(I)} {\mathfrak m}^{\infty}$.
A presentation of the symmetric algebra is well understood, ${\rm Sym}(I) \simeq B/(g_1, g_2)$, where
$$
[g_1, g_2] = [T_1,T_2,T_3] \cdot \varphi  \ .
$$
The polynomials $g_i$ are homogeneous of bi-degree $(d_i,1)$ and together they form a $B$-regular sequence.
The symmetric algebra $\Sym(I)$, the Rees algebra $\mathcal R$, and the ideal $\mathcal A$ of $\Sym(I)$ that defines $\mathcal R$ all are  naturally equipped  with two gradings: the $T$-grading and the $xy$-grading. Both gradings play crucial roles in our work. The $T$-grading is often used in the study of symmetric algebras and Rees algebras. For example, an ideal is said to be of ``linear type'' if $\mathcal A=0$, which means that the defining ideal $\mathcal J$ of the Rees algebra is generated by polynomials of $T$-degree $1$. Ideals of linear type are much studied in the literature; see, for example \cite{H80, V, HSV83,H86,HRZ, L}. Much of our work is focused on the $xy$-grading. We view $\mathcal A$ as $\, \bigoplus\mathcal A_i$, where $\mathcal A_i$ is the $S$-submodule of $\mathcal A$ which consists of all elements homogeneous in $x$ and $y$ of degree $i$; in other words, $\mathcal A_i=\bigoplus_j\mathcal A_{(i,j)}$. One major advantage of this decomposition is the fact that $\mathcal A_i$ is non-zero for only finitely many values of $i$. Both gradings come into play  in the proof of Theorem \ref{claudia}, which is one of the main results of the paper. In Theorem \ref{claudia} we identify the degrees of the minimal generators of each $\mathcal A_i$ in the range $i \ge d_1-1$ (see also Table \ref{table}). In particular, for each fixed $i$ we must determine the minimum value of $j$ for which $\mathcal A_{i,j}$ is not zero. Curiously enough, the key point in our proof is that we 
revert to the $T$-grading for this. 

  The mathematics that sets the present project in motion is due to Jouanolou \cite{jo,jo96}; see also Bus\'e \cite{bu}. Jouanolou proved that the 
 multiplication map
\begin{equation}\label{J}\mathcal A_i\otimes \operatorname{Sym}(I)_{\delta-i}\longrightarrow \mathcal A_{\delta}\simeq S(-2)\end{equation} gives a perfect pairing of $S$-modules for $\delta=d-2$. 
 Jouanolou uses Morley forms to exhibit dual bases for the modules of (\ref{J}). The perfect pairing (\ref{J})   shows that the $S$-module structure of $\mathcal A_i$ is completely determined by the $S$-module structure of $\operatorname{Sym}(I)_{\delta-i}$.
The symmetric algebra $\operatorname{Sym}(I)$ is a complete intersection defined by the regular sequence $g_1,g_2$; so, the $S$-module structure of $\operatorname{Sym}(I)_{\delta-i}$ depends on the relationship between $\delta-i$, $d_1$, and $d_2$.

Ultimately we offer three proofs of Jouanolou's perfect pairing (\ref{J}). Two of our arguments are different from Jouanolou's; furthermore, our arguments are self-contained,  and we obtain results not obtained by Jouanolou.   In particular, we relate the entries of $\varphi$   to module-theoretic properties of $\mathcal A_i$ and $\Sym(I)_{\delta-i}$, (see especially Theorem \ref{future} and Corollary \ref{2.12} in Subsection \ref{DPC}.B) and also to information about the singularities of the curve parameterized by a minimal generating set for $I$; see especially Sections \ref{d1=d2} and \ref{sextic}. A very quick proof of the abstract duality relating $\mathcal A$ and $\Sym(I)$ is given in Subsection \ref{DPC}.A. This proof computes the local cohomology with support in $\mathfrak m$ along the Koszul complex which resolves $\Sym(I)$ as a $B$-module. In Subsection \ref{DPC}.C we take advantage of the module theoretic properties of the $\mathcal A_i$ (in particular the fact that they are reflexive $S$-modules as is shown in Subsection \ref{DPC}.B) to prove that the abstract isomorphism of Subsection \ref{DPC}.A is actually given by multiplication. Finally, in Theorem \ref{Morley}, as part of our review of the theory of Morley forms in Section \ref{MorleyForms}, we give Jouanolou's own proof of the perfect pairing (\ref{J}).

In Theorem \ref{claudia}
we describe the $S$-module structure of $\mathcal A_{\ge d_1-1}$ (according to the convention described above, $\mathcal A_{\ge d_1-1}$ means $\bigoplus_j A_{(\ge d_1-1,j)}$) under the hypothesis that $d_1<d_2$ and $\varphi$ has a generalized zero in its first column.  This module is free and we identify the bi-degrees of a bi-homogeneous basis for it;  see also Table \ref{table}. In Corollary \ref{B-mod} we identify the bi-degrees of a minimal  bi-homogeneous  generating set of $\mathcal A_{\ge d_1-1}$ as an ideal of $\operatorname{Sym}(I)$. 
When one views this result in the geometric context of (\ref{gc}),
then  the hypothesis concerning the existence of a generalized zero is equivalent to assuming that $\mathcal C$ has a singularity of multiplicity $d_2$, and the ideal $\mathcal A_{\ge d_1-1}$ of the conclusion is an approximation of the ideal that defines the graph $\Gamma\subset \mathbb P^1_k\times \mathbb P^2_k$ of the parameterization $\eta:\mathbb P^1_k\twoheadrightarrow \mathcal C$.  
The part of $\Sym(I)$ that corresponds to $\mathcal A_{\ge d_1-1}$, under the duality of  (\ref{J}), is $\Sym(I)_{\le d_2-1}$. There is no contribution from $g_2$ to the $S$-module $\Sym(I)_{\le d_2-1}$ in the bi-homogeneous  $B$-resolution of $\Sym(I)$. So, basically, we may ignore $g_2$. Furthermore, the hypothesis that the first column of $\varphi$ has a generalized zero allows us to make the critical calculation over a   subring $U$ of $S$, where $U$ is a polynomial ring in two variables. 

In Section \ref{2=d1} we focus on the situation $2= d_1<d_2$.  
Bus\'e \cite{bu} has given explicit formulas for the generators of 
 $\mathcal A$ if the first column of  $\varphi$ does not have a  generalized zero.    In Theorem \ref{goal5} we carry out the analogous project  in the case when
the first column of  $\varphi$ does have a generalized zero.   These hypotheses about generalized zeros in the first column have geometric implications for  the corresponding curve. In Bus\'e's case all of the singularities of $\mathcal C$ have multiplicity at most $d_1$; whereas, in the situation of Theorem \ref{goal5}, $\mathcal C$ has at least one  singularity of multiplicity $d_2$. The proof of Theorem \ref{goal5} is based on the results of Section \ref{3} (since $\mathcal A_{\ge d_1-1}$ is equal to $\mathcal A_{\ge 1}$ when $d_1=2$ and $\mathcal A_0$ is always well understood), an analysis of the kernel of a Toeplitz matrix of linear forms in two variables (see Lemmas \ref{J08} and \ref{J08'}), and  Jouanolou's theory of Morley forms. We review the theory of Morley forms in Section \ref{MorleyForms}.

In Section \ref{d1=d2} we completely describe the $S$-module structure of $\mathcal A_{d_1-2}$ when $d_1=d_2$. A preliminary version of this section initiated the investigation    that culminated in \cite{CKPU}. The geometric significance of these calculations are emphasized in \cite{CKPU} and are reprised in 
the present paper; however the main focus of Section \ref{d1=d2} is on the Rees algebras. 

The results in  Sections 1 -- 6 suffice to provide significant information about the  defining
equations for $\mathcal R$ if $d=d_1 +d_2 \leq 6$, since then $d_1 \leq 2$ (see, especially, Section \ref{2=d1}) or $d_1=d_2$ (see, especially, Section \ref{d1=d2}). Section \ref{sextic} is concerned with the   case $d=6$, the case of a sextic curve.
We show that 
there is, essentially, a one-to-one correspondence between the bi-degrees of the
defining equations of $\mathcal R$ on the one hand and the types of the singularities on or
infinitely near the curve $\mathcal C$ on the other hand.

If $R$ is a ring, we write $\operatorname{Quot}(R)$ for the {\it total ring of quotients} of $R$; that is, $\operatorname{Quot}(R)=U^{-1}R$, where $U$ is the set of non zerodivisors on $R$. 
If $R$ is a domain, the total ring of quotients of $R$ is usually called the {\it quotient field} of $R$. 

If $M$ is a matrix, then $M^{\rm T}$ denotes the transpose of $M$. If $M$ has entries in a $k$-algebra, where $k$ is a field, then
a {\it generalized zero} of $M$ is a 
  product $pMq^{\text{\rm T}}=0$, where $p$ and $q$ are non-zero row vectors with entries from $k$. 

If $M$ is a $\ell-1\times \ell$ matrix with entries in a ring, then the ring elements  $m_1,\dots,m_{\ell}$ are the {\it signed maximal minors} of the matrix $M$; that is, $m_i$ is $(-1)^{i+1}$ times the determinant of the submatrix of $M$  obtained by removing column $i$. 
We notice that the product $\, M\, [m_1,\dots,m_{\ell}]^{\rm T}$ is zero.

If $S$ is a ring and $A$, $B$, and $C$ are $S$-modules, then the $S$-module homomorphism $F: A\otimes_S B\to C$ is a {\it perfect pairing} if the induced $S$-module homomorphisms $A\to \mathrm{Hom}_S(B,C)$ and $B\to \mathrm{Hom}_S(A,C)$,  given by $a\mapsto F(a\otimes \underline{\phantom{X}})$ and $b\mapsto F(\underline{\phantom{X}}\otimes b)$, are isomorphisms.

\section{Duality, perfect pairing, and consequences.}\label{DPC}

\begin{data}\label{data1} Let $k$ be a field, $R=k[x,y]$  a polynomial ring in $2$ variables over $k$, $\mathfrak m=(x,y)R$  the homogeneous maximal ideal of $R$, and $I$  a height $2$ ideal of $R$ minimally generated by $3$ forms of the same positive degree $d$. Let  $\delta=d-2$. Let   $\varphi$ be a homogeneous Hilbert-Burch matrix for $I$; each entry in column $i$ of $\varphi$ has degree $d_i$ with $d_1\le d_2$. Let $\mathcal A$ be the kernel of the natural surjection
$${\rm Sym}(I)\twoheadrightarrow \mathcal R  $$ from the symmetric algebra of $I$ to the Rees algebra of $I$, and  let $S$ and $B$ be the polynomial rings $S=k[T_1,T_2,T_3]$ and $B=R\otimes_kS=k[x,y,T_1,T_2,T_3]$. View $B$ as a bi-graded $k$-algebra, where $x$ and $y$ have bi-degree $(1,0)$ and each $T_i$ has bi-degree $(0,1)$. A presentation of the symmetric algebra is given by ${\rm Sym}(I)
\simeq B/(g_1,g_2)$, where $$
[g_1, g_2] = [T_1,T_2,T_3] \cdot \varphi  \ .
$$
\end{data}

\begin{remarks}\label{R2.2} Adopt Data {\rm \ref{data1}}. The Hilbert-Burch Theorem guarantees that $d_1+d_2=d$.
One readily sees that the ${\rm Sym}(I)$-ideals $$\mathcal A\quad \text{and}\quad  {\rm H}^0_{\mathfrak m}({\rm Sym}(I)) = 0 :_{{\rm Sym}(I)} {\mathfrak m}^{\infty}$$ are equal.
The polynomial $g_m$ is homogeneous of bi-degree $(d_m,1)$.
 The polynomials $g_1,g_2$   form a regular sequence on $B$ because the dimension of ${\rm Sym}(I)$ is equal to $3$ by \cite{HR}. Thus, the Koszul complex provides a bi-homogeneous $B$-resolution of the symmetric algebra:
\begin{equation}\label{Kos}
K_{\bullet}(g_1,g_2;B) \longrightarrow {\rm Sym}(I) \to 0  \  .
\end{equation}
\end{remarks}

\begin{remark} When the bi-graded $B$-modules $\mathcal A=\bigoplus_{i,j}\mathcal A_{(i,j)}$
and ${\rm Sym}(I)=\bigoplus_{i,j}{\rm Sym}(I)_{(i,j)}$
are viewed as   $S$-modules, then we write $\mathcal A=\bigoplus_i \mathcal A_i$ and ${\rm Sym}(I)=\bigoplus_i {\rm Sym}(I)_i$ , where $\mathcal A_i$ represents the $S$-module $\mathcal A_i=\bigoplus_j \mathcal A_{(i,j)}$ and  ${\rm Sym}(I)_i$ represents the $S$-module ${\rm Sym}(I)_i=\bigoplus_j {\rm Sym}(I)_{(i,j)}$.
\end{remark}

The goal of this section is to prove that $\mathcal A_{\delta}$ is a free $S$-module generated by an explicit element ${\rm syl}\in {\rm Sym}(I)_{(\delta,2)}$ and that the multiplication map
\begin{equation}\label{sectgoal}\mathcal A_i\otimes {\rm Sym}(I)_{\delta-i}\longrightarrow \mathcal A_{\delta}=S\cdot\text{syl}\end{equation} gives a perfect pairing of $S$-modules. Both of these results are due to Jouanolou \cite{jo,jo96}; see also Bus\'e \cite{bu}. Our arguments are  different from Jouanolou's; furthermore,  our arguments are self-contained, 
and 
we obtain results not obtained by Jouanolou. In particular, we relate the entries of $\varphi$   to module-theoretic properties of $\mathcal A_i$ and $\Sym(I)_{\delta-i}$, and also to information about the singularities of the curve parameterized by a minimal generating set for $I$. 
The section consists of four subsections:

\noindent  

\begin{enumerate}
\item[\ref{DPC}.A] The abstract duality relating  $\mathcal A$ and $\Sym(I)$;

\item[\ref{DPC}.B]   The torsionfreeness and reflexivity of the $S$-module ${\rm Sym}(I)_i$ and how these properties are related to the geometry of the corresponding curve;

\item[\ref{DPC}.C]  The duality is given by multiplication; and

\item[\ref{DPC}.D]   Explicit $S$-module generators for $\mathcal A_i$, when $i$ is large.

\end{enumerate}

\bigskip\bigskip
\begin{center}
{\sc \ref{DPC}.A \quad The abstract duality relating  $\mathcal A$ and $\Sym(I)$.}
\end{center}

\bigskip
The goal of this subsection is to relate the $S$-modules $\mathcal A_i$ and $\Sym(I)_{\delta-i}$ and to express $\mathcal A_i$ as the kernel of a homomorphism of free $S$-modules. This goal is attained in Corollary \ref{cor2.2} and Theorem \ref{chart}. The first step toward (\ref{sectgoal}) is to establish an abstract isomorphism between $\mathcal A$ and a shift of $\underline{\rm Hom}_S({\rm Sym}(I),S)$, where
$\underline{\rm Hom}$ denotes the graded dual.
Toward that aim, one computes local cohomology with support in $\mathfrak m$ along the resolution (\ref{Kos}),
uses the symmetry of
the Koszul complex, and the isomorphism ${\rm H}^2_{\mathfrak m}(R) \simeq \underline{\rm Hom}_k(R,k)(2)$.

\begin{theorem}\label{iso} If Data {\rm \ref{data1}} is adopted, then there is an isomorphism of  bi-graded $B$-modules
$$\mathcal A \simeq \underline{\rm Hom}_S({\rm Sym}(I),S)(-\delta,-2).$$
\end{theorem}
\begin{proof}
We first establish two isomorphisms that are essential for our proof.
From the self-duality of the Koszul complex one obtains an isomorphism of complexes of bi-graded $B$-modules,
\begin{equation}\label{bla}\underline{\rm Hom}_S(K_{\bullet}(g_1,g_2;B),S)\simeq K_{\bullet}(g_1,g_2;\underline{\rm Hom}_S(B,S))[2](d,2).\end{equation}
The symbol [2] indicates homological degree shift. The internal bi-degree shift is written $(d,2)$.
We  also use the following isomorphisms of bi-graded $B$-modules,
\begin{equation*}\begin{array}{rcll}
{\rm H}^2_{\m}(B)&\simeq& {\rm H}^2_{\m}(R \otimes_R B) \\
&\simeq& {\rm H}^2_{\m}(R) \otimes_R B  &\mbox{since $B$ is $R$-flat}\\
&\simeq& {\rm H}^2_{\m}(R) \otimes_k S &\mbox{since $B\simeq R \otimes _k S$}\\
&\simeq& \underline{\rm Hom}_k(R,k)(2)\otimes_k S &\mbox{by Serre duality}\\
&\simeq& \underline{\rm Hom}_S(B,S)(2,0).\end{array}
\end{equation*}
We deduce that
\begin{equation}\label{blabla}
{\rm H}^2_{\m}(B)\simeq \underline{\rm Hom}_S(B,S)(2,0).\end{equation}

To prove the assertion of the theorem we decompose $K_{\bullet}=K_{\bullet}(g_1,g_2;B)$ into short exact sequences
$$ 0 \to K_2 \longrightarrow K_1 \longrightarrow \mathcal{J} \to 0 \qquad \text{and}  \qquad 0 \to \mathcal{J} \longrightarrow K_0 \longrightarrow {\rm Sym}(I) \to 0.$$
The second sequence gives \begin{equation}\label{partial0} 0={\rm H}^0_{\m}(K_0) \longrightarrow {\rm H}^0_{\mathfrak m}({\rm Sym}(I)) \longrightarrow {\rm H}^1_{\m}(\mathcal{J}) \longrightarrow {\rm H}^1_{\m}(K_0)=0,\end{equation}
where the first and last modules vanish because $\operatorname {grade}{\m} B > 1$.
The first sequence above yields
\begin{equation}\label{partial} 0={\rm H}^1_{\m}(K_1) \longrightarrow {\rm H}^1_{\mathfrak m}(\mathcal{J})
\longrightarrow {\rm H}^2_{\m}(K_2) \stackrel{\partial}{\longrightarrow} {\rm H}^2_{\m}(K_1).\end{equation}
Notice that $\partial$ is the second differential of the Koszul complex $K_{\bullet}(g_1,g_2;{\rm H}^2_{\m}(B))$ because the formation of local cohomology commutes with taking direct sums.
Thus from (\ref{partial0}) and (\ref{partial}) we obtain a bi-graded isomorphism
$$ {\rm H}^0_{\mathfrak m}({\rm Sym}(I))\simeq {\rm H}_2(K_{\bullet}(g_1,g_2;{\rm H}^2_{\m}(B)).$$

On the other hand,
\begin{equation*}\begin{array}{rcll}
{\rm H}_2(K_{\bullet}(g_1,g_2;{\rm H}^2_{\m}(B))&\simeq& {\rm H}_2(K_{\bullet}(g_1,g_2;\underline{\rm Hom}_S(B,S)(2,0)) &\mbox{by (\ref{blabla})}\\
&\simeq& {\rm H}_0(\underline{\rm Hom}_S(K_{\bullet}(g_1,g_2,B),S))(-d,-2)(2,0)  &\mbox{by (\ref{bla})}\\
&\simeq& \underline{\rm Hom}_S({\rm Sym}(I),S)(2-d,-2). \end{array}
\end{equation*}
The last isomorphism holds because $K_{\bullet}(g_1,g_2;B)$ is a resolution of ${\rm Sym}(I)$.
\QED
\end{proof}

\begin{corollary}\label{cor2.2}Adopt Data {\rm \ref{data1}}. The following statements hold.

\begin{itemize}

\item[{\rm(1)}] The graded $S$-modules $\mathcal A_i  $ and ${\rm Hom}_S({\rm Sym}(I)_{\delta-i},S(-2))$ are isomorphic for all $i$.
\item[{\rm(2)}]  The $S$-module $\mathcal A_i  $ is zero for all $i > \delta$.
\item[{\rm(3)}] The graded $S$-module $\mathcal A _{\delta}$ is isomorphic to $S(-2)$.
\item[{\rm(4)}]  The   $S$-module $\mathcal A_i  $ is reflexive for all $i$.

\end{itemize}
\end{corollary}

\begin{proof} Assertion (1) follows directly from Theorem~\ref{iso}; (2) follows from (1) since ${\rm Sym}(I)_{\ell}$ is zero when $\ell$ is negative; (3) holds because ${\rm Sym}(I)_{0}=S$; and (4) holds because the  $S$-dual of every finitely generated $S$-module is reflexive.
\QED
\end{proof}

\smallskip

Theorem \ref{iso} shows that the $S$-module structure of $\mathcal A_i$ is completely determined by the $S$-module structure of ${\rm Sym}(I)_{\delta-i}$.
The symmetric algebra ${\rm Sym}(I)$ is a complete intersection defined by the regular sequence $g_1,g_2$; so, the $S$-module structure of ${\rm Sym}(I)_{\delta-i}$ depends on the relationship between $\delta-i$, $d_1$, and $d_2$.
Theorem \ref{chart}
describes the $S$-module structure of  ${\rm Sym}(I)_{\delta-i}$  and $\mathcal A_i$ as a function of where $\delta-i$ sits with respect to $d_1\le d_2$. We set up the  relevant notation in the next definition.

\begin{definition}\label{upsilon} The polynomials $g_1$ and $g_2$ in $S[x,y]$ are defined in Data \ref{data1}. At this point we name their coefficients by writing
\begin{equation}\label{g}g_m=\sum_{\ell=0}^{d_m} c_{\ell,m}x^\ell y^{d_m-\ell},\end{equation} with $c_{\ell,m}\in S_1$, for $m$ equal to $1$ or $2$. For positive integers $n$ and $m$, with $m$ equal to $1$ or $2$, let $\Upsilon_{n,m}$ be the
 $(d_{m}+n)\times n$   matrix
$$\Upsilon_{n,m}=
\left[\begin{matrix}
c_{0,m}&0&0&0&0&0\\
c_{1,m}&c_{0,m}&0&0&0&0\\
\cdot&     c_{1,m} &\ddots&0&0&0  \\
\cdot&\cdot&\ddots&\ddots&0&0\\
\cdot&\cdot&\ddots&\ddots&\ddots&0\\
\cdot&\cdot&\ddots&\ddots&\ddots&c_{0,m}\\
c_{d_{m},m}&\cdot&\ddots&\ddots&\ddots&c_{1,m}\\
0&c_{d_{m},m}&\ddots&\ddots&\ddots&\cdot\\
0&0&c_{d_{m},m}&\ddots&\ddots&\cdot\\
0&0&0&\ddots&\ddots&\cdot\\
0&0&0&0&\ddots&\cdot\\
0&0&0&0&0&c_{d_{m},m}\end{matrix}\right],
 $$
with entries from $S_1$. The matrix $\Upsilon_{n,m}$ represents the map of free $S$-modules $S[x,y]_{n-1}\to S[x,y]_{n-1+d_m}$ which is given by multiplication by $g_m$ when the bases $y^{n-1},\dots,x^{n-1}$ and $y^{n-1+d_m},\dots,x^{n-1+d_m}$ are used for $S[x,y]_{n-1}$ and $S[x,y]_{n-1+d_m}$, respectively.
\end{definition}

\begin{theorem}\label{chart} Adopt Data \ref{data1}. The following statements hold.

\begin{itemize}
\item[{\rm(1)}] If $0\le i\le d_1-2$, then
the  $S$-modules $\mathcal A_i$ and ${\rm Sym}(I)_{\delta-i}$ both have rank $i+1${\rm;} furthermore, the following  sequences of $S$-modules are exact{\rm{:}}
$$0\to  \xymatrix{ S(-1)^{d_2-i-1}\oplus S(-1)^{d_1-i-1} \ar[rrrrr] ^{\phantom{XXXXXXXX}
\left(\begin{matrix} {\Upsilon_{d_2-i-1,1}}&
 {\Upsilon_{d_1-i-1,2}}\end{matrix}\right) }&&&&&S^{\delta-i+1}}\longrightarrow {\rm Sym}(I)_{\delta-i}\to 0$$and
$$0\to \mathcal A_i\to \xymatrix{S(-2)^{\delta-i+1} \ar[rrr] ^-{\left(\begin{matrix} {\Upsilon_{d_2-i-1,1}}^{\text{\rm T}}\\
\vspace{5pt}{\Upsilon_{d_1-i-1,2}}^{\text{\rm T}}\end{matrix}\right) }&&& S(-1)^{d_2-i-1}\oplus S(-1)^{d_1-i-1}}.$$

\item[{\rm(2)}] If $d_1-1\le i\le d_2-2$, then the $S$-modules $\mathcal A_i$ and ${\rm Sym}(I)_{\delta-i}$ both have rank $d_1${\rm;} furthermore,
the following  sequences of $S$-modules are exact{\rm{:}}
$$  0\to\xymatrix{S(-1)^{d_2-i-1} \ar[rr]
^{\phantom{XXX}\Upsilon_{d_2-i-1,1}\phantom{XXX}}&& S^{\delta-i+1}}\longrightarrow {\rm Sym}(I)_{\delta-i}\to 0\, .$$
and
$$0\to \mathcal A_i\to \xymatrix{S(-2)^{\delta-i+1} \ar[rr] ^{{\Upsilon_{d_2-i-1,1}}^{\text{\rm T}}} &&  S(-1)^{d_2-i-1}}.$$

\item[{\rm(3)}] If $d_2-1\le i\le \delta$, then the $S$-modules $\mathcal A_i$ and ${\rm Sym}(I)_{\delta-i}$ both have rank $\delta-i+1${
\rm;} furthermore,
$$ {\rm Sym}(I)_{\delta-i} \simeq S^{\delta-i+1}  \qquad \mbox{and} \qquad   \mathcal A_i \simeq S(-2)^{\delta-i+1}.$$
\end{itemize}
\end{theorem}

\bigskip

\begin{proof}
The homogeneous $B$-resolution  of   ${\rm Sym}(I)$
$$0\to B(-d_1-d_2,-2)\longrightarrow B(-d_1,-1)\oplus B(-d_2,-1)\longrightarrow B\longrightarrow {\rm Sym}(I)\to 0\, ,$$
which is given in (\ref{Kos}), may be decomposed into the graded strands recorded in the statement of the Theorem.
The rank of each $S$-module  ${\rm Sym}(I)_{\delta-i}$ can be  read immediately from its resolution. The statements about the modules $\mathcal A_i$ follow from part (1) of Corollary \ref{cor2.2}.
\QED
\end{proof}

\bigskip\bigskip
\begin{center}
{\sc \ref{DPC}.B \quad The torsionfreeness and reflexivity of the $S$-module ${\rm Sym}(I)_i$ and how these properties are related to the geometry of the corresponding curve.}
\end{center}

\bigskip
We are now going to investigate the torsionfreeness and reflexivity
of the graded components of ${\rm Sym}(I)$. To do so we need to
estimate the height of ideals of minors of the matrices that appear
in parts (1) and (2) of Theorem \ref{chart}.

\begin{lemma}\label{hts} Adopt Data {\rm \ref{data1}}. Let $n$ be a positive integer   and
$\Upsilon_{n,1}$ be the $(d_1 +n) \times n$ matrix introduced in
Definition \ref{upsilon}. The following statements hold$\, :$
\begin{itemize}
\item[{\rm(1)}]
$\operatorname{ht} I_{n}(\Upsilon_{n,1}) \ge 2 \,${\rm;}
\item[{\rm(2)}]$\operatorname{ht} I_{n}(\Upsilon_{n,1}) = 3$
if and only if the first column of $\varphi$ does not have a
generalized zero.

\end{itemize}
\end{lemma}

\begin{proof} The ideal $I_n(\Upsilon_{n,1})$ is equal to the $n^{\text{th}}$ power of the ideal $I_1(\Upsilon_{1,1})$; and, for any given ideal $J$ in the polynomial ring $S$, the ideals $J$ and $J^n$ have the same height. Therefore, it suffices to prove the result when $n=1$. On the other hand, the ideal  $I_1(\Upsilon_{1,1})$ is generated by linear forms in $S_1$; so the height of $I_{1}(\Upsilon_{1,1})$ is equal to the minimal number of generators of  $I_{1}(\Upsilon_{1,1})$.
Recall that $$[y^{d_1},xy^{d_1-1},\dots,x^{d_1}]\,\Upsilon_{1,1} =g_1=[T_1,T_2,T_3]\,\varphi_1,$$where $\varphi_1$ is the first column of $\varphi$. The entries of $\varphi_1$ generate an ideal of height at least $2$ because $\operatorname{ht} I_2(\varphi)=2$.
To complete the proof it suffices to show that
\begin{equation}\label{mu} \mu(I_1(\Upsilon_{1,1}))=\mu(I_1(\varphi_1)).\end{equation}
Indeed, suppose, for the time being, that (\ref{mu}) has been established. Then
$$2\le \operatorname{ht}(I_1(\varphi_1))\implies 2\le \mu (I_1(\varphi_1))=\mu(I_1(\Upsilon_{1,1}))=\operatorname{ht}(I_1(\Upsilon_{1,1}))=\operatorname{ht}(I_n(\Upsilon_{n,1})) $$and (1) holds. Also,
$$\operatorname{ht} I_{n}(\Upsilon_{n,1}) \le 2 \iff \mu(I_1(\Upsilon_{1,1}))\le 2\iff \mu (I_1(\varphi_1))\le 2\iff \text{$\varphi_1$ has a generalized zero.}$$

Now we prove (\ref{mu}). Suppose that $I_1(\Upsilon_{1,1})$ is minimally generated by $\lambda_1,\dots,\lambda_s$ in $S_1$. It follows that
$$\Upsilon_{1,1}= \lambda_1 \rho_1+\dots+\lambda_s\rho_s$$ for  column vectors $\rho_\ell$ in $\operatorname{Mat}_{(d_1+1)\times 1}(k)$. For $1\le \ell\le s$, let $\xi_\ell$ be the homogeneous form  $$\xi_{\ell}= [y^{d_1},xy^{d_1-1},\dots,x^{d_1}]\rho_{\ell}$$ in $R_{d_1}$ and let $Z_{\ell}$ be the column vector of three constants with $\lambda_{\ell}=[T_1,T_2,T_3]Z_{\ell}$. We have
$$[T_1,T_2,T_3]\varphi_1=g_1=[y^{d_1},xy^{d_1-1},\dots,x^{d_1}]\Upsilon_{1,1}
=[y^{d_1},xy^{d_1-1},\dots,x^{d_1}](\lambda_1 \rho_1+\dots+\lambda_s\rho_s)
$$$$= \xi_1\lambda_1+\dots+\xi_s\lambda_s= \xi_1[T_1,T_2,T_3]Z_{1}+\dots+\xi_s[T_1,T_2,T_3]Z_{s}=[T_1,T_2,T_3](\sum_{\ell=1}^s\xi_{\ell}Z_{\ell}).$$The entries of the $3\times 1$ vector $\varphi_1-\sum_{\ell}\xi_{\ell}Z_{\ell}$ are homogeneous forms of degree $d_1$ in $R$; hence, this vector cannot be in the kernel of $[T_1,T_2,T_3]$ unless it is already zero. Thus, $\varphi_1= \sum_{\ell}\xi_{\ell}Z_{\ell}$.
 Let $V$ be the subspace of $R_{d_1}$ which is spanned by the entries of $\varphi_1$. We have shown that $V$ is a subspace of the vector space spanned by $\xi_1,\dots,\xi_s$.
It follows that
$$\mu(I_1(\varphi_1))=\dim V\le s= \mu(I_1(\Upsilon_{1,1})).$$ One may read the calculation in the other direction to see that $\mu(I_1(\Upsilon_{1,1}))\le \mu(I_1(\varphi_1))$.
\QED \end{proof}

\bigskip

\begin{remark}\label{curve} Adopt Data \ref{data1} and assume $k$ is algebraically closed. The
signed maximal minors $h_1, h_2, h_3$ of $\varphi$
define a morphism

$$\xymatrix{\mathbb{P}^1_k \ar[rr]
^{\phantom{XXX}[h_1:
h_2: h_3]\phantom{XXX}}&&  \mathbb{P}_k^2}$$
whose image is a rational plane curve $\mathcal C$.  The degree of the curve $\mathcal C$ 
satisfies the equality $\deg \mathcal C= d/r$, where $r$
is the degree of the field extension
$[\operatorname{Quot}(k[R_d]):\operatorname{Quot} (k[I_d])]$.  In
particular, $r=1$ if and only if the parametrization is birational
onto its image. \end{remark}

As it turns out, the heights of various ideals of
minors of interest can be expressed in terms of the singularities of
the curve $\mathcal C$. 

\medskip
\begin{lemma}\label{ckpu}Adopt Data {\rm\ref{data1}} with $d_1=d_2$. Let $C$ be the $(d_1 +1) \times 2$ matrix $C=\left(\begin{matrix}
{\Upsilon_{1,1}}&
 {\Upsilon_{1,2}}\end{matrix}\right) $ for $\Upsilon_{1,1}$ and $\Upsilon_{1,2}$ as   introduced in Definition {\rm \ref{upsilon}}, and let $\mathcal C$ be the curve of Remark {\rm\ref{curve}}. The following statements hold{\rm:}
\begin{itemize}
\item[{\rm(1)}]
$\operatorname{ht} I_{2}(C) \ge 2$ if and only if $I_1(\varphi)$ is
not a complete intersection{\rm;} furthermore, if $k$ is algebraically closed, then the previous conditions 
  hold if and only if the curve $\mathcal C$ is singular;  
\item[{\rm(2)}]$\operatorname{ht} I_{2}(C) = 3$ if and only if $\varphi$ does not have a
generalized zero{\rm;}  furthermore, if $k$ is algebraically closed, then the  previous conditions 
  hold if
and only if the curve $\mathcal C$ is singular and its singularities
have multiplicity at most $(\deg \mathcal C)/2-1$.
\end{itemize}
\end{lemma}

\begin{proof} We may harmlessly assume that $k$ is algebraically
closed.
Let $r$ be the
degree of the field extension
$[\operatorname{Quot}(k[R_d]):\operatorname{Quot} (k[I_d])]$, as described in Remark \ref{curve}. Then there exists a regular sequence $u,v$ in $R_r$  so that
$I=I_2(\varphi')R$ for some $3 \times 2$ matrix $\varphi'$ whose
entries are homogeneous polynomials of degree $(\deg \mathcal
C)/2=d/2r$ in the variables $u,v$ (see \cite{KPU-B}). The signed
maximal minors of $\varphi'$ provide a birational parametrization of
the same curve $\mathcal C$. Let $R'=k[u,v]$ and $I'$ be the ideal $I_2(\varphi')$ of $R'$. Define elements
$g_1', g_2'$ in $R'[T_1,T_2,T_3]$ via the equation  $[g_1', g_2'] =
[T_1,T_2,T_3] \cdot \varphi' \ $. Use these data to obtain matrices
$\Upsilon'_{n,m}$ as in Definition \ref{upsilon}. Finally, let $C'$
be the matrix $\left(\begin{matrix} {\Upsilon'_{1,1}}&
 {\Upsilon'_{1,2}}\end{matrix}\right)$. From \cite[3.14(2)]{CKPU} we
know that $\operatorname{ht} I_2(C') \ge 2$ if and only if the curve
$\mathcal C$ is singular, and $\operatorname{ht} I_2(C') =3$ if and
only if the curve $\mathcal C$ is singular and its singularities
have multiplicity at most $(\deg \mathcal C)/2-1$. (The result from \cite{CKPU} is
stated assuming the birationality of the parametrization. A complete proof of the geometric interpretation of $\operatorname{ht} I_2(C')\ge 2$ uses the fact that a rational plane curve of degree at least three is singular.) 

The curve $\mathcal C$ is nonsingular if and only if its homogeneous
coordinate ring $k[I'_{d/r}]$ is normal. Since the parametrization is
birational, the latter obtains if and only if
$k[I'_{d/r}]=k[R'_{d/r}]$ or, equivalently, $3=\dim _k I'_{d/r}=\dim
_k R'_{d/r}$. This holds if and only if $d/r=2$. The last equality
means that $I_1(\varphi')$ is generated by linear forms,
equivalently $I_1(\varphi')=(u,v)R'$. The latter holds if and only
if $I_1(\varphi')$ is a complete intersection, again because the
parametrization is birational. Finally, the $R'$-ideal
$I_1(\varphi')$ is a complete intersection if and only if the
$R$-ideal $I_1(\varphi)$ is.

On the other hand, the curve $\mathcal C$ is singular and its
singularities have multiplicity at most ${(\deg \mathcal C)/2-1}$ if
and only if $\varphi'$ does not have a generalized zero, as was
shown in part (4) of \cite[1.9]{CKPU}. Notice that  $\varphi'$ has a
generalized zero if and only if $\varphi$ does.

It remains to show that $I_2(C')=I_2(C)$. Extend the ordered set
$v^{d/r}, \ldots, u^{d/r}$ of monomials in $u, v$ of degree $d/r$ to
an ordered basis of $R_d$, which we call $b_0, \ldots, b_d$. Define
a $d+1 \times 2$ matrix $D$ with entries in $S_1$ via the equality
$[g_1',g_2']=[b_0, \ldots, b_d] \cdot D$. Notice that $D=
\left[\begin{matrix} C' \\ 0
\end{matrix} \right]$, and hence $I_2(C')=I_2(D)$. Finally, the
matrix $D$ is obtained from $C$ by elementary row operations that
correspond to the transition from $b_0, \ldots, b_{d}$ to the
monomial basis $y^d, \ldots, x^d$ of $R_d$. Therefore
$I_2(D)=I_2(C)$. \QED
\end{proof}

\medskip

\begin{theorem}\label{future}  Adopt Data {\rm \ref{data1}} and let $\mathcal C$ be the curve of Remark {\rm\ref{curve}}. The following statements hold.

\begin{itemize}
\item[{\rm(1)}] If $d_1\le i\le d_2-1$, then
\begin{itemize}
\item[{\rm(a)}]
the $S$-module ${\rm Sym}(I)_{i}$ is torsionfree\,{\rm;} and
\item[{\rm(b)}] the $S$-module ${\rm Sym}(I)_{i}$ is reflexive if
and only if the first column of $\varphi$ does not have a
generalized zero{\rm;} furthermore, if $k$ is algebraically closed, then the previous conditions hold if
and only if the singularities of the curve $\mathcal C$ have
multiplicity at most $d_1(\deg \mathcal C)/d$.
\end{itemize}

\item[{\rm(2)}] If $i=d_1=d_2$, then
\begin{itemize}
\item[{\rm(a)}] the $S$-module ${\rm Sym}(I)_{i}$ is torsionfree if and only if $I_1(\varphi)$ is
not a complete intersection{\rm;} furthermore, if $k$ is algebraically closed, then the previous conditions hold if and only if the curve $\mathcal C$ is singular\,{\rm;} and
\item[{\rm(b)}] the $S$-module ${\rm Sym}(I)_{i}$ is reflexive if
and only if  $\varphi$ does not have a generalized zero{\rm;} furthermore, if $k$ is algebraically closed, then the previous conditions hold if and only if the curve
$\mathcal C$ is singular and its singularities have multiplicity at
most $(\deg \mathcal C)/2-1$.
\end{itemize}
\end{itemize}
\end{theorem}

\begin{proof} We first argue the third equivalence in item
$($1.b$)$. Again, as in the proof of Lemma \ref{ckpu}, one obtains a
regular sequence $u,v$ of forms of degree $d/(\deg \mathcal C)$ so
that $I=I_2(\varphi')R$ for some $3 \times 2$ matrix $\varphi'$
whose entries in position $i,j$ are homogeneous polynomials of
degree $d_j(\deg \mathcal C)/d$ in the variables $u,v$ (see
\cite{KPU-B}). Thus one reduces to the case of a birational
parametrization. Now part (4) of \cite[1.9]{CKPU} shows that $\mathcal C$ has
a singularity of multiplicity at least $d_1(\deg \mathcal C)/d+1$ if
and only if the first column of $\varphi'$ has a generalized zero,
or, equivalently, the first column of $\varphi$ has a generalized
zero.

Write $n=i-d_1+1$. From Theorem \ref{chart} we know that the
$S$-module ${\rm Sym}(I)_{i}$ has projective dimension at most one
and that it is presented by $\Upsilon_{n,1}$ in the setting of (1)
and by $C=\left(\begin{matrix} {\Upsilon_{1,1}}&
 {\Upsilon_{1,2}}\end{matrix}\right) $ in the setting of (2).
 Thus ${\rm
Sym}(I)_{i}$ is torsionfree if and only if $I_{n}(\Upsilon_{n,1})$
or $I_2(C)$, respectively, has height at least two. Likewise, ${\rm
Sym}(I)_{i}$ is reflexive if and only if this height is at least 3.
Now it remains to appeal to Lemmas \ref{hts} and  \ref{ckpu}. \QED
\end{proof}

\begin{corollary}\label{2.12} Adopt Data {\rm \ref{data1}} and let $\mathcal C$ be the curve of Remark {\rm\ref{curve}}.
The following statements hold.

\begin{itemize}
\item[{\rm(1)}] If $d_1-1\le i\le d_2-2$, then the
$S$-module $\mathcal A_{i}$ is free if and only if the first column
of the matrix $\varphi$ has a generalized zero{\rm;} furthermore, if $k$ is algebraically closed, then the previous conditions hold if  and only if the curve $\mathcal C$ has a singularity of multiplicity equal to $d_2(\deg \mathcal
C)/d$.
\item[{\rm(2)}] If $i=d_1=d_2$, then
the $S$-module $\mathcal A_{i}$ is free if and only if $\mu(I_2(C))\le
4${\rm;} furthermore, if $k$ is algebraically closed, then the previous conditions hold if  and only if
there are at least two singularities of multiplicity $(\deg \mathcal
C)/2$ on or infinitely near $\mathcal C$.
\end{itemize}
\end{corollary}

\begin{proof} We prove part (1).
If the first column of the matrix $\varphi$ does not have  a
generalized zero, then ${\rm Sym}(I)_{\delta -i}$ is reflexive
according to Theorem \ref{future} part (1.b). Thus, part (1) of Corollary \ref{cor2.2}
shows that ${\rm Sym}(I)_{\delta -i} \simeq {\rm Hom}_S(\mathcal
A_i,S(-2))$. Since ${\rm Sym}(I)_{\delta -i} $ is not free it
follows that $\mathcal A_i$ cannot be free either. Conversely, if
the first column of the matrix $\varphi$ has a generalized zero then
$\mathcal A_i$ is free as will be shown in Theorem \ref{claudia}.
The third equivalence in item (1) is  parts (1), (2), and (4) of  \cite[1.9]{CKPU}, after reducing to the case of a birational parameterization.

The first equivalence of part (2) will be proved in Theorem
\ref{andy}. The second equivalence follows from \cite[3.22]{CKPU}
after reducing to the case of a birational parametrization. \QED
\end{proof}

\bigskip
\begin{center}
{\sc \ref{DPC}.C \quad The duality is given by multiplication.}
\end{center}

In (\ref{sectgoal}), we promised an explicit perfect pairing $\mathcal A_i\otimes_S \Sym(I)_{\delta-i}\lto A_{\delta}=S\cdot{\rm syl}$. So far, in part (1) of Corollary \ref{cor2.2}, we showed that $\mathcal A_i$ is isomorphic to ${\rm Hom}_S({\rm Sym}(I)_{\delta-i},S(-2))$. In Theorem \ref{pp} we prove that the abstract isomorphism of Corollary \ref{cor2.2} is given by multiplication. The other highlight of the present subsection is Corollary \ref{link}, where we prove that the ${\rm Sym}(I)$-ideals $\mathcal A _{\ge i}$
and $0 :_{{\rm Sym}(I)} {\mathfrak m}^{d-1-i}$ are equal.  We use this equality in subsection 2.D to record explicit generators for the $S$-modules $\mathcal A_i$ when the equality represents linkage. The Sylvester element ${\rm syl}$ is one of these explicit generators. It is  introduced in Remark \ref{???} of subsection 2.D.

\begin{theorem}\label{pp} Adopt Data {\rm \ref{data1}}.
For each $i$, the multiplication map $\mathcal A_i\otimes {\rm Sym}(I)_{\delta-i}\longrightarrow  \mathcal A _{\delta}$ induces a homogeneous isomorphism of $S$-modules
$$\mathcal A_i \longrightarrow {\rm Hom}_S({\rm Sym}(I)_{\delta-i},\mathcal A _{\delta}).$$
\end{theorem}
\begin{proof} If $i<0$ or $\delta<i$, then the assertion is trivial because ${\rm Hom}_S({\rm Sym}(I)_{\delta-i},\mathcal A _{\delta})$ and $\mathcal A_i$ both vanish due to (1) and (2) from Corollary \ref{cor2.2}. Hence, it suffices to prove the assertion for $i$ in the range $0 \le i \le \delta$. We fix such an $i$ and we denote the map induced by multiplication by $$\Phi:  \mathcal A_i \longrightarrow {\rm Hom}_S({\rm Sym}(I)_{\delta-i},\mathcal A _{\delta}).$$

Write $\Sigma$ for ${\rm Sym}(I)$. If $\p\in {\rm Spec}(S)$, then  the ring $\Sigma_{\p}=S_{\p} \otimes_S \Sigma$ is a standard graded $S_{\p}$-algebra with irrelevant ideal $(\Sigma_{\p})_+=\m\Sigma_{\p}$. Furthermore, $\Sigma_{\p}$ is a complete intersection with $\operatorname{dim} \Sigma_{\p}=\dim S_{\p}$. The source and the target of the homomorphism $\Phi$ are reflexive $S$-modules, see part (4) of Corollary \ref{cor2.2}. Thus, to prove that $\Phi$ is injective it suffices to show that $\Phi_{\p}=S_{\p} \otimes \Phi$ is injective when $\p$ is the zero ideal of $S$, and then to prove that $\Phi$ is surjective one only needs to check that $\Phi_{\p}$ is surjective for every $\p \in \operatorname{Spec}(S)$ with $\dim S_{\p}=1$.

First, let $\p$ be the zero ideal. In this case, $\Sigma_{\p}$ is an Artinian standard graded Gorenstein algebra over a field with homogeneous maximal ideal $\m\Sigma_{\p}$. Therefore ${\mathcal A}_{\p}= 0 :_{\Sigma_{\p}} {\mathfrak m}^{\infty}=\Sigma_{\p}$. In particular, $[\Sigma_{\p}]_{\delta} \not= 0$ and  $[\Sigma_{\p}]_{i}= 0$ for $i > \delta$  by Corollary \ref{cor2.2}, parts (1) and (2). In other words, $[\Sigma_{\p}]_{\delta}$ is the socle  of the Gorenstein algebra $\Sigma_{\p}$. Thus, multiplication induces an isomorphism $$[\Sigma_{\p}]_{i} \longrightarrow
{\rm Hom}_{S_{\p}}([\Sigma_{\p}]_{\delta-i},[\Sigma_{\p}]_{\delta}).$$ As $\Sigma_{\p}={\mathcal A}_{\p}$, we conclude that $\Phi_{\p}$ is an isomorphism.

Next, let $\p \in \operatorname{Spec}(S)$ with $\dim S_{\p}=1$. We need to show that $\Phi_{\p}$ is surjective. Let $\theta$ be any element of ${\rm Hom}_{S_{\p}}([\Sigma_{\p}]_{\delta-i},[{\mathcal A}_{\p}]_{\delta})$. We prove that the map $\theta$ is multiplication by some element of $[{\mathcal A}_{\p}]_i$. Notice that
\begin{equation*}\begin{array}{rcll}
{\rm Hom}_{S_{\p}}([\Sigma_{\p}]_{\delta-i},[{\mathcal A}_{\p}]_{\delta})&=& {\rm Hom}_{S_{\p}}((\m^{\delta-i}\Sigma_{\p})/(\m^{\delta-i+1}\Sigma_{\p}),[{\mathcal A}_{\p}]_{\delta})\\
&\subset& {\rm Hom}_{\Sigma_{\p}}((\m^{\delta-i}\Sigma_{\p})/(\m^{\delta-i+1}\Sigma_{\p}),{\mathcal A}_{\p})\\
&\subset& {\rm Hom}_{\Sigma_{\p}}((\m^{\delta-i}\Sigma_{\p})/(\m^{\delta-i+1}\Sigma_{\p}),\Sigma_{\p}),
\end{array}
\end{equation*}
 where the next-to-last inclusion holds because $\m[{\mathcal A}_{\p}]_{\delta}=0$ by Corollary \ref{cor2.2}, part (2).

We will prove that $\theta \in [{\rm Hom}_{\Sigma_{\p}}((\m^{\delta-i}\Sigma_{\p})/(\m^{\delta-i+1}\Sigma_{\p}),\Sigma_{\p})]_i$ can be lifted to a map $$\widetilde{\theta} \in \left[{\rm Hom}_{\Sigma_{\p}}(\Sigma_{\p}/(\m^{\delta-i+1}\Sigma_{\p}),\Sigma_{\p})\right]_{i}.$$ Any such $\widetilde{\theta} $ is induced by multiplication by an element $\lambda \in [\Sigma_{\p}]_i$. The element $\lambda$ is necessarily annihilated by $\m^{\delta-i+1}$. Recall that ${\mathcal A}_{\p}= 0 :_{\Sigma_{\p}} {\mathfrak m}^{\infty}$. Thus $\lambda$ lies in  $[{\mathcal A}_{\p}]_i$, and therefore $\widetilde{\theta} $ and $\theta$  both are induced by multiplication by an element $\lambda \in [{\mathcal A}_{\p}]_i$.

To show that $\theta$ can be lifted, we first apply ${\rm Hom}_{\Sigma_{\p}}(-, \Sigma_{\p})$ to the short exact sequence of graded $\Sigma_{\p}$-modules,
$$ 0 \to (\m^{\delta-i}\Sigma_{\p})/(\m^{\delta-i+1}\Sigma_{\p}) \longrightarrow \Sigma_{\p}/(\m^{\delta-i+1}\Sigma_{\p}) \longrightarrow \Sigma_{\p}/(\m^{\delta-i}\Sigma_{\p}) \to 0.
$$
The corresponding long exact sequence of cohomology induces the following exact sequence of $S_{\p}$-modules
\begin{equation*}\begin{array}{rcll} \left[{\rm Hom}_{\Sigma_{\p}}(\Sigma_{\p}/(\m^{\delta-i+1}\Sigma_{\p}),\Sigma_{\p})\right]_i &\longrightarrow& \left[{\rm Hom}_{\Sigma_{\p}}((\m^{\delta-i}\Sigma_{\p})/(\m^{\delta-i+1}\Sigma_{\p}),\Sigma_{\p})\right]_i \\
& \longrightarrow&  \left[{\rm Ext}^1_{\Sigma_{\p}}(\Sigma_{\p}/(\m^{\delta-i}\Sigma_{\p}),\Sigma_{\p})\right]_i.
\end{array}
\end{equation*}
It suffices to prove that $ [{\rm Ext}^1_{\Sigma_{\p}}(\Sigma_{\p}/(\m^{\delta-i}\Sigma_{\p}),\Sigma_{\p})]_i=0$. To do so, recall that $\Sigma_{\p}=S_{\p}[x,y]/(g_1,g_2)$, where $g_1, g_2$ is a regular sequence of forms of degrees $d_1, d_2$. Therefore, $$\omega_{\Sigma_{\p}} \simeq \Sigma_{\p}(-2+d_1+d_2)=\Sigma_{\p}(\delta).$$ It follows that $$ \left[{\rm Ext}^1_{\Sigma_{\p}}(\Sigma_{\p}/(\m^{\delta-i}\Sigma_{\p}),\Sigma_{\p})\right]_i \simeq \left[{\rm Ext}^1_{\Sigma_{\p}}(\Sigma_{\p}/(\m^{\delta-i}\Sigma_{\p}),\omega_{\Sigma_{\p}})\right]_{i-\delta}.$$
As $\dim \Sigma_{\p}=1$, local duality implies that the latter module vanishes if and only if  $$\left[{\rm H}^0_{\mathfrak M}(\Sigma_{\p}/(\m^{\delta-i}\Sigma_{\p}))\right]_{\delta-i} =0\, ,$$ where ${\mathfrak M}$ denotes the homogeneous maximal ideal of $\Sigma_{\p}$. To finish, we note that $$\left[{\rm H}^0_{\mathfrak M}(\Sigma_{\p}/(\m^{\delta-i}\Sigma_{\p}))\right]_{\delta-i} \subset \left[\Sigma_{\p}/(\m^{\delta-i}\Sigma_{\p})\right]_{\delta-i}=0\, .$$
\QED
\end{proof}

\begin{corollary}\label{link} Adopt Data {\rm \ref{data1}}.
For each integer $i$, the ${\rm Sym}(I)$-ideals $\mathcal A _{\ge i}$
and $0 :_{{\rm Sym}(I)} {\mathfrak m}^{d-1-i}$ are equal.
\end{corollary}

\begin{proof} Assume first that  $d-1\le i$. In this case,   $\mathcal A_{\ge i}=0$ by Corollary \ref{cor2.2}, part (2). On the other hand, in this case,   $d-1-i\le 0$\,; so, $\mathfrak m^{d-1-i}=R$ and $0 :_{{\rm Sym}(I)} {\mathfrak m}^{d-1-i}=0 :_{{\rm Sym}(I)} R=0$.

Now assume that $i\le 0$. In this case, $\mathcal A_{\ge i}=\mathcal A$. On the other hand, in this case, $d-1\le d-1-i$\,; thus,  $\mathcal A\,\mathfrak m^{d-1-i}\subset \mathcal A_{\ge d-1-i}\subset \mathcal A_{\ge d-1}=0$ and
$$\mathcal A \subset 0 :_{{\rm Sym}(I)} {\mathfrak m}^{d-1-i}\subset   0 :_{{\rm Sym}(I)} {\mathfrak m}^{\infty}=\mathcal A\, .$$

Finally,  assume that $1\le i\le d-2$. In this case,
  $\m^{d-1-i}\mathcal A_{\ge i} \subset A_{\ge d-1} = A_{\ge \delta +1} =0$, where the last equality holds by part (2) of Corollary \ref{cor2.2}. It follows that $\mathcal A_{\ge i} \subset 0 :_{{\rm Sym}(I)} {\mathfrak m}^{d-1-i}$. To see the other inclusion, let $\lambda \neq 0$ be a homogeneous element in  $ 0 :_{{\rm Sym}(I)} {\mathfrak m}^{d-1-i}$. Clearly $\lambda \in  0 :_{{\rm Sym}(I)} {\mathfrak m}^{\infty} =\mathcal A$. We prove that $\lambda$ has degree at least $i$. Suppose otherwise. In this case, we observe that
$$\lambda [{\rm Sym}(I)]_{d-2-\deg \lambda} \subset \lambda \m^{d-2-\deg \lambda} {\rm Sym}(I) \subset \lambda \m^{d-1-i} {\rm Sym}(I)=0\, ,$$ where the last inclusion holds because $d-2-\deg \lambda \ge d-1-i$.
This shows that $\lambda$ is a non-zero homogeneous element with the property that
  multiplication by $\lambda  $ induces the zero homomorphism from $[{\rm Sym}(I)]_{d-2-\deg \lambda}$ to $\mathcal A_{d-2}$. This contradicts the injectivity of the isomorphism established in Theorem \ref{pp}.
\QED
\end{proof}

\bigskip
\begin{center}
{\sc \ref{DPC}.D \quad Explicit $S$-module generators for $\mathcal A_i$, when $i$ is large.}
\end{center}

When $i$ is chosen so that $g_1$ and $g_2$ lie in $\mathfrak m^{d-1-i}B$, then the equality of Corollary~\ref{link} has an interpretation in terms of linkage.  In the present subsection we exploit that interpretation in order to exhibit explicit generators. Hong, Simis, and Vasconcelos \cite[Sect.~3]{HSV12} have used linkage in a similar manner. 

\begin{definition}\label{matrices} For each positive index $\ell$,  define  $\Lambda_\ell$ to be the
  $(\ell+1) \times \ell$  matrix
$$\Lambda_\ell=\left[\begin{matrix}
-x&0 &0&\cdots&0&0\\
y&-x&0&\cdots&0&0\\
0& y&.&\cdots&0&0\\
0&0&.&.&0&0\\
\vdots&\vdots&\vdots&.&.&\vdots\\
0&0&0&\cdots&.&-x\\
0&0&0&\cdots&0&y\end{matrix}\right].$$
Notice that  $\Lambda_\ell$ is a Hilbert-Burch matrix for the row vector $\left [\begin{matrix}y^\ell& xy^{\ell-1}&\cdots&x^\ell\end{matrix}\right]$. \end{definition}

\bigskip
If $\ell$ is an index with $1\le \ell\le d_1$, then the polynomials $g_1,g_2$ of Data \ref{data1} are both in $(x,y)^{\ell}B$. Let $\Xi_\ell$ be an $(\ell+1) \times 2$ matrix of bi-homogeneous elements of $B$, of bi-degree $(d_1-\ell,1)$ in column $1$ and bi-degree $(d_2-\ell,1)$ in column $2$, with
\begin{equation}\label{Xi}\left [\begin{matrix}g_1&g_2\end{matrix}\right] =\left [\begin{matrix}y^\ell& xy^{\ell-1}&\cdots&x^\ell\end{matrix}\right]
\Xi_\ell\end{equation} and
 $\Psi_\ell$  be the $(\ell+1) \times (\ell+2)$ matrix
\begin{equation}\label{Psi}\Psi_\ell=\left [\begin{matrix} \Lambda_\ell &\Xi_\ell\end{matrix}\right]
\end{equation}of bi-homogeneous elements of $B$.
\bigskip

\begin{corollary} \label{XXX} Adopt Data {\rm \ref{data1}}. Let $i$ be an index with
$d_2-1 \le i \le \delta$ and $\Psi_{d-1-i}$ be a matrix as described in {\rm(\ref{Psi})}.
\begin{itemize}
\item[$($1$)$] The ideals
$$\mathcal A_{\ge i}, \quad  \mathcal A_{(i,2)}  {\rm Sym}(I),\quad\text{and}\quad {\rm I}_{d-i}(\Psi_{d-1-i}) {\rm Sym}(I)
$$of   ${\rm Sym}(I)$ are equal.
In particular, the minimal bi-homogeneous generators of the ${\rm Sym}(I)$-ideal $\mathcal A$ all have $\{x,y\}$-degree at most $d_2-1$.
\item[$($2$)$] The $S$-module $\mathcal A_i$ is free of rank $d-i-1${\rm;} a basis consists of the maximal minors of the matrix $\Psi_{d-1-i}$ that involve the last two columns.
\end{itemize}
\end{corollary}
\begin{note} The bound on the generator degrees that is given in item (1)  is also given in \cite[Thm.~4.6]{BD}.\end{note}

\begin{proof}
Recall that $\mathcal A_{\ge i} =0 :_{{\rm Sym}(I)} {\mathfrak m}^{d-1-i}$ by Corollary \ref{link}.
The latter ideal equals $$\frac{(g_1, g_2)B :_B \m^{d-1-i}}{(g_1, g_2)B}.$$
The homogeneous $B$-ideal $\m^{d-1-i}B$ is perfect of grade 2 and it contains the homogeneous $B$-regular sequence $g_1, g_2$
because the hypothesis $d_2-1 \le i \le \delta$ guarantees that \begin{equation}\label{bigdeal}1\le d-1-i \le d_1.\end{equation} 
 The matrix $\Lambda_{d-1-i}$ of Definition \ref{matrices} is a homogeneous matrix of relations among the generators
$y^{d-1-i}, xy^{d-2-i}, \cdots, x^{d-1-i}$ of $\m^{d-1-i}B$, and $\Xi_{d-1-i}$ is a homogeneous matrix of coefficients when writing $g_1, g_2$ in terms of these generators. The inequalities of (\ref{bigdeal})  guarantee that  the matrix $\Xi_{d-1-i}$ of (\ref{Xi}) is defined.  Now, according to \cite{Gaeta}, the linked ideal $(g_1, g_2)B :_B \m^{d-1-i}$ is generated by the maximal minors of the $d-i$ by $d-i+1$ matrix $\Psi_{d-1-i}=\left [\begin{matrix} \Lambda_{d-1-i} &\Xi_{d-1-i}\end{matrix}\right]$. Thus indeed,
$$\frac{(g_1, g_2)B :_B \m^{d-1-i}}{(g_1, g_2)B}= \frac{{\rm I}_{d-i}(\Psi_{d-1-i}) B}{(g_1, g_2)B}={\rm I}_{d-i}(\Psi_{d-1-i}) {\rm Sym}(I).$$
The maximal minors of $\Psi_{d-1-i}$ are $g_1, g_2$ (up to sign),
together with the minors $\Delta_1, \ldots, \Delta_{d-1-i}$ that involve the last two columns. We conclude that
$$ \mathcal A_{\ge i} = (\Delta_1, \ldots, \Delta_{d-1-i}){\rm Sym}(I)$$ Observe that each $\Delta_i$ has  bi-degree $(i,2)$. This completes the proof of (1).

Assertion (2) follows from Theorem \ref{chart}, part (3), and the fact that the elements
 $\Delta_1, \ldots, \Delta_{d-1-i}$ generate $\mathcal A_i$ as an $S$-module.
\QED
\end{proof}

\bigskip
\begin{remark}\label{???} Corollary \ref{XXX} describes $\mathcal A_{i_0}$ for each $i_0$ with $d_2-1\le i_0\le \delta$. We highlight the content of Corollary \ref{XXX} at the boundaries  $i_0=d_2-1$ and $i_0=\delta$.

Take $i_0=\delta$. In this situation, the $S$-module $\mathcal A_{\delta}$ is free of rank $1$ with basis element any Sylvester form ``${\rm syl}$'', where ${\rm syl}$ is the image in ${\rm Sym}(I)$ of the determinant of any fixed $2\times 2$ matrix $\Xi_1$ described in   (\ref{Xi}). (The $k$-vector space $\mathcal A_{(\delta,2)}$ is one-dimensional and we use the name ``${\rm syl}$'' for any basis element of this vector space.) Notice that the entries of column $m$ of $\Xi_1$ are homogeneous of bi-degree $(d_m-1,1)$ and $\syl$ is homogeneous of bi-degree $(\delta,2)$. Let $i$ be arbitrary. One explicit realization of the isomorphism
$$\mathcal A_i  \simeq {\rm Hom}_S({\rm Sym}(I)_{\delta-i},S(-2))$$ of Corollary \ref{cor2.2} is that this isomorphism is induced by the composition
\begin{equation}\label{sigma}\xymatrix{\mathcal A_i\otimes {\rm Sym}(I)_{\delta-i}\ar[rr]^{\phantom{XXXX}\rm mult}&& \mathcal A_{\delta}}\stackrel{\sigma}{\longrightarrow} S(-2),\end{equation}where
${\rm mult}$ is the multiplication in ${\rm Sym}(I)$, as described in Theorem \ref{pp}, and $\sigma$ is the inverse of the isomorphism $S(-2)\to \mathcal A_{\delta}$ which sends $1$ to $\syl$.

Take  $i_0=d_2-1$. Corollary \ref{XXX} guarantees that $\mathcal A_{\ge d_2-1}$ is generated as a $B$-module by the maximal minors of $\Psi_{d_1}=[\Lambda_{d_1},\ \Xi_{d_1}]$ which involve both columns of $\Xi_{d_1}$. Each of these minors is homogeneous of bi-degree $(d_2-1,2)$, and one choice for
$\Xi_{d_1}$ is $$\left[\begin{array}{ll}c_{0,1}\quad&c_{0,2}y^{d_2-d_1}+\dots+c_{d_2-d_1,2}x^{d_2-d_1}\\c_{1,1}&c_{d_2-d_1+1,2}x^{d_2-d_1}\\\vdots&\vdots\\c_{d_1,1}&c_{d_2,2}x^{d_2-d_1} \end{array}\right],$$

\noindent for $c_{\ell,m}$ as described in (\ref{g}).
\end{remark}
\bigskip


\medskip

\section{The case of a generalized zero in the first column of $\varphi$.}\label{3}

\begin{data}\label{data3} Adopt Data {\rm \ref{data1}} with $d_1<d_2$. Assume that
   there is a generalized zero in the first column of $\varphi$.
\end{data}

In this section we are in the situation of Data \ref{data3} and  we
describe $\mathcal A_{\ge d_1-1}$ as an $S$-module and as a
$B$-module. The hypothesis that there is a generalized zero in the
first column of $\varphi$ has geometric significance. Indeed, as described in Remark \ref{curve}, the
homogeneous minimal generators of $I$ describe a morphism
$\eta:\mathbb P^1_k\to \mathbb P^2_k$ whose image is a rational curve
$\mathcal C$. If the morphism $\eta$ is birational onto $\mathcal C$
(or equivalently, if the curve  $\mathcal C$ has degree $d$), then
the assumption that $\varphi$ has a generalized zero in the first
column is equivalent to the assumption that  $\mathcal C$ has a
singularity of multiplicity $d_2$; see the General Lemma in
\cite[1.9]{CKPU}; or \cite[Thm.~3]{SCG} or \cite[Thm.~1]{CWL}.

The $S$-module structure of $\mathcal A_{\ge d_1-1}$ is completely described in Theorem \ref{claudia}: $\mathcal A_{\ge d_1-1}$ is a free $S$-module of finite rank and a complete list of the  bi-degrees of an $S$-module basis for  $\mathcal A_{\ge d_1-1}$ is given; see also Table \ref{tbl1}. The $B$-module structure of $\mathcal A_{\ge d_1-1}$ is described in Corollary \ref{B-mod}

The part of $\Sym(I)$ that corresponds to $\mathcal A_{\ge d_1-1}$, under the duality of Theorem \ref{iso}, is $\Sym(I)_{\le d_2-1}$. There is no contribution from $g_2$ to the $S$-module $\Sym(I)_{\le d_2-1}$ in the bi-homogeneous  $B$-resolution of $\Sym(I)$. So, basically, we may ignore $g_2$ in the present section. Furthermore, the hypothesis that the first column of $\varphi$ has a generalized zero allows us to make the critical calculation over a   subring $U$ of $S$, where $U$ is a polynomial ring in two variables. In the proof of Theorem \ref{claudia}, we decompose various  bi-graded
complexes over $R\otimes_kU$ into their $R$-graded components and their $U$-graded components. Ultimately, the critical calculation is to produce a lower bound for the degrees of the syzygies of a $U$-module homomorphism.

Lemma \ref{Hilseries}  is a statement about $U$-module homomorphisms. This lemma explains how, sometimes, 
lower bounds for the degrees of syzygies suffice to determine these degrees.
The statement of Lemma \ref{Hilseries} may be deduced from a
 classification of   matrices   whose entries are linear forms from $U$. This classification was known by  Weierstrass (in the singular case) and Kronecker (in the general case); see \cite[Chapt XII]{Gant}. The proof we give for  Lemma \ref{Hilseries} uses Hilbert series to relate the twists in a homogeneous resolution  to the betti numbers. This  technique, which is now standard, was introduced by Peskine and Szpiro \cite{PS} and was used with great success by Herzog and K\"uhl \cite{HK};  it was also a motivation for   Boij-S\"oderberg  theory.

\begin{lemma}\label{Hilseries} Let $M$ be a  graded module of finite length with a linear presentation over the polynomial ring   $U=k[T_1,T_2]$.
If a homogeneous resolution of $M$ has the form
$$0\lto \bigoplus_{\ell=1}^m U(-b_\ell)\lto U(-1)^n\lto U^{n-m}\lto M\lto 0\, ,$$ then $\sum_{\ell=1}^m b_{\ell}=n$.
\end{lemma}

\begin{proof} The Hilbert series of $M$  is $h_M(t)/(1-t)^2$, where $h_M(t)=n-m-nt+\sum_{\ell=1}^mt^{b_{\ell}}$. Since $M$ has finite length, this series is a polynomial; hence, $(1-t)^2$ divides $h_M(t)$. Therefore, $h'_M(1)=0$, which gives the assertion. \QED \end{proof}

\begin{theorem}\label{claudia} Adopt Data {\rm\ref{data3}}. If $d_1-1\le i\le d_2-1$, then $$\mathcal A_i\simeq \bigoplus\limits_{\ell=1}^{d_1}S(-a_{\ell}),$$ where
$$\left\lfloor \frac{d+d_1-1-i}{d_1} \right\rfloor=a_1\le \dots \le a_{d_1}= \left\lceil \frac{d+d_1-1-i}{d_1} \right\rceil$$ and $$ \sum\limits_{\ell=1}^{d_1}a_{\ell}=d+d_1-1-i\, .$$
\end{theorem}

\begin{remark}\label{alt}We offer an alternate phrasing for Theorem \ref{claudia}. If $d_1-1\le i\le d_2-1$ and ${\alpha_i}$ and ${\beta_i}$ are integers with
$$d+d_1-1-i={{{\alpha_i}}} d_1+{\beta_i}\quad\text{and}\quad  0\le {\beta_i}\le d_1-1,$$ then 
 $$\mathcal A_i\simeq S(-{{\alpha_i}})^{d_1-{\beta_i}}\oplus S(-\alpha_i-1)^{{\beta_i}}.$$ 
Of course, in this language, ${{\alpha_i}}$ is equal to $\left\lfloor \frac{d+d_1-1-i}{d_1} \right\rfloor$ and ${\beta_i}$ is the ``remainder that is obtained when $d+d_1-1-i$ is divided by $d_1$''. Observe that the parameter $\alpha_i$ is always at least $2$,  the exponent $d_1-{\beta_i}$ is positive, and the other exponent, ${\beta_i}$, is non-negative.  
\end{remark}

\begin{proof1} The case $i=d_2-1$ is covered in part (3) of Corollary \ref{chart}. Fix an integer $i$ with $$d_1-1\le i\le d_2-2\, .$$
We prove the following four ingredients.
\begin{equation}\label{free}\text{The $S$-module $\mathcal A_i$ is free of rank $d_1$.}\end{equation}
Once (\ref{free}) is established, then we define the shifts $a_{\ell}$ by $\mathcal A_i\simeq \bigoplus_{\ell=1}^{d_1}S(-a_{\ell})$ and we prove
\begin{equation}\label{sum}\sum\limits_{\ell=1}^{d_1}a_{\ell}=d+d_1-1-i, \end{equation} 
\begin{equation}\label{guts}\mathcal A_{(i,j)}=0\quad\text{for $j\le \alpha_i-1$},\quad\text{and}\end
{equation}
\begin{equation}\label{four}\dim_k \mathcal A_{(i,{{\alpha_i}})}=d_1-{\beta_i},\end{equation} where we have used the   language of Remark \ref{alt}.
Once the four ingredients have been established, then it is not difficult to complete the proof. If $\beta_i=0$, then the result follows immediately from (\ref{free}) and (\ref{four}); and if $0<\beta_i$, then one may apply the pigeon hole principle.    
Indeed, if (\ref{free})~--~(\ref{four}) hold and $0< \beta_i$, then
$${\alpha_i}+1\le a_{d_1-{\beta_i}+1}\le \dots\le a_{d_1}$$ and 
$$\begin{array}{rl}({{\alpha_i}}+1){\beta_i}\le a_{d_1-{\beta_i}+1}+\dots+a_{d_1}&\hskip-10pt{}=\sum\limits_{\ell=1}^{d_1} a_\ell -(d_1-{\beta_i}){{\alpha_i}}=(d+d_1-1-i)-(d_1-{\beta_i}){\alpha_i}\\&\hskip-10pt{}=({\alpha_i} d_1+{\beta_i})-(d_1-{\beta_i}){\alpha_i}=({\alpha_i}+1){\beta_i}\, ;\end{array}$$
hence, $a_{\ell} ={\alpha_i}+1$ for
$d_1-{\beta_i}+1\le  \ell \le d_1$.

The assertion about the rank of $\mathcal A_i$ is established in part (2) of   Theorem \ref{chart}.

The hypothesis that the first column of $\varphi$ has a generalized zero ensures that after performing row operations on $\varphi$  and renaming the generators of $S_1$, we have that $\varphi$ has the form $$\varphi=\left(\begin{matrix} f_1&*\\f_2&*\\0&*\end{matrix}\right)$$ and $g_1$, which is equal to $[T_1,T_2,T_3]$ times the first column of $\varphi$, only involves $T_1$ and $T_2$. Indeed,    $g_1=f_1T_1+f_2T_2$.
The hypothesis from Data \ref{data1} that $I$ has height two guarantees that $f_1$ and $f_2$ are a regular sequence of forms of degree $d_1$  in $R$.
At this point, we introduce the subrings $U=k[T_1,T_2]$ of $S$ and $C=R\otimes_kU$ of $R\otimes_k S=B$. Notice that $C[T_3]=B$ and $g_1$ is in $C$.

Let $K_{\bullet}$ be the following bi-graded complex of $C[T_3]$-modules:
\begin{equation}\label{Kdot}K_{\bullet}: \quad C[T_3](-d_1,-1) \stackrel{g_1}{\longrightarrow} C[T_3] \longrightarrow {\rm Sym}(I)\to 0.\end{equation}
The graded strand
 $[K_{\bullet}]_{(\delta-i,\underline{\phantom{x}})}$ of $K_{\bullet}$ is exact because $\delta-i\le d_2-1$ by our assumption $d_1-1\le i$.
Taking graded $S$-duals  we obtain the complex
$$\underline{\rm Hom}_S(K_{\bullet},S):\quad 0\to
\underline{\rm Hom}_S( {\rm Sym}(I),S) \longrightarrow
 \underline{\rm Hom}_S(C[T_3],S)\stackrel{g_1}{\longrightarrow}
\underline{\rm Hom}_S(C[T_3],S)(d_1,1).$$Furthermore, the graded strand
$\underline{\rm Hom}_S(K_{\bullet},S)_{(i-\delta,\underline{\phantom{x}})}$ is exact. Theorem \ref{iso} guarantees that \begin{equation}\label{star}\underline{\rm Hom}_S( {\rm Sym}(I),S)\simeq \mathcal A(\delta,2).\end{equation}
Observe that
\begin{equation}\label{2star}\underline{\rm Hom}_S(C[T_3],S)\simeq \underline{\rm Hom}_S(R\otimes_kS,S)=\underline{\rm Hom}_k(R,k)\otimes_kS \simeq \underline{\rm Hom}_k(R,k)\otimes_kU \otimes _k k[T_3].\end{equation} Combine (\ref{star}) and (\ref{2star}) to see that  the complex   $\underline{\rm Hom}_S(K_{\bullet},S)$ may be identified with
$$0\to \mathcal A(\delta,2) \to  \underline{\rm Hom}_k(R,k)\otimes_kU \otimes _k k[T_3]
\stackrel{g_1}{\longrightarrow}\underline{\rm Hom}_k(R,k)(d_1)\otimes_kU(1) \otimes _k k[T_3].$$ Since $g_1\in R\otimes_k U$, multiplication by $g_1$ gives a bi-homogeneous $(R\otimes_kU)$- module homomorphism $$\psi: \underline{\rm Hom}_k(R,k)\otimes_kU  
\stackrel{g_1}{\longrightarrow}
\underline{\rm Hom}_k(R,k)(d_1)\otimes_kU(1).$$
We focus on the kernel and cokernel of various $R$-graded and $U$-graded components of $\psi$.
We have
\begin{equation}\label{wehave}(\operatorname{ker} \psi)_{i-\delta}(-2)[T_3] \simeq \mathcal A_{i},\quad\text{for $d_1-1\le i\le d_2-2$}. \end{equation} It follows immediately that $\mathcal A_i$ is free over $S$ because $\ker \psi_{i-\delta}$  is a second syzygy over $U$, a polynomial ring in $2$ variables.
Thus, (\ref{free}) is established. Notice that $(\operatorname{ker} \psi)_{i-\delta}\simeq \bigoplus_{\ell=1}^{d_1}U(-a_\ell+2)$. The left-most map in the complex $K_{\bullet}$ of (\ref{Kdot}) is injective; consequently, a calculation similar to the one that produced (\ref{wehave}) yields
 $$(\operatorname{coker} \psi)_{i-\delta}[T_3]\simeq \operatorname{Ext}^1_S({\rm Sym}(I)_{\delta-i},S)\quad\text{for $d_1-1\le i\le d_2-2$}.$$ According to part (1.a) of Theorem \ref{future}, $\operatorname{Ext}^1_S({\rm Sym}(I)_{\delta-i},S)$ vanishes locally in codimension one whenever $d_1-1\le i\le d_2-2$; hence,  $(\operatorname{coker} \psi)_{i-\delta}$ has finite length as a $U$-module.
Now we can apply Lemma \ref{Hilseries} to $(\operatorname{coker} \psi)_{i-\delta}(-1)$, which has a homogeneous free $U$-resolution of the form:
$$\hskip-2pt 0\to (\ker \psi)_{i-\delta}(-1)\to \underline{\rm Hom}_k(R,k)_{i-\delta}\otimes_kU(-1)
 \to \underline{\rm Hom}_k(R,k)_{i-\delta+d_1}\otimes_k U \to (\operatorname{coker} \psi)_{i-\delta}(-1)\to 0.$$
As $(\operatorname{ker} \psi)_{i-\delta}(-1)\simeq \bigoplus_{\ell=1}^{d_1}U(-a_\ell+1)$ and
$$\dim_k \underline{\rm Hom}_k(R,k)_{i-\delta}= \dim_kR_{\delta-i}=\delta+1-i\, ,$$  Lemma \ref{Hilseries} gives that $\sum_{\ell=1}^{d_1}(a_{\ell}-1)=\delta+1-i$; or equivalently, $\sum_{\ell=1}^{d_1}a_{\ell}=d+d_1-1-i$. The second ingredient, (\ref{sum}), has been established.

We re-phrase (\ref{wehave}) as
\begin{equation}\label{Aij} \mathcal A_{(i,j)}\simeq (\operatorname{ker} \psi_{(i-\delta,j-2)})[T_3],\quad\text{for $d_1-1\le i\le d_2-2$ and all $j$}, \end{equation} in order to focus on the  individual components $\mathcal A_{(i,j)}$ of the free $S$-module  $\mathcal A_i$. Each such component is a finite dimensional $k$-vector space.
We see that

\begin{tabular}{lcllrclcccc}
Claim  (\ref{guts}) holds  &$\iff$&$\mathcal A_{(i,j)}=0$&\quad for& $j$&$\le$&$\left\lfloor  \frac{d-1-i}{d_1}\right\rfloor$\\\\
  &$\iff$&$(\operatorname{ker} \psi)_{(i-\delta,j-2)}=0$&\quad for& $j-2$&$\le$&$\left\lfloor \frac{d-1-i}{d_1}\right\rfloor-2 $
\\\\
&$\iff$&$(\operatorname{ker} \psi)_{(m,n)}=0$&\quad for& $n$&$\le$&$\left\lfloor \frac{1-m}{d_1}\right\rfloor-2\, $.
\end{tabular}

\noindent To prove the most recent version of (\ref{guts}) we fix an integer $n$ and consider the homogeneous $R$-module homomorphism
$$\psi_{(\underline{\phantom{xx}},n)}: \underline{\rm Hom}_k(R,k)\otimes_k U_n\longrightarrow
\underline{\rm Hom}_k(R,k)(d_1)\otimes_k U_{n+1}.$$ Recall that $\psi$ is multiplication by $f_1T_1+f_2T_2$. We choose the $k$-bases $T_1^n,\dots,T_2^n$ and $T_1^{n+1},\dots,T_2^{n+1}$ for $U_n$ and $U_{n+1}$, respectively and we see that the $R$-module homomorphism $\psi_{(\underline{\phantom{xx}},n)}$ is represented by the $(n+2)\times (n+1)$ matrix
$$\left[\begin{matrix}
f_1&0 &0&\cdots&0\\
f_2&f_1&0&\cdots&0\\
0& f_2&f_1&\cdots&0\\
\vdots&\vdots&\ddots&\ddots&\vdots\\
0&\cdots&0&f_2&f_1\\
0&\cdots&0&0&f_2\end{matrix}\right].$$
Take the graded $k$-dual of $\psi_{(\underline{\phantom{xx}},n)}$ to obtain a homogeneous $R$-module homomorphism $${_n}\chi: R(-d_1)\otimes_k U_{n+1}^*\longrightarrow R\otimes _kU_n^*,$$ 
given by the $(n+1)\times (n+2)$ matrix
$$\left[\begin{matrix} f_1&f_2&0&\cdots&0\\
0&f_1&f_2&\cdots&0\\
\vdots&\vdots&\ddots&\ddots&\vdots\\
0&0&\cdots&f_1&f_2\end{matrix}\right],$$ with entries in $R$, where  $(\underline{\phantom{x}})^*$ means $\operatorname{Hom}_k(\underline{\phantom{x}},k)$. (The name ${_n}\chi$ emphasizes that this map depends on $n$; however, this map is not created as a graded strand of some other, previously named,  object.)
Notice that $\psi_{(m,n)}$ is injective  if and only if the graded component ${_n\chi}_{-m}$ of ${_n\chi}$ is surjective;
moreover, in general, \begin{equation}\label{ker=coker}\dim_k \operatorname{ker}\psi_{(m,n)}=\dim_k \operatorname{coker} {_n\chi}_{-m}.\end{equation}
 The Buchsbaum-Rim complex yields a homogeneous free resolution  of $\operatorname{coker} {_n\chi}$:
\begin{equation}\label{Bu-Ri}0\to R(-(n+2)d_1)\to R(-d_1)\otimes_kU_{n+1}^* \stackrel{{_n\chi}}{\longrightarrow} R\otimes_kU_n^*.\end{equation} This resolution shows that the socle of $\operatorname{coker} {_n\chi}$ is concentrated in degree $(n+2)d_1-2$. Thus, ${_n\chi}_{-m}$ is surjective if $-m\ge (n+2)d_1-1$. However $ -m\ge (n+2)d_1 -1 \iff n\le \frac{1-m}{d_1}-2$, and therefore Claim (\ref{guts}) holds.

We now compute $\dim_k\mathcal A_{(i,{\alpha_i})}$, as required by Claim (\ref{four}).  The generators of the free $S$-module $\mathcal A_i$ all have degree at least ${\alpha_i}$.
Therefore, the number of minimal generators of degree $\alpha_i$ for the $S$-module $\mathcal A_i$ is the dimension of the $k$-vector space $\mathcal A_{(i,\alpha_i)}$, and, by  
 (\ref{Aij}), the vector spaces $\mathcal A_{(i,{\alpha_i})}$ and 
$\operatorname{ker} \psi_{(i-\delta,{\alpha_i}-2)}$ have the same dimension. In this calculation, we dig more deeply into   the     sequence of ideas that was used in the proof of (\ref{guts}). In particular, (\ref{Aij}) and (\ref{ker=coker}) give
\begin{equation}\label{E-1}\dim_k\mathcal A_{(i,{\alpha_i})}=\dim_k\operatorname{ker} \psi_{(i-\delta,{\alpha_i}-2)}=\dim_k(\operatorname{coker}(_{{\alpha_i}-2}\chi_{\delta-i})).\end{equation}
We again use (\ref{Bu-Ri}) to make this computation: 
\begin{equation}\label{E0}\begin{array}{rl}\dim_k(\operatorname{coker}(_{{\alpha_i}-2}\chi_{\delta-i}))&\hskip-10pt{}=\dim_k (R\otimes_kU_{{\alpha_i}-2}^*)_{ \delta-i }-\dim_k(R(-d_1)\otimes_kU^*_{{\alpha_i}-1})_{ \delta-i }\\&\hskip-10pt\phantom{{}={}}{}+\dim_k R(-{\alpha_i} d_1)_{ \delta-i}\, .\end{array}\end{equation}
  We notice that 
\begin{equation}\label{E1}(R(-d_1)\otimes_kU^*_{{\alpha_i}-1})_{\delta-i}\ne 0 \quad\text{and}\end{equation}\begin{equation}\label{E2}  R(-{\alpha_i} d_1)_{\delta-i}=0\, ,\end{equation} for $d_1-1\le i\le d_2-2$. Indeed, the hypothesis $i\le d_2-2$ guarantees that $0\le d_2-2-i=\delta-d_1-i$\,; hence,  $d_1\le \delta-i$ and (\ref{E1}) holds. Also, $a-(b-1)\le \lfloor \frac ab\rfloor b$ for all positive  integers $a$ and $b$; hence,  
$$\delta-i<d-i=(d+d_1-1-i)-(d_1-1)\le \left\lfloor\frac{d+d_1-1-i}{d_1}\right\rfloor d_1={\alpha_i} d_1$$
and (\ref{E2}) holds.
Combine (\ref{E-1}), (\ref{E0}), (\ref{E1}), and (\ref{E2}) to see that  $$\begin{array}{rcl}\dim_k\mathcal A_{(i,{\alpha_i})}=\dim_k(\operatorname{coker}(_{{\alpha_i}-2}\chi_{\delta-i})) &=&\dim_k (R\otimes_kU^*_{{\alpha_i}-2})_{\delta-i}-\dim_k(R(-d_1)\otimes_kU^*_{{\alpha_i}-1})_{\delta-i}\\
 &=&(\delta-i+1)({\alpha_i}-1)-(\delta-i+1-d_1){\alpha_i}\\
&=&-(\delta-i+1)+d_1{\alpha_i}\\
&=&-(d-1-i)+d_1{\alpha_i}\\
&=&d_1-(d+d_1-1-i)+d_1{\alpha_i}\\
&=&d_1-(d_1{\alpha_i}+{\beta_i})+d_1{\alpha_i}\\
&=&d_1-{\beta_i}\, , 
\end{array}$$ which completes the proof of (\ref{four}). All four claims (\ref{free}) -- (\ref{four}) have been established. The proof is complete. 
\QED
\end{proof1}

\begin{Table}\label{table}Adopt Data \ref{data3}. Table \ref{tbl1} records the $S$-module structure of $$\mathcal A_{\ge d_1-1}\simeq\bigoplus S(-(i,j))^{n_{i,j}}.$$ The exponent $n_{i,j}$ sits in position $(i,j)$, where $i$ is plotted on the horizontal axis and $j$ is plotted on the vertical axis. In other words, a minimal homogeneous basis for the free $S$-module $\mathcal A_{i}$    has $n_{i,j}$ generators of bi-degree $(i,j)$.

 We describe, in words,   the transition of the generator degrees {\bf to} $\mathcal A_{i-1}$ {\bf from} $\mathcal A_i$,  beginning at the right side of the table. Recall that $\mathcal A_{i}=0$ for $\delta+1\le i$.

If $d_2-1\le i\le \delta$, then the generators of $\mathcal A_i$ are concentrated in the unique degree $2$. The rank of $\mathcal A_{\delta}$ is $1$ and if  $d_2-1\le i\le \delta-1$, then $\operatorname{rank} \mathcal A_{i}=\operatorname{rank} \mathcal A_{i+1}+1$. The relevant proof for this part of the table is contained in Corollary \ref{XXX}. Theorem \ref{claudia} contains the proof for the range $d_1-1\le i\le d_2-1$. In this range the rank of $\mathcal A_i$ remains constant at $d_1$ and the generators of $\mathcal A_i$ live in two degrees, or, occasionally, only one degree. As one looks from right to left, one free rank one  summand of {\bf lowest} shift in $\mathcal A_i$ is replaced by  a free rank one summand with shift {\bf one} higher in $\mathcal A_{i-1}$.

One continues this pattern all the way until the left boundary of $\mathcal A_{\ge d_1-1}$; namely $i=d_1-1$. The shifts in $\mathcal A_{d_1-1}$ are given by $\lfloor \frac d{d_1}\rfloor: (d_1-r)$ and $(\lfloor \frac d{d_1}\rfloor+1): r$ where $r$ is the remainder of $d$ upon division by $d_1$; that is, $d=d_1\lfloor \frac d{d_1}\rfloor+r$, with $0\le r\le d_1-1$, and $$\textstyle{\mathcal A_{d_1-1}= S(-\lfloor \frac d{d_1}\rfloor)^{d_1-r}\oplus S(-\lfloor \frac d{d_1}\rfloor-1)^r.}$$

\begin{table}
\begin{center}
\begin{tabular}{|c||c|c|c|c|c|c|c|c|c|c|c|c|c|}\hline
${\scriptstyle T\text{-deg}}$&&&&&&&&&&&&&\\\hline
$\lfloor \frac d{d_1}\rfloor+1$&${\scriptstyle r}$&$\cdots$&&&&&&&&&&&\\\hline
$\lfloor \frac d{d_1}\rfloor$&${\scriptstyle d_1-r}$&$\cdots$&&&&&&&&&&&\\\hline
\vdots& &&&&&&&&&&&&\\\hline
${\scriptstyle  \lfloor \frac d{d_1}\rfloor-\lambda+1}$&&$\cdots$&${\scriptstyle d_1}$&${\scriptstyle d_1-1}$&${\scriptstyle d_1-2}$&$\cdots$& &&&&&&\\\hline
${\scriptstyle  \lfloor \frac d{d_1}\rfloor-\lambda}$&& &&${\scriptstyle 1^*}$&${\scriptstyle 2}$&$\cdots$& &&&&&&\\\hline
\vdots& &&&&&&&&&&&&\\\hline
${\scriptstyle 3}$&& &&&&$\cdots$& ${\scriptstyle 1}$&&&&&&\\\hline
${\scriptstyle 2}$&& &&&&$\cdots$&${\scriptstyle d_1-1}$&${\scriptstyle d_1}$&${\scriptstyle d_1-1}$&${\scriptstyle d_1-2}$&$\cdots$&${\scriptstyle 1}$&\\
\hline\hline
&${\scriptstyle d_1-1}$&$\cdots$&${\scriptstyle\lambda d_1+r-1}$&${\scriptstyle\lambda d_1+r}$&${\scriptstyle\lambda d_1+r+1}$&$\cdots$&${\scriptstyle d_2-2}$&${\scriptstyle d_2-1}$&${\scriptstyle d_2}$&${\scriptstyle d_2+1}$&$\cdots$&${\scriptstyle \delta}$&${\scriptstyle xy\text{-deg}}$\\\hline
\end{tabular}

\medskip

\caption{{\bf The generator degrees for the free $S$-module $\mathcal A_{\ge d_1-1}$.} This table is described in words in Table \ref{table}. {\rm(}Also, we call the $1^*$ that appears in position ${(i,j)=(\lambda d_1+r,\lfloor \frac d{d_1}\rfloor-\lambda)}$
an ``exterior corner point''. We never refer to any entry in the left-most column as an exterior corner point. In Remark {\rm\ref{ecp}} we calculate that the exterior corner points occur when  $1\le \lambda\le \lfloor \frac d{d_1}\rfloor-2$.{\rm)}}\label{tbl1}
%
%
%
\end{center}
\end{table}
\end{Table}

\begin{remark}\label{ecp} The ``exterior corner points'' $(i,j)$ in Table \ref{tbl1} are very important
when one considers the $B$-module structure of $\mathcal A_{\ge
d_1-1}$ because degree considerations show that the corresponding
basis element of the $k$-vector space $\mathcal A_{(i,j)}$ is part
of a minimal bi-homogeneous generating set for the $B$-module
$\mathcal A_{\ge d_1-1}$. For this reason we carefully record where
these corner points occur.  Continue to write $d=d_1\lfloor \frac
d{d_1}\rfloor+r$, with $0\le r\le d_1-1$. 
Let $(i,j)$ be integers.
 We claim that    \begin{equation}\label{11-10}
d_1\le i\le
d_2-1 \quad\text{and} \quad \mathcal A_{i}=S(-j)^1\oplus S(-j-1)^{d_1-1}\end{equation} if
and only if $(i,j)=(\lambda d_1+r,\lfloor \frac
d{d_1}\rfloor-\lambda)$ for some integer $\lambda$ with $1\le
\lambda \le \lfloor \frac d{d_1}\rfloor-2$. 
(Notice that we  never refer to any entry in the left-most column of Table \ref{tbl1} as an exterior corner point.)
\end{remark}

\begin{proof} Fix an integer $j$. Apply Theorem \ref{claudia}, by way of Remark \ref{alt}, to see that  
$$\begin{array}{ccl}\text{(\ref{11-10}) holds} &\iff &d_1\le i\le
d_2-1 \quad\text{and}\quad d+d_1-1-i=jd_1+d_1-1\\&\iff &d_1\le i\le
d_2-1 \quad\text{and}\quad d-i=jd_1.\end{array}$$
 Write $d=d_1\lfloor \frac
d{d_1}\rfloor+r$ to see that 
$$\begin{array}{ccl}\text{(\ref{11-10}) holds} &\iff &d_1\le i\le
d_2-1 \quad\text{and}\quad d_1\lfloor \frac
d{d_1}\rfloor+r-i=d_1j\\
&\iff &d_1\le i\le
d_2-1 \quad\text{and}\quad \lfloor \frac
d{d_1}\rfloor-\frac{i-r}{d_1}=j.
\end{array}$$ Let $\lambda= \frac{i-r}{d_1}$. It follows  that
$$\begin{array}{ccl}\text{(\ref{11-10}) holds} &\iff &\begin{cases}\text{there exists an integer $\lambda$ with}
\ i=\lambda d_1+r,\ j= \lfloor \frac d{d_1}\rfloor-\lambda,\ \text{and}\\
d_1\le \lambda d_1+r\le
d_2-1.\end{cases}\end{array}$$
To complete the proof we show that 
$$\textstyle{d_1\le \lambda d_1+r\le d_2-1 \iff 1\le \lambda \le \lfloor \frac d{d_1}\rfloor-2.}$$
The left hand inequalities $d_1\le \lambda d_1+r$ and $1\le \lambda$ are equivalent because   $0\le r\le d_1-1$.
To see that $$\textstyle{\lambda d_1+r\le d_2-1 \iff \lambda \le \lfloor \frac d{d_1}\rfloor-2,}$$ add $d_1-r$ to both sides of the left hand inequality and use $d_1+d_2-r=d_1\lfloor \frac{d}{d_1}\rfloor $ to see that
$$\textstyle{\lambda d_1+r\le d_2-1\iff \lambda d_1+d_1\le d_1 \lfloor \frac{d}{d_1}\rfloor-1\iff \lambda\le \lfloor \frac{d}{d_1}\rfloor -\frac 1{d_1}-1\iff \lambda \le  \lfloor \frac{d}{d_1}\rfloor -2.}$$
\QED
\end{proof}

\begin{remark} In the notation of Table \ref{table}, we see that

\begin{itemize}\item[{\rm(1)}]$\mathcal A_{(d_1-1,\lfloor \frac d{d_1}\rfloor)}$ is a $k$-vector space of dimension $d_1-r$,

\item[{\rm(2)}] $\mathcal A_{(d_1-1,\lfloor \frac d{d_1}\rfloor+1)}$ is a $k$-vector space of dimension $r$,  and

\item[{\rm(3)}] if $\lambda$ is an index with $1\le \lambda\le \lfloor \frac d{d_1}\rfloor-2$, then
$\mathcal A_{(\lambda d_1+r,\lfloor \frac d{d_1}\rfloor-\lambda)}$
is a $k$-vector space of dimension $1$.
\end{itemize}
\end{remark}

\begin{definition}\label{elts}
Let $v_{1},\dots, v_{r}$ be a $k$-basis for   $\mathcal
A_{(d_1-1,\lfloor \frac d{d_1}\rfloor+1)}$; $u_1,\dots, u_{d_1-r}$ be a
$k$-basis for $\mathcal A_{(d_1-1,\lfloor \frac d{d_1}\rfloor)}$;
for each $\lambda$, with $1\le \lambda\le \lfloor \frac
d{d_1}\rfloor-2$, let $w_{\lambda}$ be a $k$-basis for $\mathcal
A_{(\lambda d_1+r,\lfloor \frac d{d_1}\rfloor-\lambda)}$.
\end{definition}

In Corollary \ref{B-mod} we prove that the elements of Definition \ref{elts} form a minimal bi-homogeneous generating set for the $B$-module $\mathcal A_{\ge d_1-1}$.  In the proof of  Corollary \ref{B-mod} we appeal to  the following fact about Hilbert functions that may well be
known to experts.

\begin{lemma}\label{Andy} Let $k$ be an algebraically closed field, $R$ a
positively graded $k$-algebra with $\dim_k R_1 >1$, $M$ a graded
$R$-module and $i \ge 0$ a fixed integer. Assume that for every
non-zero element $\ell \in R_1$ the map $M_i \to M_{i+1}$ induced by
multiplication with $\ell$ is injective. Then for any non-zero
finite dimensional $k$-subspace $V$ of $M_i$ one has $\dim_k R_1V
> \dim_k V$.
\end{lemma}

\begin{proof} Suppose that $\dim_k R_1 V \le \dim_k V$. Choose two
$k$-linearly independent elements $x$ and $y$ in $R_1$. Since
$\dim_k xV=\dim_k yV=\dim_k V \ge \dim_k R_1V$, it follows that
$xV=R_1V=yV$. Let $e_1, \ldots, e_n$ be a $k$-basis of $V$. One has
\begin{equation}\label{eqM}
x[e_1, \ldots, e_n]=y [e_1, \ldots, e_n] \cdot \Phi
\end{equation}
for some $n \times n$ matrix $\Phi$ with entries in $k$. Let $v\in
k^n$ be a non-zero  eigenvector of $\Phi$ belonging to an eigenvalue $\lambda
\in k$ and write $z=[e_1, \ldots, e_n] \cdot v$. Notice that $z$ is
a non-zero element of $M_i$. Multiplying equation (\ref{eqM}) by $v$
from the right we obtain $xz=\lambda y z$. Thus $(x-\lambda y)z=0$
in $M_{i+1}$. Since $x-\lambda y$  is a non-zero element in $R_1$,
we have a contradiction to our assumption on $M$. \QED
\end{proof}

\begin{corollary}\label{B-mod}  Adopt Data \ref{data3}. The set $\{v_{1},\dots, v_{r}\} \cup \{u_1,\dots, u_{d_1-r}\} \cup
\{ w_{\lambda}\mid 1\le \lambda\le \lfloor \frac d{d_1}\rfloor-2\} $ of elements from Definition {\rm\ref{elts}}
is a minimal bi-homogeneous generating set for the $B$-module
$\mathcal A_{\ge d_1-1}$.
\end{corollary}
\begin{proof} We may harmlessly assume that $k$ is algebraically
closed. Write $X=\{v_{\epsilon}\} \cup \{u_{\mu}\} \cup \{
w_{\lambda}\} $ for the set described in the statement. First notice
that no element of $X$ is a $B$-linear combination of the others.
This is a consequence of following facts: $\mathcal A_{\ge d_1-1}$
is concentrated in $R$-degrees $\ge d_1-1$, $\{v_{\epsilon}\} \cup
\{u_{\mu}\}$ forms a homogeneous $S$-basis of $\mathcal A_{d_1-1}$
concentrated in $S$-degrees $\ge \lfloor \frac d{d_1}\rfloor$ (see
Theorem \ref{claudia}), the $w_{\lambda} $ have $R$-degrees $>
d_1-1$ and $S$-degrees $<\lfloor \frac d{d_1}\rfloor$, and as
$\lambda$ increases the $R$-degrees of the $w_{\lambda} $ increase
strictly whereas their $S$-degrees decrease strictly.

Thus it suffices to show that $\mathcal A_{\ge d_1-1}= X B$. To do
so we prove by induction on $i \ge d_1-1$ that
\begin{equation}\label{inductionstep}\mathcal A_i \subset \mathcal A_{i-1}R_1 +
XB.\end{equation} Notice that from part (1) of Corollary \ref{XXX}
we have $\mathcal A_{\ge d_2} \subset \mathcal A_{d_2-1}R$. Hence we
may assume that $i$ is in the range $d_1-1 \le i\ \le d_2-1$. For
$i=d_1-1$ equation (\ref{inductionstep}) follows from Theorem
\ref{claudia}. As for the induction step let $i \ge d_1$. From Table
\ref{tbl1} we know that $\mathcal A_{i-1}$ is a graded free
$S$-module with $t$ homogeneous basis elements in degree $j$ and
$d_1-t$ homogeneous basis elements in the next higher degree $j+1$,
where $t$ is in the range $0 \le t \le d_1-1$. Write $U$ for the
$k$-subspace of $\mathcal A_{(i-1,j)}$ spanned by the basis elements
of degree $j$ and $V$ for the $k$-subspace of $\mathcal
A_{(i-1,j+1)}$ spanned by the basis elements of degree $j+1$.
Likewise, the graded $S$-module $\mathcal A_i$ has $t+1$ homogeneous
basis elements of degree $j$ and $d_1-t-1$ homogeneous basis
elements of degree $j+1$. These basis elements span $k$-subspaces $W
\subset \mathcal A_{(i,j)}$ and $Z\subset \mathcal A_{(i,j+1)}$,
respectively. Notice that $W=\mathcal A_{(i,j)}$. Finally, write $Y$ for the $k$-subspace of $\mathcal
A_{(i,j)}$ spanned by the elements $w_{\lambda} $ of bi-degree
$(i,j)$. Accordingly it suffices to prove that
\begin{equation}\label{EQ1}
R_1U+Y=W
\end{equation}
and
\begin{equation}\label{EQ2} S_1W+xV=S_1W+Z.
\end{equation}

If $U=0$ then $t=0$ and hence $Y=W$ by definition, which shows
(\ref{EQ1}) in this case. Now we turn to the case $U \neq 0$.  We
first argue that we may apply Lemma \ref{Andy}. Recall that the
element $g_1 \in B=R[T_1,T_2,T_3]$ is of the form
$T_1f_1+T_2f_2+T_3f_3$, where $f_1, f_2, f_3$ are the entries of the
first column of $\varphi$. These entries generate an $R$-ideal of
height $2$ because $I=I_2(\varphi)$ has height two, and therefore
$g_1=T_1f_1+T_2f_2+T_3f_3$ is a prime element of $B$ according to
\cite[Theorem]{hochster}. Thus $B/(g_1)$ is a domain. Since $[{\rm
Sym}(I)]_{(m,\underline{\phantom{x}})}=[B/(g_1)]_{(m,\underline{\phantom{x}})}$
 for $m\le d_2-1$ and since $i \le d_2-1 $, we deduce that
 multiplication by any linear form in $R$ induces an injective
 $S$-linear map from $[{\rm
Sym}(I)]_{(i-1,\underline{\phantom{x}})}$ to $[{\rm
Sym}(I)]_{(i,\underline{\phantom{x}})}$. Now Lemma \ref{Andy} shows
that $\dim_k R_1U \ge \dim_k U +1$ since $U\neq 0$. On the other
hand, we have $R_1U \subset \mathcal A_{(i,j)}=W$, $\dim_k U =t$, and $\dim_k W = t+1$.
Thus Claim (\ref{EQ1}) has been proved.

To show (\ref{EQ2}) we suppose that the containment $S_1W +xV\subset
S_1W+Z$ is strict. In this case there exists a proper $k$-subspace
$Z'$ of $Z$,  so that
$$ S_1W+xV\subset S_1W+Z'.
$$
In particular, $x(S_1U+V) \subset S_1W +Z'$ and then
\begin{equation}\label{EQ3}
x(SU+SV) \subset SW+SZ'.
\end{equation}
However, the first $S$-module is isomorphic to $\mathcal
A_{(i-1,\underline{\phantom{x}})}$ and thus has rank $d_1$, whereas
the number of generators of the  second $S$-module is at most
$$ \dim_k W + \dim_k Z' \le \dim_k W + \dim_k Z -1 =d_1-1
$$
and hence this module has rank at most $d_1-1$. This is a contradiction to the
inclusion (\ref{EQ3}). \QED
\end{proof}

\bigskip

\section{Morley forms.}\label{MorleyForms}
\medskip

Adopt Data (\ref{data1}). The multiplication map
$$\mathcal A_i\otimes_S\Sym(I)_{\delta-i}\lto \mathcal A_{\delta}\simeq S(-2)$$
is a perfect pairing of $S$-modules. This perfect pairing is the starting point for the entire theory. It was first established by Jouanolou \cite{jo,jo96}. We found Bus\'e's description \cite{bu} of Jouanolou's work  to be very helpful. Our proof of this perfect pairing is given in Theorem \ref{pp}. We have already highlighted this perfect pairing  in (\ref{J}) and (\ref{sectgoal}). 
The above perfect pairing induces a homogeneous isomorphism of graded $S$-modules
\begin{equation}\tag{4.0.1}\mathcal A_i \longrightarrow {\rm Hom}_S({\rm Sym}(I)_{\delta-i},S(-2)).\end{equation}
There are situations where we are able to identify an explicit basis for ${\rm Hom}_S({\rm Sym}(I)_{\delta-i},S(-2))$; see,  for example,  Lemmas \ref{J08} and \ref{J08'}.   
Our proof of Theorem \ref{pp} is highly non-constructive. 
 On the other hand, Jouanolou's proof is  constructive. He uses Morley forms to exhibit an explicit inverse for the isomorphism (4.0.1). We summarize Jouanolou's theory of Morley forms in the present  section, and then we apply these ideas in Theorem \ref{goal5} to exhibit an explicit generating set for $\mathcal A$ when, in the language of  
  Data \ref{data1},    $2=d_1< d_2$ and $\varphi$ has a generalized zero in the first column.

  Most of the present section is purely expository.  We include this material for the reader's convenience and in order to put the ideas of Morley forms into the ambient notation.
As far as we know, the calculation of ``$q_{\beta,\delta-i-\beta}$'' in part (5) of Observation \ref{extract2'} does not appear elsewhere in the literature; on the other hand, this calculation is straightforward. These ``$q_{\beta,\delta-i-\beta}$'' are the  ingredient from the present section that is used in the proof of Theorem \ref{goal5}; see Corollary \ref{qchi}. We have calculated the ``$q_{\beta,\delta-i-\beta}$'' in greater generality than we use in the present paper because the calculation of these ``$q_{\beta,\delta-i-\beta}$'' is not the obstruction to generalizing Theorem \ref{goal5}; the obstruction is finding the appropriate generalization of Lemmas \ref{J08} and \ref{J08'}.

Begin with Data \ref{data1} and consider the ring  $B\otimes_S B$. The  ideal of the diagonal, $$(x\otimes 1-1\otimes x, y\otimes 1-1\otimes y)\, ,$$ is the kernel of the multiplication map $B\otimes_S B \rightarrow B$.  It is clear that  the elements  $g_j\otimes 1-1\otimes g_j$ of $B\otimes_S B$ belongs to this  ideal for $1\le j\le 2$. Let $H$ be a $2\times 2$ matrix with entries in $B\otimes_S B$ so that  
\begin{equation}\tag{4.0.2}[g_1\otimes 1-1\otimes g_1\, , \ g_2\otimes 1-1\otimes g_2]=
[x\otimes 1-1\otimes x\ , \ y\otimes 1-1\otimes y]\cdot H.\end{equation} Define $\Delta$ to be the image of $\det H$ in $\Sym(I)\otimes_S\Sym(I)$ under the natural map
$$B\otimes_S B\stackrel{\rho\otimes \rho}{\lto} \Sym(I)\otimes_S\Sym(I),$$ where $\rho$ is the natural quotient map
$$B\lto B/(g_1,g_2)=\Sym(I).$$
Notice that $\Delta$ is bi-homogeneous of bi-degree $(\delta,2)$ in $\Sym(I)\otimes_S\Sym(I)$, where $x\otimes 1$, $1\otimes x$, $y\otimes 1$, $1\otimes y$ have bi-degree $(1,0)$ and $T_1$,$T_2$, $T_3$ have bi-degree $(0,1)$. The element $\Delta$ is uniquely determined up to multiplication by a unit in $k$ because $\Delta$ generates the image  in ${\rm Sym}(I) \otimes_S {\rm Sym}(I)$ of the ideal
\[(g_1\otimes 1-1\otimes g_1, \, g_2\otimes 1-1\otimes g_2): (x\otimes 1-1\otimes x\, , \, y\otimes 1-1\otimes y)\subset B\otimes_S B\, .\]

 One may also view $\Sym(I)\otimes_S\Sym(I)$ as having three degrees: $x\otimes 1$, $y\otimes 1$ have tri-degree $(1,0,0)$, $1\otimes x$, $1\otimes y$ have tri-degree $(0,1,0)$, and $T_1$,$T_2$, $T_3$ have tri-degree $(0,0,1)$. 
Write $$\Delta= \sum_{i=0}^\delta \operatorname{morl}_{(i,\delta-i)},$$where each 
$$\operatorname{morl}_{(i,\delta-i)}\in (\Sym(I)\otimes_S\Sym(I))_{(i,\delta-i,2)}.$$
The tri-homogeneous elements \begin{equation*}\{\operatorname{morl}_{(i,\delta-i)}\mid 0\le i\le \delta\}\end{equation*} are the {\it Morley forms associated to the regular sequence $g_1,g_2$ in $B$.}

\begin{observation}\label{morl-ob} Adopt {\rm\ref{data1}} and let ${\rm syl}$ be a fixed generator for the one-dimensional vector space $\mathcal A_{(\delta,2)}$, as described in Remark {\rm \ref{???}}. Then the following statements hold{\rm :}
\begin{itemize}
 \item[{\rm(1)}]$\operatorname{morl}_{(\delta,0)}=\alpha_1\cdot \operatorname{syl}\otimes 1\in \Sym(I)_{\delta}\otimes_S\Sym(I)_{0}$, for some unit $\alpha_1$ in $k$,
 \item[{\rm(2)}] $\operatorname{morl}_{(0,\delta)}=1\otimes \alpha_2\cdot \operatorname{syl}\in \Sym(I)_{0}\otimes_S\Sym(I)_{\delta}$,  for some unit $\alpha_2$ in $k$, 
\item[{\rm(3)}] if $L$ is an element of $B$, then $$(L\otimes 1-1\otimes L)\Delta=0$$   in $\Sym(I)\otimes_S\Sym(I)$, and
\item[{\rm(4)}] if $b$ is an element of the $S$-module  $\Sym(I)_{\ell}$, for some non-negative integer $\ell$,  then $$(b\otimes 1)\operatorname{morl}_{(i,\delta-i)}=(1\otimes b) \operatorname{morl}_{(i+\ell,\delta-i-\ell)}$$
 in $\Sym(I)_{i+\ell}\otimes_S\Sym(I)_{\delta-i}\, $.
\end{itemize}
 \end{observation}

\begin{proof} To prove (1), observe that $\operatorname{morl}_{(\delta,0)}$ is equal to the image of $\Delta$ under the natural ring surjection
$$ \Sym(I)\otimes_S \Sym(I) \lto \Sym(I)\otimes_S \frac{\Sym(I)}{\mathfrak m \Sym(I)}= \Sym(I)\otimes_S S=
\Sym(I).$$ On the other hand, the image of (4.0.2) in $B\otimes_S B/(B\otimes_S\mathfrak m B)=B$ is $[g_1,g_2]=[x,y]\cdot  \bar H$, where $\bar H$ denotes the image of $H$, and after permuting the rows of $\bar H$ this becomes the equation that is used, in Remark \ref{???}, to define $\operatorname{syl}$. 

 The proof of (2) is completely analogous to the proof of (1). One sets $x\otimes 1$ and $y\otimes 1$ to zero instead of setting $1\otimes x$ and $1\otimes y$ to zero.

To show (3), notice that $L\otimes 1-1\otimes L$ belongs to the ideal of the diagonal $(x\otimes 1-1\otimes x,y\otimes 1-1\otimes y)$. Multiply both sides of (4.0.2) on the right by the classical adjoint of $H$ to see that 
the ideal $$(x\otimes 1-1\otimes x,y\otimes 1-1\otimes y)\det H$$  is contained in the ideal $(g_1\otimes 1-1\otimes g_1,g_2\otimes 1-1\otimes g_2)$. The second ideal is sent to zero under the homomorphism $\rho\otimes \rho: B\otimes_S B\lto \Sym(I)\otimes_S \Sym(I)$. 

We now prove part (4). If $\ell=0$, then $b$ is in $S$ and there is nothing to show. If $\ell$ is positive, then assertion (3) guarantees that  $(b\otimes 1-1\otimes b)\Delta=0$. One completes the proof by examining the component of $(b\otimes 1-1\otimes b)\Delta=0$ in $\Sym(I)_{i+\ell}\otimes_S\Sym(I)_{\delta-i}$. \QED \end{proof} 

\bigskip Now that the Morley forms have been defined, we set up the rest of the notation that is used in the statement of Theorem \ref{Morley}, where we establish that the Morley forms provide an inverse to the isomorphism of (4.0.1). 
If $u\in \operatorname{Hom}_S(\Sym(I)_{\delta-i},S)$, then $u$ induces a map
\begin{equation}\label{4.25}\Sym(I)\otimes_S\Sym(I)_{\delta-i}\stackrel{1\otimes u}{\longrightarrow}\Sym(I)\otimes_S S=\Sym(I).\end{equation}
When this map is applied to $\operatorname{morl}_{(i,\delta-i)}\in \Sym(I)_i\otimes_S\Sym(I)_{\delta-i}$, the result is
$$(1\otimes u)(\operatorname{morl}_{(i,\delta-i)})\in \Sym(I)_i\, .$$ It is shown in the proof of Theorem \ref{Morley} that \begin{equation}\label{actually}\text{$(1\otimes u)(\operatorname{morl}_{(i,\delta-i)})$ actually is in $\mathcal A_i\, $.}\end{equation} Once (\ref{actually}) has been established, then it makes sense to define the $S$-module homomorphism \begin{equation}\label{BLOP}\nu_1: \operatorname{Hom}_S(\Sym_{\delta-i}(I),S)\lto \mathcal A_i\end{equation} by
$$\nu_1(u)=(1\otimes u)(\operatorname{morl}_{(i,\delta-i)}).$$

We also  define the $S$-module homomorphism \begin{equation}\label{BLIP}\nu_2: \mathcal A_i\lto \operatorname{Hom}_S(\Sym(I)_{\delta-i},S). \end{equation} If $a$ is in  $\mathcal A_i$, then multiplication by $a$ is an $S$-module homomorphism $$\mu_a: \Sym(I)_{\delta-i}\lto \mathcal A_{\delta}\, .$$ It is well-known that $\mathcal A_{\delta}$ is the free $S$-module generated by any fixed Sylvester element ${\rm syl}$. (Our proof of this statement may be found in Remark \ref{???}.) Let $\mu_{\rm syl}^{-1}:\mathcal A_{\delta}\to S$ be the inverse of  the isomorphism $\mu_{\rm syl}:S\to \mathcal A_{\delta}$. (The notation is consistent because $\Sym(I)_0=S$.) In Remark \ref{???}, the homomorphism $\mu_{\rm syl}^{-1}$ is called $\sigma$.  If $a$ is in $\mathcal A_i$, then we define $\nu_2(a)$ to be the homomorphism in   $\operatorname{Hom}_S(\Sym(I)_{\delta-i},S)$ which is given by
$$\Sym(I)_{\delta-i}\stackrel{\mu_a}{\longrightarrow} \mathcal A_{\delta}\stackrel {\mu_{\rm syl}^{-1}}{\longrightarrow}S.$$

We point out that the homomorphisms $\nu_1$ and $\nu_2$ are not homogeneous; see, however, Remark~\ref{R4.3}.


\medskip \begin{theorem}\label{Morley} {\rm(}{\bf Jouanolou} \cite[\S 3.6]{jo96} and \cite[\S 3.11]{jo}{\rm)} Adopt Data {\rm\ref{data1}} and let $i$ be an integer with $0\le i\le \delta$.
\begin{itemize}
\item[{\rm(1)}] If $u\in \operatorname{Hom}(\Sym(I)_{\delta-i},S)$, then $(1\otimes u)(\operatorname{morl}_{(i,\delta-i)})\in \mathcal A_i\, $.

\item[{\rm(2)}]  The homomorphisms
$$\nu_2: \mathcal A_i\lto \operatorname{Hom}_S(\Sym(I)_{\delta-i},S)\quad\text{and}\quad 
 \nu_1: \operatorname{Hom}_S(\Sym(I)_{\delta-i},S)\lto \mathcal A_i,$$
as described in {\rm(\ref{BLIP})} and {\rm(\ref{BLOP})}, respectively,
are inverses of one another {\rm(}up to multiplication by a unit of $k${\rm)}.
\end{itemize}
\end{theorem}

\begin{proof} Let $\syl\in \mathcal A_{\delta}$ be a fixed Sylvester form and  $\alpha_1$ and $\alpha_2$ be the fixed units in $k$ with 
$$\operatorname{morl}_{(\delta,0)}=\alpha_1\cdot \operatorname{syl}\otimes 1\quad\text{and}
\quad \operatorname{morl}_{(0,\delta)}=1\otimes \alpha_2\cdot \operatorname{syl}\, ,$$ as described in parts (1) and (2) of Observation \ref{morl-ob}. Apply parts (4) and (1) of Observation \ref{morl-ob} to see that
\begin{equation}\label{4.9}(b\otimes 1)\operatorname{morl}_{(i,\delta-i)}=(1\otimes b)\operatorname{morl}_{(\delta,0)}
=\alpha_1 \cdot {\rm syl}\otimes b
\in
\Sym(I)_\delta\otimes_S \Sym(I)_{\delta-i}\,,\end{equation}
for all $b$ in the $S$-module $\Sym(I)_{\delta-i}$.

Let $u$ be an arbitrary element of $\operatorname{Hom}_S(\Sym(I)_{\delta-i},S)$. Apply $$1\otimes u:\Sym(I)\otimes _S \Sym(I)_{\delta-i}\to \Sym(I),$$ as described in (\ref{4.25}), to each side of (\ref{4.9}) to obtain 
$$b\cdot (1\otimes u)(\operatorname{morl}_{(i,\delta-i)})=\alpha_1\cdot \operatorname{syl}\cdot u(b)\in \mathcal A_{\delta}\, .$$
Notice that this holds in particular for all $b$ in $ R_{\delta-i}\subset \Sym(I)_{\delta-i}$. Therefore,
$$(1\otimes u)(\operatorname{morl}_{(i,\delta-i)})\in \mathcal A :_{\Sym(I)}\mathfrak m^{\delta-i}= \mathcal A\, .$$ The last equality holds because $\mathcal A=0:_{\Sym(I)}\mathfrak m^{\infty}$; see Remark \ref{R2.2}. We have established assertion (1). We have also established half of assertion (2) because we have shown that $(\nu_2\circ \nu_1)(u)$ sends the element $b$ of $\Sym(I)_{\delta-i}$ to 
$$\mu_{\rm{syl}}^{-1}(b\cdot \nu_1(u))=\mu_{\rm{syl}}^{-1}(b\cdot (1\otimes u)(\operatorname{morl}_{(i,\delta-i)})) 
=\alpha_1\cdot \mu_{\rm{syl}}^{-1}(\operatorname{syl}\cdot u(b))=\alpha_1\cdot u(b);$$and therefore, $\nu_2\circ \nu_1$ is equal to multiplication by the unit $\alpha_1$  on $\operatorname{Hom}_S(\Sym(I)_{\delta-i},S)$. 

Now we prove the rest of (2). Let $a\in \mathcal A_i$. We compute 
$$\begin{array}{rll}(\nu_1\circ \nu_2)(a)&\hskip-10pt{}=(1\otimes \nu_2(a))(\operatorname{morl}_{(i,\delta-i)})\\
&\hskip-10pt{}{}=(1\otimes \mu_{\rm syl}^{-1}\circ \mu_a)(\operatorname{morl}_{(i,\delta-i)})\\
&\hskip-10pt{}{}=((1\otimes \mu_{\rm syl}^{-1})\circ (1\otimes \mu_a))(\operatorname{morl}_{(i,\delta-i)})\\
&\hskip-10pt{}{}=(1\otimes \mu_{\rm syl}^{-1})( (1\otimes a)\operatorname{morl}_{(i,\delta-i)})\\
&\hskip-10pt{}{}=(1\otimes \mu_{\rm syl}^{-1})( (a\otimes 1)\operatorname{morl}_{(0,\delta)})&\quad \text{by Observation \ref{morl-ob}, part (4)}\\ 
&\hskip-10pt{}{}=(1\otimes \mu_{\rm syl}^{-1})(a\otimes \alpha_2\cdot \syl)=\alpha_2\cdot a&\quad \text{by Observation \ref{morl-ob}, part (2).}\end{array}$$
\QED
\end{proof}

\begin{remark}\label{R4.3} If one replaces $S$ by $S(-2)$ in part (2) of Theorem~\ref{Morley}, then the isomorphism of $S$-modules $\nu_2$ becomes homogeneous and hence so does $\nu_1$. 
\end{remark}

\begin{remark}\label{qs} Adopt Data \ref{data1} and let $i$ be an integer with $0\le i\le \delta$. As an $S$-module $\Sym(I)_{\delta-i}$ is minimally generated by the monomials $x^{\beta}y^{\delta-i-\beta}$, with $0\le \beta\le \delta-i$. Thus, there are elements $q_{\beta,\delta-i-\beta}$ in $\Sym(I)_i$, with $0\le \beta\le \delta-i$, so that 
$$\mathrm {morl}_{(i,\delta-i)} =\sum_{\beta=0}^{\delta-i} q_{\beta,\delta-i-\beta}\otimes x^\beta y^{\delta-i-\beta}\quad\text{in}\ \Sym(I)_i\otimes_S\Sym(I)_{\delta-i}.$$ Furthermore, if $u\in \mathrm {Hom}_S(\Sym(I)_{\delta-i},S)$, then 
$$\nu_1(u)=(1\otimes u)(\mathrm {morl}_{(i,\delta-i)})=\sum_{\beta=0}^{\delta-i}q_{\beta,\delta-i-\beta}\cdot u(x^{\beta}y^{\delta-i-\beta}).$$\end{remark}

\bigskip
\begin{corollary}\label{qchi} Retain the notation and hypotheses of Remark {\rm\ref{qs}} with $d_1-1\le i\le d_2-2$. Recall the exact sequence 
\begin{equation}\label{rcl}0\to \mathrm{Hom}_S(\Sym(I)_{\delta-i},S)\to \xymatrix{S^{\delta-i+1} \ar[rr] ^{{\Upsilon_{d_2-i-1,1}}^{\text{\rm T}}} &&  S(1)^{d_2-i-1}}\end{equation} from part {\rm(2)} of Theorem {\rm\ref{chart}}. If $\chi$ is an element of $S^{\delta-i+1}$ with 
$\Upsilon_{d_2-i-1,1}^{\text{\rm T}}\cdot \chi=0$, then $\chi$ represents an element $u_{\chi}$ of $\mathrm{Hom}_S(\Sym(I)_{\delta-i},S)$ and 
$$\nu_1(u_{\chi})=\left[\begin{matrix} q_{0,\delta-i}&\ldots& q_{\delta-i,0}\end{matrix}\right]\cdot \chi.$$
\end{corollary}
\bigskip
\begin{proof} From part (2) of Theorem \ref{chart} we have an exact sequence 
$$  0\to\xymatrix{S^{d_2-i-1}(-1) \ar[rr]
^{\phantom{XXX}\Upsilon_{d_2-i-1,1}\phantom{XXX}}&& S^{\delta-i+1}}\longrightarrow {\rm Sym}(I)_{\delta-i}\to 0,$$
where $S^{\delta-i+1}\simeq B_{\delta-i}:=B_{(\delta-i,\underline{\phantom{x.}}\,)}$ is considered with the $S$-basis $y^{\delta-i},\cdots,x^{\delta-i}$. Apply $\operatorname{Hom}_S(\underline{\phantom{x.}},S)$ to this sequence and use Remark \ref{qs}.
\QED\end{proof}

\bigskip

An explicit formula for each $q_{\beta,\delta-i-\beta}$ is given in part (5) of the following Observation.
\begin{observation}\label{extract2'} Adopt Data {\rm\ref{data1}}. Statements {\rm(1)} -- {\rm(5)} below hold{\rm;}    {\rm(1)} -- {\rm(4)} take place in $B\otimes_SB$ and {\rm(5)} takes place  in $\Sym(I)\otimes_S\Sym(I)$.

\noindent {\rm(1)} If $a$ and $b$ are non-negative integers, then 
$$x^ay^b\otimes 1-1\otimes x^ay^b=
(x\otimes 1-1\otimes x)\sum_{\beta=0}^{a-1}x^{a-1-\beta}\otimes x^\beta y^b+
(y\otimes 1-1\otimes y)\sum_{\gamma =0}^{b-1}x^ay^{b-1-\gamma }\otimes y^\gamma .$$

\noindent {\rm(2)} If $g=\sum\limits_{\ell=0}^{d}c_{\ell}x^\ell y^{d-\ell}$ is an element of $B$, with each $c_{\ell}$ in $S$, then
$g\otimes 1-1\otimes g$ is equal to $$(x\otimes 1-1\otimes x)\left( \sum\limits_{\ell=0}^{d}\sum\limits_{\beta=0}^{\ell -1}c_{\ell}x^{\ell -1-\beta}\otimes x^\beta y^{d-\ell}\right)+ (y\otimes 1-1\otimes y)\left(\sum\limits_{\lambda=0}^{d}\sum\limits_{\gamma =0}^{{d-\lambda}-1}c_{\lambda}x^\lambda y^{{d-\lambda}-1-\gamma }\otimes y^\gamma \right).$$

\noindent {\rm(3)} If $g_1=\sum\limits_{\ell=0}^{d_1}c_{\ell,1}x^\ell y^{d_1-\ell}$ and $g_2=\sum\limits_{\ell=0}^{d_2}c_{\ell,2}x^{\ell} y^{d_2-\ell}$, then
$$\left[\begin{array}{ll} g_1\otimes 1-1\otimes g_1&g_2\otimes 1-1\otimes g_2\end{array}\right]=
\left[\begin{array}{ll}x\otimes 1-1\otimes x&y\otimes 1-1\otimes y\end{array}\right]H,$$ for
$$H= \left[\begin{array}{ll}
\sum\limits_{\ell=0}^{d_1}\ \sum\limits_{\beta=0}^{\ell -1}c_{\ell,1}x^{\ell -1-\beta}\otimes x^\beta y^{d_1-\ell}&
\sum\limits_{\ell=0}^{d_2}\ \sum\limits_{\beta=0}^{\ell -1}c_{\ell,2}x^{\ell -1-\beta}\otimes x^\beta y^{d_2-\ell}\\
\sum\limits_{\lambda=0}^{d_1}\ \sum\limits_{\gamma =0}^{{d_1-\lambda}-1}c_{\lambda,1}x^\lambda y^{{d_1-\lambda}-1-\gamma }\otimes y^\gamma &
\sum\limits_{\lambda=0}^{d_2}\ \sum\limits_{\gamma =0}^{{d_2-\lambda}-1}c_{\lambda,2}x^\lambda y^{{d_2-\lambda}-1-\gamma }\otimes y^\gamma 
 \end{array}\right].$$

\noindent {\rm(4)} If $H$ is the matrix of {\rm(3)}, then the determinant of $H$ is equal to
$$\sum\limits_{i=0}^{\delta}\ \sum\limits_{\beta=0}^{\delta -i}
\ \sum\limits_{w=0}^i  \left[\sum\limits_{(\ell,m)\in \mathfrak S_1}\  
c_{\ell,1}c_{m,2}
-   \sum\limits_{(\ell,m)\in \mathfrak S_2}
c_{m,1}c_{\ell,2} \right]   
x^{w} y^{i-w}\otimes x^{\beta}y^{\delta-i-\beta},$$ where $\mathfrak S_1$ and $\mathfrak S_2$ are the following sets of pairs of non-negative integers{\rm:}
\begin{equation}\label{sets}\mathfrak S_1=\left\{(\ell,m)\left\vert 
\begin{array}{l}\ssize \ell+m=w+1+\beta,\\\ssize \beta+1\le \ell\le d_1, \ \mathrm{and}\\\ssize 0\le m\le d_2-i-1+w\end{array}
\right.\right\}\quad\text{and}\quad
\mathfrak S_2=
\left\{(\ell,m)\left\vert 
\begin{array}{l}\ssize \ell+m=w+1+\beta,\\\ssize \beta+1\le \ell\le d_2, \ \mathrm{and}\\\ssize 0\le m\le d_1-i-1+w\end{array}
\right.\right\}.\end{equation}

\noindent {\rm(5)} In the language of Remark {\rm\ref{qs}}, once $i$ is fixed with $0\le i\le \delta$, then the Morley form $\mathrm{morl}_{(i,\delta-i)}$ is equal to $\sum\limits_{\beta=0}^{\delta-i}q_{\beta,\delta-i-\beta}\otimes x^\beta y^{\delta-i-\beta}$ in $\Sym(I)_i\otimes_S\Sym(I)_{\delta-i}$ with
\begin{equation}\label{zzz}q_{\beta,\delta-i-\beta}=\sum\limits_{w=0}^i \left[\sum\limits_{(\ell,m)\in \mathfrak S_1}\  
c_{\ell,1}c_{m,2}
-   \sum\limits_{(\ell,m)\in \mathfrak S_2}
c_{m,1}c_{\ell,2} \right]   
x^{w} y^{i-w}\in \Sym(I)_{i}\, ,\end{equation} where $\mathfrak S_1$ and $\mathfrak S_2$ are the sets of {\rm(\ref{sets})}.
\end{observation}
 
\begin{proof} Assertions (1) -- (3) are straightforward calculations. We prove (4). Once (4) is established, then (5) follows immediately because the image of $\det H$ in $\Sym(I)\otimes_S\Sym(I)$ is equal to
$$\sum_{i=0}^\delta \mathrm {morl}_{(i,\delta-i)}=\sum_{i=0}^\delta\sum_{\beta=0}^{\delta-i}q_{\beta,\delta-i-\beta}\otimes x^\beta y^{\delta-i-\beta}.$$

We calculate $\det H=A-B$ with 
$$\begin{array}{rcl}
A&=&\left(\sum\limits_{\ell=0}^{d_1}\ \sum\limits_{\beta=0}^{\ell -1}c_{\ell,1}x^{\ell -1-\beta}\otimes x^\beta y^{d_1-\ell}\right)
\left(\sum\limits_{\lambda=0}^{d_2}\ \sum\limits_{\gamma =0}^{{d_2-\lambda}-1}c_{\lambda,2}x^\lambda y^{{d_2-\lambda}-1-\gamma }\otimes y^\gamma \right)
 \\
B&=& \left(\sum\limits_{\lambda=0}^{d_1}\ \sum\limits_{\gamma =0}^{{d_1-\lambda}-1}c_{\lambda,1}x^\lambda y^{{d_1-\lambda}-1-\gamma }\otimes y^\gamma \right)
\left(\sum\limits_{\ell=0}^{d_2}\ \sum\limits_{\beta=0}^{\ell -1}c_{\ell,2}x^{\ell -1-\beta}\otimes x^\beta y^{d_2-\ell}\right).\end{array}$$
We put the summation signs that  involve $\beta$ on the left in order to see that
$$\begin{array}{rcl}
A&=&\sum\limits_{\beta=0}^{d_1 -1}\ \sum\limits_{\ell=\beta+1}^{d_1}\ \sum\limits_{\lambda=0}^{d_2}\ 
\sum\limits_{\gamma =0}^{{d_2-\lambda}-1}\ c_{\ell,1}c_{\lambda,2}x^{\ell -1-\beta+\lambda}y^{{d_2-\lambda}-1-\gamma }\otimes x^\beta y^{d_1-\ell+\gamma } 
 \\
B&=& \sum\limits_{\beta=0}^{d_2 -1}\ \sum\limits_{\ell=\beta+1}^{d_2}\ \sum\limits_{\lambda=0}^{d_1}\ \sum\limits_{\gamma =0}^{{d_1-\lambda}-1}\ 
c_{\lambda,1}c_{\ell,2} 
 x^{\ell -1-\beta+\lambda}y^{{d_1-\lambda}-1-\gamma }\otimes x^\beta y^{d_2-\ell+\gamma } .\end{array}$$
Replace $\gamma $ with $d_2-2-i-\beta+\ell$ in $A$ and with $d_1-2-i-\beta+\ell$ in $B$ to obtain

$$\begin{array}{rcl}
A&=&\sum\limits_{\beta=0}^{d_1 -1}\ \sum\limits_{\ell=\beta+1}^{d_1}\ \sum\limits_{\lambda=0}^{d_2}\ 
\sum\limits_{i=\ell-1-\beta+\lambda}^{d_2-2-\beta+\ell}
\ c_{\ell,1}c_{\lambda,2}x^{\ell -1-\beta+\lambda}y^{\beta+i+1-\ell-\lambda}\otimes x^\beta y^{\delta-i-\beta} 
 \\
B&=& \sum\limits_{\beta=0}^{d_2 -1}\ \sum\limits_{\ell=\beta+1}^{d_2}\ \sum\limits_{\lambda=0}^{d_1}\ \sum\limits_{i=
\ell-1-\beta+\lambda}^{d_1-2-\beta+\ell}\ 
c_{\lambda,1}c_{\ell,2} 
 x^{\ell -1-\beta+\lambda}y^{\beta+i+1-\ell-\lambda}\otimes x^\beta y^{\delta-i-\beta} .\end{array}$$
We re-arrange the order of summation by putting the sum involving $i$ first. We see that $i$ satisfies:
$$0\le \ell-1-\beta+\lambda\le i\le d_2-2-\beta+\ell\le \delta-\beta\le \delta$$ in $A$. The analogous inequalities hold in $B$; in particular,  $i$ also satisfies $0\le i\le \delta$.   The old constraints on $i$ now become constraints on $\lambda$ and  $\ell$. It follows that
$$\begin{array}{rcl}
A&=&\sum\limits_{i=0}^\delta\ \sum\limits_{\beta=0}^{d_1 -1}\ \sum\limits_{\ell=\max\{\beta+1,i-d_2+2+\beta\}}^{d_1}\ \sum\limits_{\lambda=0}^{\min\{d_2,\beta+i+1-\ell\}}\ 
c_{\ell,1}c_{\lambda,2}x^{\ell -1-\beta+\lambda}y^{\beta+i+1-\ell-\lambda}\otimes x^\beta y^{\delta-i-\beta} 
 \\
B&=&\sum\limits_{i=0}^\delta\ \sum\limits_{\beta=0}^{d_2 -1}\ \sum\limits_{\ell=\max\{\beta+1,i-d_1+2+\beta\}}^{d_2}\ \sum\limits_{\lambda=0}^{\min\{d_1,\beta+i+1-\ell\}}\ 
c_{\lambda,1}c_{\ell,2} 
 x^{\ell -1-\beta+\lambda}y^{\beta+i+1-\ell-\lambda}\otimes x^\beta y^{\delta-i-\beta} .\end{array}$$

  In $A$, the third summation sign represents the empty sum unless $\beta\le d_1-1$ and $\beta\le \delta-i$. Thus, the following four choices for the second summation sign all yield the same value for $A$:
$$\sum_{0\le \beta},\quad \sum_{\beta=0}^{d_1-1},\quad \sum_{\beta=0}^{\delta-i},\quad\text{or}\quad \sum_{\beta=0}^{\min\{d_1-1,\delta-i\}}.$$An analogous statement holds for $B$. We conclude that

 $$\begin{array}{rcl}
A&=&\sum\limits_{i=0}^\delta\ \sum\limits_{\beta=0}^{\delta-i}\ \sum\limits_{\ell=\max\{\beta+1,i-d_2+2+\beta\}}^{d_1}\ \sum\limits_{\lambda=0}^{\min\{d_2,\beta+i+1-\ell\}}\ 
c_{\ell,1}c_{\lambda,2}x^{\ell -1-\beta+\lambda}y^{\beta+i+1-\ell-\lambda}\otimes x^\beta y^{\delta-i-\beta} 
 \\
B&=&\sum\limits_{i=0}^\delta\ \sum\limits_{\beta=0}^{\delta-i}\ \sum\limits_{\ell=\max\{\beta+1,i-d_1+2+\beta\}}^{d_2}\ \sum\limits_{\lambda=0}^{\min\{d_1,\beta+i+1-\ell\}}\ 
c_{\lambda,1}c_{\ell,2} 
 x^{\ell -1-\beta+\lambda}y^{\beta+i+1-\ell-\lambda}\otimes x^\beta y^{\delta-i-\beta}.\end{array}$$ Replace $\lambda$ with 
$w-\ell+1+\beta$ 
to obtain $$\begin{array}{rcl}
A&=&\sum\limits_{i=0}^\delta\ \sum\limits_{\beta=0}^{\delta-i}\ \sum\limits_{\ell=\max\{\beta+1,i-d_2+2+\beta\}}^{d_1}\ \sum\limits_{w=\ell-1-\beta}^{\min\{i,d_2+\ell-1-\beta\}}\ 
c_{\ell,1}c_{w-\ell+1+\beta,2}x^{w}y^{i-w}\otimes x^\beta y^{\delta-i-\beta} 
 \\
B&=&\sum\limits_{i=0}^\delta\ \sum\limits_{\beta=0}^{\delta-i}\ \sum\limits_{\ell=\max\{\beta+1,i-d_1+2+\beta\}}^{d_2}\ \sum\limits_{w=\ell-1-\beta}^{\min\{i,d_1+\ell-1-\beta\}}\ 
c_{w-\ell+1+\beta,1}c_{\ell,2} 
 x^{w}y^{i-w}\otimes x^\beta y^{\delta-i-\beta} .\end{array}$$
Exchange the third and fourth summation signs. Keep in mind that $0\le \ell-1-\beta\le w\le i$. We see that 
$$\begin{array}{rcl}
A&=&\sum\limits_{i=0}^\delta\ \sum\limits_{\beta=0}^{\delta-i}\ \sum\limits_{w=0}^{i}\ \sum\limits_{\ell=\max\{\beta+1,i-d_2+2+\beta,w+1+\beta-d_2\}}^{\min\{d_1,w+1+\beta\}}\ 
c_{\ell,1}c_{w-\ell+1+\beta,2}x^{w}y^{i-w}\otimes x^\beta y^{\delta-i-\beta} 
 \\
B&=&\sum\limits_{i=0}^\delta\ \sum\limits_{\beta=0}^{\delta-i}\ \sum\limits_{w=0}^i\ \sum\limits_{\ell=\max\{\beta+1,i-d_1+2+\beta,w+1+\beta-d_1\}}^{\min\{d_2,w+1+\beta\}}\ 
c_{w-\ell+1+\beta,1}c_{\ell,2} 
 x^{w}y^{i-w}\otimes x^\beta y^{\delta-i-\beta}  .\end{array}$$
The parameter $w$ satisfies $0\le w\le i$; hence, $w+1+\beta-d_s\le i -d_s+2+\beta$, for $s$ equal to $1$ or $2$, and 
$$\begin{array}{rcl}
A&=&\sum\limits_{i=0}^\delta\ \sum\limits_{\beta=0}^{\delta-i}\ \sum\limits_{w=0}^{i}\ \sum\limits_{\ell=\max\{\beta+1,i-d_2+2+\beta\}}^{\min\{d_1,w+1+\beta\}}\ 
c_{\ell,1}c_{w-\ell+1+\beta,2}x^{w}y^{i-w}\otimes x^\beta y^{\delta-i-\beta} 
 \\
B&=&\sum\limits_{i=0}^\delta\ \sum\limits_{\beta=0}^{\delta-i}\ \sum\limits_{w=0}^i\ \sum\limits_{\ell=\max\{\beta+1,i-d_1+2+\beta\}}^{\min\{d_2,w+1+\beta\}}\ 
c_{w-\ell+1+\beta,1}c_{\ell,2} 
 x^{w}y^{i-w}\otimes x^\beta y^{\delta-i-\beta}  .\end{array}$$
Let $m=w-\ell+1+\beta$. The four constraints
$$\beta+1\le \ell,\quad i-d_2+2+\beta\le \ell,\quad  \ell\le d_1,\quad \ell\le w+1+\beta$$ on $\ell$ in $A$ are equivalent to
$$\beta+1\le \ell,\quad m\le d_2-i-1+w, \quad \ell\le d_1,\quad 0\le m\, ,$$ respectively; and
the four constraints
$$\beta+1\le \ell,\quad i-d_1+2+\beta\le \ell, \quad \ell\le d_2,\quad \ell\le w+1+\beta$$
on $\ell$ in $B$ are equivalent to
$$\beta+1\le \ell,\quad m\le d_1-i-1+w,\quad  \ell\le d_2,\quad 0\le m\, ,$$respectively.
It follows that 
$$A=\sum\limits_{i=0}^\delta\ \sum\limits_{\beta=0}^{\delta-i}\ \sum\limits_{w=0}^{i}\ \sum\limits_{(\ell,m)\in \mathfrak S_1}\ 
c_{\ell,1}c_{m,2}x^{w}y^{i-w}\otimes x^\beta y^{\delta-i-\beta}$$$$ B=\sum\limits_{i=0}^\delta\ \sum\limits_{\beta=0}^{\delta-i}\ \sum\limits_{w=0}^i\ \sum\limits_{(\ell,m)\in \mathfrak S_2}\ 
c_{m,1}c_{\ell,2} 
 x^{w}y^{i-w}\otimes x^\beta y^{\delta-i-\beta}, $$as desired. \QED
\end{proof}

 \bigskip 

The answer of part (5) of Observation \ref{extract2'} can be simplified when $d_1=2$ and $1\le i\le d_2-1$. This hypothesis is in effect when we apply Observation \ref{extract2'} in the proof of Theorem \ref{goal5}. Once $d_1=2$, then $\delta$ is equal to $d_2$ and we use $d_2$ in place of $\delta$ in our simplification. If $P$ is a statement, then define
\begin{equation}\label{chi}{\underline{\chi}}(P)=\begin{cases} 1&\text{if $P$ is true}\\0&\text{if $P$ is false}.\end{cases}\end{equation}

\begin{corollary}\label{extract2*} If $d_1=2$, $1\le i\le d_2-1$, and $0\le \beta\le d_2-i$, then the element 
$q_{\beta,d_2-i-\beta}$ of {\rm (\ref{zzz})} is equal to 
$$ q_{\beta,d_2-i-\beta}=\begin{cases}{\underline{\chi}}(\beta=0)\left(\sum\limits_{w=0}^i  c_{1,1}c_{w,2}x^{w} y^{i-w} 
+ \sum\limits_{w=1}^i  
 c_{2,1}c_{w-1,2}
x^{w} y^{i-w}\right) 
+{\underline{\chi}}(\beta=1)\sum\limits_{w=0}^i    
c_{2,1}c_{w,2}
x^{w} y^{i-w}\\
-{\underline{\chi}}(\beta\le d_2-i-1)c_{0,1}c_{i+1+\beta,2}x^{i}-c_{1,1}c_{i+\beta,2}x^{i}-c_{0,1}c_{i+\beta,2}x^{i-1} y\, .\end{cases}$$\end{corollary}

\begin{proof}The parameter $d_1$ is equal to $2$; so $\delta=d_2$.  Write $q_{\beta,\delta-i-\beta}$, from (\ref{zzz}), as $A+B$, 
where $$A=\sum\limits_{w=0}^i \sum\limits_{(\ell,m)\in \mathfrak S_1}\  
c_{\ell,1}c_{m,2}x^{w} y^{i-w}\quad\text{and}\quad B=-\sum\limits_{w=0}^i 
  \sum\limits_{(\ell,m)\in \mathfrak S_2}
c_{m,1}c_{\ell,2}   
x^{w} y^{i-w}.$$
Since  $d_1=2$, the constraint $\beta+1\le \ell\le d_1$ in the definition of $\mathfrak S_1$ allows only three possible values for the pair $(\beta,\ell)$; namely, $(\beta,\ell)$ is equal to $(0,1)$, or $(0,2)$, or $(1,2)$. For each of these pairs, one sets 
$m=w+1+\beta-\ell$ and then one verifies that $0\le m\le d_2-i-1+w$ becomes $1\le w$ when $(\beta,\ell)=(0,2)$ and automatically holds otherwise. It follows that 
$$A= {\underline{\chi}}(\beta=0)\left(\sum\limits_{w=0}^i  c_{1,1}c_{w,2}x^{w} y^{i-w} 
+ \sum\limits_{w=1}^i  
 c_{2,1}c_{w-1,2}
x^{w} y^{i-w}\right)\\
+{\underline{\chi}}(\beta=1)\sum\limits_{w=0}^i    
c_{2,1}c_{w,2}
x^{w} y^{i-w}.$$
Now we simplify $B$. The constraint $0\le m\le d_1-i-1+w$ in the definition of $\mathfrak S_2$ becomes $${0\le m\le w+1-i},$$ when $d_1=2$. On the other hand, the parameter $w$ in $B$ is always at most $i$. Thus, the pair $(w,m)$ must satisfy $0\le m\le w+1-i$ and $w\le i\le w+1$. It follows that there are only three possible values for the pair $(w,m)$, namely, $(w,m)$ is equal to $(i,0)$ or $(i,1)$, or $(i-1,0)$. Use $\ell+m=w+1+\beta$ to define $\ell$. Verify that $\beta+1\le \ell\le d_2$ becomes $\beta\le d_2-i-1$ when $(w,m)=(i,0)$ and holds automatically otherwise. It follows that 
$$B=-{\underline{\chi}}(\beta\le d_2-i-1)c_{0,1}c_{i+1+\beta,2}x^{i}-c_{1,1}c_{i+\beta,2}x^{i}-c_{0,1}c_{i+\beta,2}x^{i-1} y\, .$$ \QED
\end{proof}

\section{Explicit generators for $\mathcal A$ when $d_1=2$.}\label{2=d1}

Adopt Data (\ref{data1}) with  $d_1=2$. If $d_2=2$, then the generators of $\mathcal A_{\ge 1}$ are explicitly described in Corollary \ref{XXX}. If   $\varphi$ does
not have a generalized zero in the first column, then Bus\'e \cite[Prop.~3.2]{bu} gave explicit formulas for the generators of 
 $\mathcal A_{\ge 1}$.  The present  section is concerned with the following situation.
\begin{data}\label{data4}
Adopt Data (\ref{data1}) with  $2=d_1< d_2$. Assume also that $\varphi$ has a generalized zero in the first column.\end{data}

Let $\mathcal C$ be the curve of Remark \ref{curve}. We recall that if the parameterization $\mathbb P^1_k\to \mathcal C$ is birational, then the hypothesis that $\varphi$ has a generalized zero in the first column is equivalent to the statement that there is a singularity of multiplicity $d_2$ on $\mathcal C$. 

In this section we assume that Data \ref{data4} is in effect and we describe explicitly {\bf all} of the defining
equations of the Rees algebra $\mathcal R$. Of course, the results of Section \ref{3} apply in the present section; so we know the degrees of the generators of $\mathcal A_{\ge d_1-1}=\mathcal A_{\ge 1}$, a priori, from Table \ref{tbl1}. Indeed, in the context of the present section, Table \ref{tbl1} is given in Tables \ref{tbl2} and \ref{tbl3}. There are two ways in which the present tables are simpler than the general table. First of all, the description of the generator degrees depends on the remainder of a division by $d_1$. When $d_1$ is $2$, there are only two possible remainders: $0$ or $1$. Secondly, Table \ref{tbl1}, together with Corollary \ref{B-mod}, describes the degrees of  the generators of $\mathcal A_{\ge d_1-1}$ as  an $S$-module and a $B$-module. 
Furthermore, the $S$-module $\mathcal A_{\ge d_1-1}$ is free according to part (1) of Corollary~\ref{2.12}. When $d_1=2$, then $\mathcal A_{\ge d_1-1}$ is, in fact, equal to all of $\mathcal A_{\ge 1}$. At any rate, in the present section we give much more than the degrees of the generators. We give explicit formulas for the minimal generators of $\mathcal A_{\ge 1}$. 


\begin{table}\begin{center}
\begin{tabular}{|c||c|c|c|c|c|c|c|c|c|c|c|}\hline
$T$-deg&&&&&&&&&&&\\\hline
$d$&$1^*$&&&&&&&&&&\\\hline
\vdots& &&&&&&&&&&\\\hline
$\lceil \frac d{2}\rceil$&&$1^*$&&&&&&&&&\\\hline
$  \lfloor \frac d{2}\rfloor $&&$1^*$&$2$&$1$&&& &&&&\\\hline
$   \lfloor \frac d{2}\rfloor-1 $&& &&$1^*$&$\cdots$&& &&&&\\\hline
\vdots& &&&&&&&&&&\\\hline
$4$&& &&& $\cdots$&$1$&&&&&\\\hline
$3$&& &&&&$1^*$& $2$&$1$&&&\\\hline
$2$&& &&&&&&$1^*$&$2$&$1$&\\
\hline\hline
&$ 0 $&$ 1 $&$ 2$&$3$&$\cdots$&$d_2-4$&$d_2-3$&$d_2-2$&$d_2-1$&$d_2$&$xy$-deg\\\hline
\end{tabular}
\medskip\caption{{\bf The generator degrees for the free $S$-module $\mathcal A$, in the presence of Data \ref{data4},  when $d$ is odd.} {\rm(}The elements that correspond to the generator degrees marked by $*$ are minimal generators for the $\Sym(I)$-ideal $\mathcal A$.{\rm)}}\label{tbl2}
\end{center}\end{table}

\begin{table}
\begin{center}
\begin{tabular}{|c||c|c|c|c|c|c|c|c|c|c|c|}\hline
$T$-deg&&&&&&&&&&&\\\hline
$d$&$1^*$&&&&&&&&&&\\\hline
\vdots& &&&&&&&&&&\\\hline
$\frac d{2}$&&$2^*$&1&&&&&&&&\\\hline
$   \frac d{2}-1 $&& &$1^*$&$2$&$\cdots$&& &&&&\\\hline
\vdots& &&&&&&&&&&\\\hline
$4$&& &&& $\cdots$&$1$&&&&&\\\hline
$3$&& &&&&$1^*$& $2$&$1$&&&\\\hline
$2$&& &&&&&&$1^*$&$2$&$1$&\\
\hline\hline
&$ 0 $&$ 1 $&$ 2$&$3$&$\cdots$&$d_2-4$&$d_2-3$&$d_2-2$&$d_2-1$&$d_2$&$xy$-deg\\\hline
\end{tabular}

\medskip
\caption{{\bf The generator degrees for the free $S$-module $\mathcal A$, in the presence of Data \ref{data4},  when $d$ is even and and the morphism $\mathbb P^1_k\to \mathcal C$ of Remark \ref{curve}  is birational.} {\rm(}The elements that correspond to the generator degrees marked by $*$ are minimal generators for the $\Sym(I)$-ideal $\mathcal A$.{\rm)}}\label{tbl3}

\end{center}
\end{table}
\medskip
\begin{remark}\label{R5.2}We notice, with significant interest, how similar Tables \ref{tbl2} and \ref{tbl3} are to the degree tables of \cite[5.6]{BD}. It appears that the tables of \cite{BD} are the transpose of the tables given here. This observation is particularly striking because there is virtually no overlap between the data of \cite{BD} and the data used here. In the present section, $d_1=2$ and $d_2=d-2$; so the two parameters $d_1$ and $d_2$ are almost as far apart as possible. On the other hand, in \cite{BD}, the parameters are as close as possible: $d_1=\lfloor \frac d2\rfloor$ and $d_2=\lceil \frac d2\rceil$. 
\end{remark}

We recall that the $S$-module $\mathcal A_0$ is free of rank $1$ and is generated in degree $(0,\deg \mathcal C)$ by the implicit equation $F(T_1,T_2,T_3)$ of the curve $\mathcal C$ of Remarks \ref{curve}. Furthermore, $F$ is an $r^{\text{th}}$ root of the resultant of $g_1$ and $g_2$, where $r=1$ if the parameterization $\mathbb P^1_k\to \mathcal C$ of Remark \ref{curve} is birational, and $r=2$ otherwise. Degree considerations show that, when the parameterization $\mathbb P^1_k\to \mathcal C$ is birational, then $F$ together with the minimal generators of the $\Sym(I)$-ideal $\mathcal A_{\ge 1}$, form a minimal generating set for the $\Sym(I)$-ideal $\mathcal A$. The parameterization $\mathbb P^1_k\to \mathcal C$ is guaranteed to be birational when $d_2$ is odd; see \cite[4.6(1)]{KPU-B} or \cite[0.10]{CKPU}.
If the parameterization $\mathbb P^1_k\to \mathcal C$ is not birational, then one can reparameterize in order to obtain a birational parameterization. The column degrees of the Hilbert-Burch matrix $\varphi'$, which corresponds to the new parameterization, are $d_1'=1$ and $d_2'=\frac{d_2}2$. The matrix $\varphi'$ is ``almost linear''. The defining equations of the Rees algebra associated to $\varphi'$ are recorded explicitly in \cite[Sect.~3]{KPU1}; see also \cite[2.3]{CHW}. Thus we may assume in any case that the parametrization is birational or, equivalently, that the curve $\mathcal C$ has degree $d$. 

We first show how to modify the arbitrary Data \ref{data4} into data in a canonical form.

\begin{observation}\label{modify} If Data {\rm\ref{data4}} is adopted  with $k$ a field which is  closed under taking square roots, then one may 
assume that the first column of $\varphi$ is either  $[x^2+y^2,xy,0]^{\rm T}$ or $[y^2,x^2,0]^{\rm T}$. \end{observation}
\begin{proof}Let $\varphi_1$ represent the first column of $\varphi$. We may apply invertible row operations to $\varphi$ and linear changes of variables to $R=k[x,y]$ without changing the ideal $I$, the symmetric algebra $\Sym(I)$,  the Rees algebra $\mathcal R$, or any other essential feature of Data \ref{data4}. The hypothesis about the generalized zero allows us to use invertible row operations to put a zero into the bottom position of $\varphi_1$. The hypothesis about the height of $I_2(\varphi)$ implies that the two remaining entries of $\varphi_1$ are non-zero. 
They each factor into a product of linear forms. 
If both of these entries are perfect squares, then, after  a linear change of variables,  $\varphi_1=[y^2,x^2,0]^{\rm T}$.
Otherwise, $\varphi_1$ can be put in the form 
$\varphi_1=[\alpha_1 x^2+\alpha_2 y^2,xy,0]^{\rm T}$, for some constants $\alpha_1$ and $\alpha_2$ in $k$. The hypothesis about the height of $I_2(\varphi)$ ensures that both $\alpha$'s are units in $k$. Another linear change of variables yields the result. 
\QED\end{proof}

\begin{corollary}\label{A-ell}If Data {\rm\ref{data4}} is adopted and $i$ is an integer with $1\le i\le d_2-2$, 
then the matrix
$\Upsilon_{d_2-i-1,1}^{\text{\rm T}}$ of Corollary {\rm\ref{qchi}} is the $d_2-i-1\times d_2-i+1$ matrix 
\begin{equation}\label{first-A}\left[\begin{matrix}
T_1&T_2&T_1&0&\dots&0\\
0&\ddots&\ddots&\ddots&\ddots&\vdots\\
\vdots&\ddots&\ddots&\ddots&\ddots&0\\
0&\dots&0&T_1&T_2&T_1\end{matrix}\right]\text{\qquad  if  the first column of $\varphi$ is $[x^2+y^2,xy,0]^{\mathrm{T}}$ }\end{equation}
\begin{equation}\label{second-A}\left[\begin{matrix}
T_1&0&T_2&0&\dots&0\\
0&\ddots&\ddots&\ddots&\ddots&\vdots\\
\vdots&\ddots&\ddots&\ddots&\ddots&0\\
0&\dots&0&T_1&0&T_2\end{matrix}\right]\text{\qquad  if  the first column of $\varphi$ is $[y^2,x^2,0]^{\mathrm{T}}.\phantom{+xy,}$}\end{equation}

\end{corollary}
\begin{proof}According to Definition \ref{upsilon}, $\Upsilon_{d_2-i-1,1}^{\text{\rm T}}$ is the  $d_2-i-1\times d_2-i+1$ matrix 
$$\left[\begin{matrix}
c_{0,1}&c_{1,1}&c_{2,1}&0&\dots&0\\
0&\ddots&\ddots&\ddots&\ddots&\vdots\\
\vdots&\ddots&\ddots&\ddots&\ddots&0\\
0&\dots&0&c_{0,1}&c_{1,1}&c_{2,1}\end{matrix}\right],$$where $g_1=c_{0,1}y^2+c_{1,1}xy+c_{2,1}x^2$. On the other hand,
$$g_1= [T_1,T_2,T_3]\left[\begin{matrix}x^2+y^2\\xy\\0\end{matrix}\right]=T_1(x^2+y^2)+T_2xy\quad\text{ or }\quad g_1= [T_1,T_2,T_3]\left[\begin{matrix}y^2\\x^2\\0\end{matrix}\right]=T_1y^2+T_2x^2.$$
\QED \end{proof}

\bigskip 
Our intention is to apply the technique of Corollary \ref{qchi} when the hypotheses of Observation \ref{modify} are in effect. For that reason, we next find the relations on matrices like those of (\ref{first-A}) and (\ref{second-A}). In the language of Definition \ref{next-def}, the matrix of (\ref{first-A}) is $A_{\ell}$ with $\ell=d_2-i+1$ and the matrix of
 (\ref{second-A}) is $\mathfrak  A_{\ell}$ with $\ell=d_2-i+1$.
\begin{definition}\label{next-def}
For each integer $\ell$, with $3\le \ell$, let $A_\ell$ and $\mathfrak A_{\ell}$ be the following $(\ell-2) \times \ell$ matrices with entries in the polynomial ring $U=k[T_1,T_2]$:
$$A_{\ell}=\left[\begin{matrix}
T_1&T_2&T_1&0&\dots&0\\
0&\ddots&\ddots&\ddots&\ddots&\vdots\\
\vdots&\ddots&\ddots&\ddots&\ddots&0\\
0&\dots&0&T_1&T_2&T_1\end{matrix}\right]\quad\text{and}\quad \mathfrak  A_{\ell}=\left[\begin{matrix}
T_1&0&T_2&0&\dots&0\\
0&\ddots&\ddots&\ddots&\ddots&\vdots\\
\vdots&\ddots&\ddots&\ddots&\ddots&0\\
0&\dots&0&T_1&0&T_2\end{matrix}\right].$$
  \end{definition}

\bigskip
In Lemmas \ref{J08} and \ref{J08'} we resolve $\operatorname{coker} A_{\ell}$ and   $\operatorname{coker} \mathfrak  A_{\ell}$, respectively. In each case, the answer depends on the parity of $\ell$. We decompose each $A_{\ell}$ into four pieces of approximately equal size. The relations on $A_{\ell}$ are constructed from maximal minors of these smaller matrices. In fact, up to re-arrangement of the rows and columns, there are only two constituent pieces for the $A_{\ell}$. We call the two primary constituent pieces $B_k$ and $L_{a\times b}$. \begin{equation}\label{Bk}\begin{array}{l}\text{For each  integer $k$, with $2\le k$, 
let $B_k$ be the $(k-1)\times k$ matrix}\\ \text{$A_{k+1}$ with the last column removed.}\end{array}\end{equation} For example, 
$$B_2=\left[\begin{matrix}T_1&T_2\end{matrix}\right],\quad \text{and}\quad B_3=\left[\begin{matrix} T_1&T_2&T_1\\0&T_1&T_2\end{matrix}\right].$$For each pair of positive integers $a,b$, 
let $L_{a\times b}$ be the $a\times b$ matrix with $T_1$ in the lower left hand corner and zero everywhere else.
\begin{equation}\label{dagger}\begin{array}{l}\text{If $C$ is a matrix, then let $C^\dagger$ represent the matrix that is obtained from $C$ by}\\ \text{rearranging both the rows and the columns of $C$ in the exact opposite order.}\end{array}\end{equation} 
In other words, if $J_k$ is the $k\times k$ matrix with
$$(J_k)_{i,j}=\begin{cases} 1&\text{if $i+j=k+1$}\\0&\text{otherwise,}\end{cases}$$ and $C$ is an $a\times b$ matrix, then $$C^\dagger=J_a CJ_b\,.$$ For example, $$B_3^{\dagger}=\left[\begin{matrix}T_2&T_1&0\\T_1&T_2&T_1\end{matrix}\right],$$ and 
the matrix $L_{a\times b}^\dagger$ is the $a\times b$ matrix with $T_1$ in the upper right hand corner and zero everywhere else. We find it mnemonically helpful to write $u_{a\times b}=L_{a\times b}^\dagger$.
 
Observe that 
$$A_\ell=\begin{cases}\left[\begin{matrix} B_k&L_{k-1\times k}\\u_{k-1\times k}&B_k^\dagger\end{matrix}\right]&\text{if $\ell=2k$}\\ \\
\left[\begin{matrix} B_{k}&L_{k-1\times k+1}\\u_{k\times k}&B_{k+1}^\dagger\end{matrix}\right]&\text{if $\ell=2k+1$.}
\end{cases}$$

\bigskip

We collect a few properties of the objects we have defined.
\begin{observation}\label{ez} The following statements hold.
\begin{itemize}\item[{\rm(1)}] The operation $^\dagger$ respects matrix multiplication in the sense that $(AB)^\dagger=A^\dagger B^\dagger$. 

\item[{\rm(2)}] The last column of $A_{k+1}$ is $[0,\dots,0,T_1]^{\rm T}$.
\item[{\rm(3)}] If the first column of $A_{k+1}$ is deleted, then the resulting matrix is $B_k^\dagger$.

\item[{\rm(4)}] If the first row of $B_{k}^\dagger$ is deleted, then the resulting matrix is $A_k$.
\item[{\rm(5)}] If the last row of $B_{k}$ is deleted, then the resulting matrix is $A_k$.

\item[{\rm(6)}] If $v=[v_1,\dots,v_k]^{\rm T}$ is a vector and $B_k^\dagger v=0$, then 
\begin{itemize}\item[{\rm(a)}]the matrix that is obtained from $B_k^\dagger$ by deleting the first row sends $v'=[0,v_1,\dots,v_{k-1}]^{\rm T}$ to zero, and
\item[{\rm(b)}]the matrix that is obtained from $B_{k+1}^\dagger$ by deleting the first row sends $v''=[0,v_1,\dots,v_k]^{\rm T}$ to zero.
\end{itemize}
\item[{\rm(7)}] If $v=[v_1,\dots,v_{k+1}]^{\rm T}$ is a vector and $B_{k+1}^\dagger v=0$, then the matrix that is obtained from $B_k$ by deleting the last row sends $v'=[v_{k-1},\dots,v_1,0]^{\rm T}$ to zero.
\end{itemize}
\end{observation}
\begin{proof} Assertions (1) -- (5) are obvious. We prove (6). 
Start with $B_k^\dagger v=0$. Apply (3) to see that 
\begin{equation}\label{simple}A_{k+1}\left[\begin{matrix} 0\\v\end{matrix}\right]=0.
\end{equation}
This is assertion (6.b) according to (4). 
The only non-zero entry in the last column of $A_{k+1}$ lives in the last row, see (2). It follows from (\ref{simple}) that $A_{k+1}$, with the last row and column deleted, sends $v'$ to zero. This means that $A_kv'=0$. Now (4) implies (6.a). 

We now prove (7). The only non-zero entry in the last column of $B_{k+1}^\dagger$ lives in the last row according to (3) and (2), and the matrix  obtained from $B_{k+1}^\dagger$ by removing the last row and column is $B_{k}^\dagger$.  Thus   $B_{k+1}^\dagger v=0$ implies that $B_{k}^\dagger v''=0$ with $v''=[v_1,\dots,v_{k}]^{\rm T}$. Now (6.a) and (4) give $A_k v'''=0$ with $v'''=[0,v_1,\dots,v_{k-1}]^{\rm T}$. Apply the operation $^\dagger$ to $A_k v'''=0$ and use (1). The desired conclusion follows because $A_k^\dagger=A_k$, $v'''^\dagger=v'$, and $A_k$ is obtained from $B_k$ by deleting the last row according to (5). 
\QED \end{proof}

\begin{lemma}\label{J08} Adopt the notation of Definition \ref{next-def} and {\rm(\ref{Bk})}.
\begin{itemize}
\item[\rm (1)] If $\ell$ is the even integer $2k$, $m_1,\dots,m_k$ are the signed maximal order minors of $B_k$, and $$C=\left[\begin{matrix} m_1&\dots&m_{k-1}&m_k&0&-m_k&\dots&-m_2\\-m_2&\dots&-m_k&0&m_k&m_{k-1}&\dots&m_1\end{matrix}\right]^{\rm T},$$
then \begin{equation}\label{even-l}0\to U(-k)^2  \stackrel{C}{\longrightarrow} U(-1)^{\ell}\stackrel{A_{\ell}}{\longrightarrow} U^{\ell-2}\end{equation}
is exact.

\item[\rm (2)] If $\ell$ is the odd integer $2k+1$, $m_1,\dots,m_k$ are the signed maximal order minors of $B_k$,
$M_1,\dots, M_{k+1}$ are the signed maximal order minors of $B_{k+1}^\dagger$,
 and $$C=\left\{\begin{matrix} \left[\begin{matrix} m_1&\dots&m_{k-1}&m_k&0&-m_k&\dots&-m_1\\-M_{k-1}&\dots&-M_1&0&M_1&M_2&\dots&M_{k+1}\end{matrix}\right]^{\rm T}&\text{for $5\le \ell$}\\
\left[\begin{matrix} 1&0&-1\\0&T_1&-T_2\end{matrix}\right]^{\rm T}&\text{for $3= \ell$}
,\end{matrix}\right.$$
then \begin{equation}\label{odd-l}0\to U(-k)\oplus U(-k-1)  \stackrel{C}{\longrightarrow} U(-1)^{\ell}\stackrel{A_{\ell}}{\longrightarrow} U^{\ell-2}\end{equation}
is exact.
\end{itemize}
\end{lemma}

\begin{proof} Denote the $i^{\text{th}}$ column of  each matrix $C$ by $C_i$.

\smallskip    We prove (1). We first show that (\ref{even-l}) is a complex. 
Separate   the first column  of $C$ into two pieces: $$C_1=\left[\begin{matrix} t\\b\end{matrix}\right]\quad\text{with}\quad t=\left[\begin{matrix} m_1,\dots,m_k\end{matrix}\right]^{\rm T} \quad\text{and}\quad b=\left[\begin{matrix} 0,-m_k,\dots,-m_2\end{matrix}\right]^{\rm T}.$$ The definition of the $m$'s gives $B_kt=0$. The product $L_{k-1\times k} b$ is also zero because the only non-zero entry in $L_{k-1\times k}$ is multiplied by zero. The product of row $k$ from $A_\ell$ times $C_1$ is $T_1 m_k-T_1m_k=0$. The last $k-2$ rows of $u_{k-1\times k}$ are zero. Apply (1) from Observation \ref{ez} to the equation $B_kt=0$ to see that
$B_k^\dagger$ sends $t^\dagger=[m_k,\dots,m_1]^{\rm T}$ to zero. It follows from (6.a) in Observation \ref{ez}  that the last $k-2$ rows of $B_k^\dagger$ kill $b$. (The minus signs do not cause any difficulty.)  We have shown that $A_\ell C_1=0$. It follows from (1) of Observation \ref{ez} that $A_{\ell}^\dagger C_1^\dagger=0$; but $A_{\ell}^\dagger=A_\ell$ and $C_1^\dagger=C_2$. It follows that $A_{\ell}C_2=0$ and therefore $A_{\ell}C=0$ and (\ref{even-l}) is a complex. We apply the Buchsbaum-Eisenbud criterion to show that (\ref{even-l}) is exact.   The ideal of  $2\times 2$ minors of $C$ has grade  two since
$$\left\vert \begin{matrix}m_k&0\\0&m_k\end{matrix}\right\vert= T_1^{\ell-2}\quad\text{and}\quad \left\vert \begin{matrix}m_1&-m_2\\-m_2&m_1\end{matrix}\right\vert\equiv T_2^{\ell-2} \mod (T_1).$$ Assertion (1) has been established. 

\smallskip     Assertion (2) is obvious when $\ell=3$. Henceforth, we assume $5\le \ell$. We next show that (\ref{odd-l}) is a complex. 
Write $$C_1=\left[\begin{matrix} t\\b\end{matrix}\right]\quad\text{with}\quad t=\left[\begin{matrix} m_1,\dots,m_k\end{matrix}\right]^{\rm T} \quad\text{and}\quad b=\left[\begin{matrix} 0,-m_k,\dots,-m_1\end{matrix}\right]^{\rm T}.$$
The definition of the $m$'s ensures that 
$B_k t=0$. The product  $L_{k-1\times k+1} b$ is zero because the only non-zero entry of $L_{k-1\times k+1}$ is multiplied by zero. The product of row $k$ of  $A_\ell$ times  $C_1$ is $T_1m_k-T_1m_k=0$. The last $k-1$ rows of $u_{k\times k}$ are identically zero. Apply (6.b) from Observation \ref{ez} to $B_k^\dagger t^\dagger=0$ in order to see that the last $k-1$ rows of $B_{k+1}^\dagger$ times $b$ are equal to zero. (Again, the signs play no role in this part of the calculation.) We have shown that $A_{\ell}C_1=0$.
Write $$C_2=\left[\begin{matrix} t\\b\end{matrix}\right]\quad\text{with}\quad t=\left[\begin{matrix} -M_{k-1},\dots,-M_1,0\end{matrix}\right]^{\rm T} \quad\text{and}\quad b=\left[\begin{matrix} M_1,\dots,M_{k+1}\end{matrix}\right]^{\rm T}.$$
The definition of the $M$'s guarantees that $B_{k+1}^\dagger b=0$. The product $u_{k\times k}t$ is zero  because the only non-zero entry of $u_{k\times k}$ is multiplied by zero. The product of row $k-1$ of $A_\ell$ times $C_2$ is $-T_1M_1+T_1M_1=0$. The top $k-2$ rows of $L_{k-1\times k+1}$ are identically zero. Assertion (7) of Observation \ref{ez}  shows that the top $k-2$ rows of $B_k$ times $t$ are equal to zero.   We have shown that $A_\ell C=0$. Again the ideal $I_2(C)$ has grade two since
$$\left\vert \begin{matrix}m_k&0\\0&M_1\end{matrix}\right \vert =   \pm T_1^{\ell-2}\quad\text{and}\quad 
\left\vert \begin{matrix}m_1&-M_{k-1}\\-m_1&M_{k+1}\end{matrix}\right \vert \equiv \pm T_2^{\ell-2}\mod (T_1).$$ The complex (\ref{odd-l}) is also exact by the Buchsbaum-Eisenbud criterion.
\QED
\end{proof}

\medskip

We next resolve  $\operatorname{coker} \mathfrak  A_{\ell}$ for the matrices $\mathfrak  A_{\ell}$ of Definition \ref{next-def}.
\begin{example} \label{56} When $\ell$ is $5$ or $6$, the syzygy module for the matrices $\mathfrak  A_5$  and $\mathfrak  A_6$ are generated by the columns of $\mathfrak C_5$ and $\mathfrak C_6$, respectively, for 
$$\mathfrak C_5=\left[\begin{matrix}0&T_2^2\\T_2&0\\0&-T_1T_2\\-T_1&0\\0&T_1^2\end{matrix}\right]\quad\text{ and }\quad \mathfrak C_6=
\left[\begin{matrix}0&T_2^2\\T_2^2&0\\0&-T_1T_2\\-T_1T_2&0\\0&T_1^2\\T_1^2&0\end{matrix}\right].$$Indeed, the rows and columns of  $\mathfrak  A_5$  and $\mathfrak  A_6$ may be rearranged to convert these matrices into the block matrices
$$\left[\begin{array}{ll|lll}
T_1&T_2&0&0&0\\\hline
0&0&T_1&T_2&0\\
0&0&0&T_1&T_2\end{array}\right]\quad\text{and}\quad \left[\begin{array}{lll|lll}
T_1&T_2&0&0&0&0\\
0&T_1&T_2&0&0&0\\\hline
0&0&0&T_1&T_2&0\\
0&0&0&0&T_1&T_2\end{array}\right],$$respectively. The syzygies of the constituent pieces are well understood. 
\end{example}
In Lemma \ref{J08'} we prove that the pattern established in Example \ref{56} holds for all $\ell$ with $3\le \ell$. We find it convenient to let $V$ be a free module of rank two with basis $s,t$ over the polynomial ring $U=k[T_1,T_2]$. The $(\ell-2) \times \ell$ matrix $\mathfrak  A_\ell$ represents the composition 
\begin{equation}\label{comp}\Sym_{\ell-1}V\lto \Sym_{\ell+1}V \lto \frac{\Sym_{\ell+1}V}{(t^{\ell+1},st^\ell,s^\ell t,s^{\ell+1})}\, ,\end{equation} where the first map is multiplication by $\xi=T_2t^2+T_1s^2$ and the second map is the natural quotient map. The basis for  $\Sym_{\ell-1}V$ is $t^{\ell-1},st^{\ell-2},\dots,s^{\ell-2}t,s^{\ell-1}$ and the basis for $\frac{\Sym_{\ell+1}V}{(t^{\ell+1},st^\ell,s^\ell t,s^{\ell+1})}$ is $s^2t^{\ell-1},s^3t^{\ell-2},\dots,s^{\ell-1}t^2$. For each positive integer $\alpha$, let
\begin{equation}\label{ker-elt} \kappa_\alpha=\sum_{a+b=\alpha}(T_2t^2)^a(-T_1s^2)^b\in \Sym_{2\alpha}V.\end{equation}
Notice that the columns of $\mathfrak C_5$ are $st\kappa_1$ and $\kappa_2$ in $\Sym_4V$ and the columns of $\mathfrak C_6$ are $s\kappa_2$ and $t\kappa_2$ in $\Sym_5V$.

\begin{lemma}\label{pre-J08'} If $\ell\ge 3$ is an integer, then the  kernel of the composition of {\rm(\ref{comp})} is generated by 
$$\begin{cases}
st\kappa_{\frac{\ell-3}2}\text{ and } \kappa_{\frac{\ell-1}2}\text{ in } \Sym_{\ell-1}V&\text{if $\ell$ is odd}\\
s\kappa_{\frac{\ell-2}2}\text{ and } t\kappa_{\frac{\ell-2}2}\text{ in } \Sym_{\ell-1}V&\text{if $\ell$ is even}.
\end{cases}$$
\end{lemma}

\begin{proof}Observe that $$\xi\cdot \kappa_{\alpha}=(T_2t^2+T_1s^2)\sum_{a+b=\alpha}(T_2t^2)^a(-T_1s^2)^b=
(T_2t^2)^{\alpha+1}+(-T_1s^2)^{\alpha+1}\in \Sym_{2\alpha+2}V.$$ It is now clear that
$$0\to U(-({\textstyle\frac{\ell-1}2)})\oplus U(-({\textstyle\frac{\ell+1}2)})\stackrel{\left[\begin{matrix} st\kappa_{\frac{\ell-3}2}&\kappa_{\frac{\ell-1}2}\end{matrix}\right]}{\xrightarrow{\hspace*{3cm}} }\Sym_{\ell-1}V(-1)\stackrel{\text{\rm (\ref{comp})}}{\xrightarrow{\hspace*{1cm}} } \frac{\Sym_{\ell+1}V}{(t^{\ell+1},st^\ell,s^\ell t,s^{\ell+1})}
$$
and
$$0\to U(-({\textstyle\frac{\ell}2)})^2\stackrel{\left[\begin{matrix} s\kappa_{\frac{\ell-2}2}&t\kappa_{\frac{\ell-2}2}\end{matrix}\right]}{\xrightarrow{\hspace*{3cm}} }\Sym_{\ell-1}V(-1)\stackrel{\text{\rm (\ref{comp})}}{\xrightarrow{\hspace*{1cm}} } \frac{\Sym_{\ell+1}V}{(t^{\ell+1},st^\ell,s^\ell t,s^{\ell+1})}
$$
are complexes, when $\ell$ is  odd or even, respectively. Apply the Buchsbaum-Eisenbud criterion to see that these complexes are exact: the determinant of the first two rows of the  matrices  $\left[\begin{matrix} st\kappa_{\frac{\ell-3}2}&\kappa_{\frac{\ell-1}2}\end{matrix}\right]$ and $\left[\begin{matrix} s\kappa_{\frac{\ell-2}2}&t\kappa_{\frac{\ell-2}2}\end{matrix}\right]$ is plus or minus a power of $T_2$, and the determinant of the last two rows of these matrices  is plus or minus a power of $T_1$.
\QED\end{proof}

\begin{lemma}\label{J08'} Adopt the notation of Definition {\rm (\ref{next-def})}.
\begin{itemize}
\item[\rm(1)] If $\ell$ is the even integer $2k$ and 
$$\mathfrak C_{\ell}=\left[\begin{matrix} 0&T_2^{k-1}\\T_2^{k-1}&0\\0&(-T_1)T_2^{k-2}\\(-T_1)T_2^{k-2}&0\\0&(-T_1)^2T_2^{k-3}\\\vdots&\vdots\\0&(-T_1)^{k-1}\\(-T_1)^{k-1}&0 \end{matrix}\right],$$ then
$$0\to U(-k)^2  \stackrel{\mathfrak C_{\ell}}{\longrightarrow} U(-1)^{\ell}  \stackrel{\mathfrak A_{\ell}}{\longrightarrow} U^{\ell-2}$$
is exact.
\item[\rm(2)] If $\ell$ is the odd integer $2k+1$ and $$\mathfrak C_{\ell}=\left[\begin{matrix} 0&T_2^{k}\\T_2^{k-1}&0\\0&(-T_1)T_2^{k-1}\\(-T_1)T_2^{k-2}&0\\0&(-T_1)^2T_2^{k-2}\\\vdots&\vdots\\(-T_1)^{k-1}&0\\0&(-T_1)^{k}\end{matrix}\right],$$ then
$$0\to U(-k)\oplus U(-k-1)  \stackrel{\mathfrak C_{\ell}}{\longrightarrow} U(-1)^{\ell}  \stackrel{\mathfrak A_{\ell}}{\longrightarrow} U^{\ell-2}$$
is exact.
\end{itemize}
\end{lemma}
\begin{proof} 
The present Lemma is a restatement of Lemma \ref{pre-J08'}.
\QED\end{proof}

\medskip

We are now ready to prove the main result of this section. Adopt Data \ref{data4} with $k$ a field closed under taking square roots. According to Remark~\ref{R5.2}, we may assume that the parametrization $\mathbb P_k^1 \twoheadrightarrow \mathcal C$  is birational. We know from Corollary \ref{B-mod} -- see also Tables \ref{tbl2} and \ref{tbl3} -- that 
there exist bi-homogeneous elements $g_{(i,j)}\in \mathcal A_{(i,j)}$ such that the minimal generating set of the $B$-module $\mathcal A$ is given by
\begin{equation}\label{gens}
\{g_{(0,d_1+d_2)}\}\cup\left\{g_{\left(i,\frac{d_2+2-i}2\right)}\left\vert \begin{array}{l}1\le i\le d_2-2\\\text{and  $d_2-i$ is even}\end{array}\right.\right\}\cup \begin{cases} \left\{g_{\left(1,\frac{d_2+3}2\right)}\right\}&\text{if $d_2$ is odd}\vspace{5pt}\\ \left\{g_{\left(1,\frac{d_2+2}2\right)},g'_{\left(1,\frac{d_2+2}2\right)}\right\}&\text{if $d_2$ is even}.\end{cases}\end{equation}
The element $g_{(0,d_1+d_2)}$ is the resultant of $g_1$ and $g_2$, when these polynomials are viewed as homogeneous forms in $S[x,y]$ of degree $d_1$ and $d_2$, respectively. In Theorem \ref{goal5} we record explicit formulas for the rest of the $g_{(i,j)}$ from (\ref{gens}). According to Observation \ref{modify}, we may assume that the first column of $\varphi$ has one of two forms. The linear form $c_{w,2}$ of $S$ is defined in (\ref{g}) by the equation
$$g_2=\sum_{w=0}^{d_2}c_{w,2}x^wy^{d_2-w}.$$
The  results of \cite{BD14} and Theorem~\ref{goal5} were obtained more or less simultaneously; indeed,  \cite{BD14} and this article were posted on the arXiv three days apart.
\begin{theorem}\label{goal5} Adopt Data \ref{data4} with $k$ a field closed under taking square roots. Without loss of generality, the parametrization $\mathbb P_k^1 \twoheadrightarrow \mathcal C$  is birational. 
 We give explicit formulas for the elements  of {\rm(\ref{gens})}. The element $g_{i,j}$ or $g'_{i,j}$ is in $\mathcal A_{i,j}$ and the   set $\{g_{i,j},g_{i,j}'\}$ of {\rm(\ref{gens})} is a  bi-homogeneous minimal generating set for the $B$-module $\mathcal A$.

\begin{itemize}
\item[(1)]  Assume  that the first column of $\varphi$ is $[x^2+y^2,xy,0]^{\rm T}$.
\begin{itemize}
\item[(a)] If  $1\le i\le d_2-4$ and
$d_2-i$ is even, then 
$$
g_{\left(i,\frac{d_2+2-i}2\right)}= \begin{cases} 
\sum\limits_{w=0}^i  T_2c_{w,2}x^{w} y^{i-w} m_{1} 
 + \sum\limits_{w=1}^i  
 T_1c_{w-1,2}
x^{w} y^{i-w}m_{1}+\sum\limits_{w=0}^i    
T_1c_{w,2}
x^{w} y^{i-w}m_{2} \\
+ T_1  
x^{i}\left(\sum\limits_{\beta=2}^{\frac{d_2-i}2}c_{d_2+2-\beta,2}m_\beta-\sum\limits_{\beta=1}^{\frac{d_2-i}2}c_{\beta+i,2}m_\beta\right)\\
+(T_2x^{i}+T_1x^{i-1} y)\sum\limits_{\beta=1}^{\frac{d_2-i}2}(c_{d_2+1-\beta,2}-c_{\beta+i-1,2})m_\beta\, ,
\end{cases}$$ where $m_1,\dots,m_{\frac{d_2-i}2}$ are the signed maximal minors of the matrix $B_{\frac{d_2-i}2}$ of {\rm(\ref{Bk})}.
\item[(a$'$)] If $i=d_2-2$,  then
$$g_{\left(d_2-2,2\right)}= \begin{cases}\sum\limits_{w=0}^{d_2-3}  
 T_2c_{w,2}
x^{w} y^{d_2-2-w} + \sum\limits_{w=1}^{d_2-2} 
 T_1c_{w-1,2}
x^{w} y^{d_2-2-w} 
-  
T_1c_{d_2-1,2}   
x^{d_2-2}
\\   
-  
T_1c_{d_2-2,2}   
x^{d_2-3} y+T_2c_{d_2,2}   
x^{d_2-2} 
+  
T_1c_{d_2,2}   
x^{d_2-3} y\, . \end{cases}$$

\item [(b)] If $i=1$ and $d_2$ is odd, then 
$$g_{\left(1,\frac{d_2+3}2\right)}=\begin{cases}  
+\underline{\chi}(d_2=3)   
(T_1c_{0,2}y+T_1c_{1,2}x)
M_{1}\\
-\underline{\chi}(5\le d_2)
\left( T_2c_{0,2}y + T_2c_{1,2}x 
+   
 T_1c_{0,2}
x\right) M_{\frac{d_2-3}2}\\
-\underline{\chi}(7\le d_2)       
(T_1c_{0,2}y+T_1c_{1,2}x)
M_{\frac{d_2-5}2}-\sum\limits_{\beta=1}^{\frac{d_2-1}2} T_1c_{\frac{d_2+1}2+\beta,2}xM_{\beta}\\
+\sum\limits_{\beta=1}^{\frac{d_2-3}2} (T_1c_{\frac{d_2+1}2-\beta,2}x+T_2c_{\frac{d_2-1}2-\beta,2}x+T_1c_{\frac{d_2-1}2-\beta,2} y)M_{\beta}
\\-\sum\limits_{\beta=1}^{\frac{d_2+1}2}(T_2c_{\frac{d_2-1}2+\beta,2}x+T_1c_{\frac{d_2-1}2+\beta,2}  y)M_{\beta}\, ,
\end{cases}$$
where $M_1,\dots,M_{\frac{d_2+1}2}$ are the signed maximal order minors of the matrix $B_{\frac{d_2+1}2}^\dagger$,
and the matrix $B_{\frac{d_2+1}2}$ and the operation ``$\dagger$'' are defined at {\rm(\ref{Bk})} and {\rm(\ref{dagger})}, respectively.

\item [(c)] If $i=1$ and $d_2$ is even, then 
$$g_{\left(1,\frac{d_2+2}2\right)}=\begin{cases}
    (T_2c_{0,2} y +T_2c_{1,2}x +  T_1c_{0,2}x)m_{1}+  (T_1c_{0,2} y+T_1c_{1,2}x)m_{2}\\
-\sum\limits_{\beta=1}^{\frac{d_2}2} (T_1c_{\beta+1,2}x +T_2c_{\beta,2}x +T_1c_{\beta,2}y)m_{\beta} 
\\
+ \sum\limits_{\beta=3}^{\frac{d_2}2}  T_1c_{d_2+3-\beta,2}x m_{\beta} 
+\sum\limits_{\beta=2}^{\frac{d_2}2} (T_2c_{d_2+2-\beta,2}x +T_1c_{d_2+2-\beta,2} y)m_{\beta}
\end{cases}
$$
 and 
$$\hspace{1.2cm} g'_{\left(1,\frac{d_2+2}2\right)} 
=  \begin{cases}
 -\left(T_2c_{0,2}y+T_2c_{1,2}x +T_1c_{0,2}x\right) m_{2} 
-\underline{\chi}(6\le d_2)  (T_1c_{0,2}y+T_1c_{1,2}x) m_{3}\\
+\sum\limits_{\beta=2}^{\frac{d_2}2}(T_1c_{\beta,2}x +T_2c_{\beta-1,2}x +T_1c_{\beta-1,2} y-T_1c_{d_2-\beta+2,2}x )m_{\beta}\\
 - \sum\limits_{\beta=1}^{\frac{d_2}2}(T_2c_{d_2+1-\beta,2}x +T_1c_{d_2+1-\beta,2} y)m_{\beta}\, ,\end{cases}$$ 
where $m_1,\dots,m_{\frac{d_2}2}$ are the signed maximal order minors of the matrix $B_{\frac{d_2}2}$ of {\rm(\ref{Bk})}.
\end{itemize}

\item[(2)] Assume  that the first column of $\varphi$ is $[y^2,x^2,0]^{\rm T}$.
\begin{itemize}
\item[(a)] If  $1\le i\le d_2-2$ and
$d_2-i$ is even, then 
$$g_{\left(i,\frac{d_2+2-i}2\right)}= \sum\limits_{w=0}^i  c_{w,2}x^{w} y^{i-w}T_2^{\frac{d_2-i}2}+
\sum\limits_{\lambda=1}^{\frac{d_2-i}2}     
(c_{i+2\lambda-1,2}y+  c_{i+2\lambda,2}x)x^{i-1} (-T_1)^{\lambda}T_2^{\frac{d_2-i}2-\lambda}.   $$ 
\item [(b)] If $i=1$ and $d_2$ is odd, then 
$$g_{\left(1,\frac{d_2+3}2\right)}=  \sum\limits_{\lambda=1}^{\frac{d_2+1}2}   
c_{2\lambda-1,2}   
  y (-T_1)^{\lambda}T_2^{\frac{d_2+1}2-\lambda} 
+\sum\limits_{\lambda=0}^{\frac{d_2-1}2}   
c_{2\lambda,2}   
x  (-T_1)^{\lambda}T_2^{\frac{d_2+1}2-\lambda}.
$$

\item [(c)] If $i=1$ and $d_2$ is even, then 
$$\begin{array}{rcl}g_{\left(1,\frac{d_2+2}2\right)}&=& 
\sum\limits_{\lambda=1}^{\frac{d_2}2}  
c_{2\lambda-1,2}   
  y (-T_1)^{\lambda}T_2^{\frac{d_2}2-\lambda} 
+\sum\limits_{\lambda=0}^{\frac{d_2}2}    
c_{2\lambda,2}   
x(-T_1)^{\lambda}T_2^{\frac{d_2}2-\lambda}\quad\text{and}\\
g'_{\left(1,\frac{d_2+2}2\right)}& 
=&  \sum\limits_{\lambda=0}^{\frac{d_2}2}   c_{2\lambda,2}    y (-T_1)^{\lambda}T_2^{\frac{d_2}2-\lambda} 
+\sum\limits_{\lambda=0}^{\frac{d_2-2}2}   
c_{2\lambda+1,2}   
x (-T_1)^{\lambda}T_2^{\frac{d_2}2-\lambda}. \end{array}$$ 
\end{itemize}
\end{itemize}

\end{theorem}

\begin{proof} We first prove (1.a) and (1.a$'$). Apply the technique of Corollary \ref{qchi}. The matrix ${\Upsilon^{\text{\rm T}}_{d_2-i-1,1}}$ of (\ref{rcl}) is given in 
(\ref{first-A}) and this matrix is called $A_{\ell}$, with $\ell=d_2-i+1$, in Definition \ref{next-def}. The kernel of $A_{\ell}$ is calculated in part (2) of Lemma \ref{J08} because the index $\ell=d_2-i+1$ is odd. This kernel is free of rank two. The homogeneous minimal generating set for the kernel of $A_{\ell}$ has one generator in each of two consecutive degrees. Let $\chi=[\chi_0,\dots,\chi_{\ell-1}]^{\rm T}$ be the generator with the smaller degree. The technique of Corollary \ref{qchi} and the formula (\ref{gens}) 
then give
\begin{equation}\label{dot-product}g_{\left(i,\frac{d_2+2-i}2\right)}=\sum_{\beta=0}^{\ell-1} q_{\beta,d_2-i-\beta}\chi_{\beta} \, .\end{equation} An explicit formula for 
$q_{\beta,d_2-i-\beta}$ is given in Corollary \ref{extract2*} and 
\begin{equation}\label{local-chi}\chi=\begin{cases}
[m_1,\dots,m_k,0,-m_k,\dots,-m_1]^{\text{\rm T}}&\text{if $4\le d_2-i$}\\
[1,0,-1]^{\text{\rm T}}&\text{if $2= d_2-i$}
 \end{cases}\end{equation}is given in part (2) of Lemma \ref{J08}, 
  where $k=\frac{\ell-1}2=\frac{d_2-i}2$, $B_k$ is the matrix described in (\ref{Bk}), and $m_i$ is the $i^{\text{th}}$ signed maximal order minor of $B_k$. 

If $i=d_2-2$, then (\ref{dot-product}) yields $g_{\left(d_2-2, 2\right)}=q_{0,2}-q_{2,0}$, 
with $$ q_{0,2}= 
\begin{cases}\sum\limits_{w=0}^{d_2-3}  c_{1,1}c_{w,2}x^{w} y^{d_2-2-w} 
+ \sum\limits_{w=1}^{d_2-2}  
 c_{2,1}c_{w-1,2}
x^{w} y^{d_2-2-w} \\
-c_{0,1}c_{d_2-1,2}x^{d_2-2}-c_{0,1}c_{d_2-2,2}x^{d_2-3} y\end{cases}
 $$
and $q_{2,0}= 
-c_{1,1}c_{d_2,2}x^{d_2-2}-c_{0,1}c_{d_2,2}x^{d_2-3} y$. The polynomial $g_1=c_{0,1}y^2+c_{1,1}xy+c_{2,1}x^2$ is equal to 
$T_1(x^2+y^2)+T_2xy$; so
\begin{equation}\label{local-c}c_{0,1}=T_1,\quad c_{1,1}=T_2,\quad\text{and}\quad c_{2,1}=T_1,\end{equation} and the computation of (1.a$'$) is complete. 

If $4\le d_2-i$, then (\ref{local-chi}) gives $$\chi_{\beta}=\begin{cases} m_{\beta+1}&\text{if $0\le \beta\le \frac{d_2-i}2-1$}\\  
0&\text{if $ \beta=\frac{d_2-i}2$}\\
-m_{d_2-i+1-\beta}&\text{if $\frac{d_2-i}2+1\le \beta\le  d_2-i$} \, ;\end{cases}$$ and therefore, 
  $g_{\left(i,\frac{d_2+2-i}2\right)}=\theta_1+\theta_2$, with $$\theta_1=\sum\limits_{\beta=0}^{\frac{d_2-i}2-1}q_{\beta,d_2-i-\beta}m_{\beta+1}\quad\text{and}\quad \theta_2=-\sum\limits_{\beta=\frac{d_2-i}2+1}^{d_2-i}q_{\beta,d_2-i-\beta}m_{d_2-i+1-\beta} \, .$$ Use Corollary \ref{extract2*} to calculate 
$$\theta_1=\begin{cases} \sum\limits_{w=0}^i  c_{1,1}c_{w,2}x^{w} y^{i-w} m_{1}
+  \sum\limits_{w=1}^ic_{2,1}c_{w-1,2}x^{w} y^{i-w}m_{1}  
+\sum\limits_{w=0}^i    c_{2,1}c_{w,2}x^{w} y^{i-w}m_{2}\\
-\sum\limits_{\beta=1}^{\frac{d_2-i}2}c_{0,1}c_{i+\beta,2}x^{i}m_{\beta} 
-\sum\limits_{\beta=1}^{\frac{d_2-i}2}c_{1,1}c_{i+\beta-1,2}x^{i}m_{\beta} 
-\sum\limits_{\beta=1}^{\frac{d_2-i}2}c_{0,1}c_{i+\beta-1,2}x^{i-1} ym_{\beta}\end{cases}
$$
and
$$\theta_2=
\sum\limits_{\beta=2}^{\frac{d_2-i}2}c_{0,1}c_{d_2+2-\beta,2}x^{i}m_{\beta} 
+\sum\limits_{\beta= 1}^{\frac{d_2-i}2}c_{1,1}c_{d_2+1-\beta,2}x^{i}m_{\beta}
+\sum\limits_{\beta= 1}^{\frac{d_2-i}2}c_{0,1}c_{d_2+1-\beta,2}x^{i-1} ym_{\beta}.
$$
Use (\ref{local-c}) to complete the proof of (1.a). 

To prove (1.b), we consider the minimal syzygy of degree $k$
$$\chi=\left[\begin{matrix} -M_{k-1}&\dots&-M_1&0&M_1&\dots&M_{k+1}\end{matrix}\right]^{\rm T}$$ of the matrix $A_{d_2}$, where $k=\frac{d_2-1}2$, as given in part (2) of Lemma \ref{J08}. This formulation makes sense and gives 
$$\chi=\left[\begin{matrix}  0&T_1&-T_2 \end{matrix}\right]^{\rm T}$$ when $d_2=3$. In other words, $\chi=[\chi_0,\dots,\chi_{d_2-1}]^{\rm T}$, with $$\chi_{\beta}=\begin{cases}
-M_{\frac{d_2-3}2-\beta}&\text{if $0\le \beta\le \frac{d_2-5}2 $}\\
0&\text{if $\beta=\frac{d_2-3}2$}\\
M_{\beta+1-\frac{d_2-1}2}&\text{if $\frac{d_2-1}2\le \beta\le d_2-1$}.
\end{cases}$$The techniques of Corollary \ref{qchi} give 
$$g_{\left(1,\frac{d_2+3}2\right)}=\sum_{\beta=0}^{\ell-1} q_{\beta,d_2-1-\beta}\chi_{\beta}=-\sum_{\beta=0}^{\frac{d_2-1}2-2} q_{\beta,d_2-1-\beta}M_{\frac{d_2-3}2-\beta}+\sum_{\beta=\frac{d_2-1}2}^{d_2-1} q_{\beta,d_2-1-\beta}M_{\beta-\frac{d_2-1}2+1}.$$We apply Corollary \ref{extract2*} and (\ref{local-c}) as we simplify this expression.

The computation of (1.c) proceeds in the same manner. One begins with the relations
$$\chi=\left[\begin{matrix} m_1&\dots& m_{\frac{d_2}2}&0&-m_{\frac{d_2}2}&\dots&-m_2\end{matrix}\right]^{\rm T}
\ \ \text{and}\ \  \chi'=\left[\begin{matrix} -m_2&\dots&-m_{\frac{d_2}2}&0&m_{\frac{d_2}2}&\dots&m_1\end{matrix}\right]^{\rm T}$$
on the matrix $A_{d_2}$, as given in  part (1) of Lemma \ref{J08}, where the matrix $B_{\frac{d_2}2}$ is defined in (\ref{Bk}) and $m_1,\dots,m_{d_2}$ are the maximal order minors of $B_{\frac{d_2}2}$. The relations $\chi$ and $\chi'$ are  used to produce $g_{\left(1,\frac{d_2+2}2\right)}$ and $g_{\left(1,\frac{d_2+2}2\right)}'$, respectively. 

Now we prove (2.a). Again we use the method of Corollary \ref{qchi}. Now the matrix ${\Upsilon^{\text{\rm T}}_{d_2-i-1,1}}$ of (\ref{rcl}) is given in 
(\ref{second-A}) and this matrix is called $\mathfrak A_{\ell}$, with $\ell=d_2-i+1$, in Definition \ref{next-def}. The relation
$$\chi=\left[\begin{matrix}0\\T_2^{\alpha}\\0\\(-T_1)T_2^{\alpha-1}\\\vdots\\(-T_1)^{\alpha}\\0 \end{matrix}\right],$$ with $\alpha=\frac{d_2-i-2}2$, is read from part (2) of Lemma \ref{J08'}. We see that $$\chi_\beta=\begin{cases} 0&\text{if $\beta$ is even}\\
(-T_1)^{\lambda}T_2^{\frac{d_2-i-2}2-\lambda}&\text{if $\beta=2\lambda+1$ and $0\le \lambda\le \frac{d_2-i-2}2$}\, .\end{cases}$$ 
Use Corollary \ref{qchi} and Corollary \ref{extract2*} to see that $g_{\left(i, \frac{d_2+2-i}2\right)}$ is equal to 
$$\begin{array}{rl}&\sum\limits_{\lambda=0}^{\frac{d_2-i}2-1} q_{2\lambda+1,d_2-i-(2\lambda+1)}(-T_1)^{\lambda}T_2^{\frac{d_2-i}2-\lambda-1}\\
=&
\begin{cases}\sum\limits_{w=0}^i c_{2,1}c_{w,2}x^{w} y^{i-w}T_2^{\frac{d_2-i}2-1}\\
+\sum\limits_{\lambda=0}^{\frac{d_2-i}2-1} \left(-c_{0,1}c_{i+2\lambda+2,2}x^{i}-c_{1,1}c_{i+2\lambda+1,2}x^{i}-c_{0,1}c_{i+2\lambda+1,2}x^{i-1} y\right)(-T_1)^{\lambda}T_2^{\frac{d_2-i}2-\lambda-1}.\end{cases}\end{array}
$$
The polynomial $g_1=c_{0,1}y^2+c_{1,1}xy+c_{2,1}x^2$ is equal to 
$T_1y^2+T_2x^2$; so
\begin{equation*}c_{0,1}=T_1,\quad c_{1,1}=0,\quad\text{and}\quad c_{2,1}=T_2\, ,\end{equation*} and the computation of (2.a) is complete. 

The computations of (2.b) and (2.c) proceed  in the same manner. One uses
$$\chi=\left[\begin{matrix}T_2^{\frac{d_2-1}2}\\0\\\vdots\\0\\(-T_1)^{\frac{d_2-1}2} \end{matrix}\right], \quad \chi=\left[\begin{matrix}  T_2^{\frac{d_2}2-1}\\0\\\vdots\\0\\(-T_1)^{\frac{d_2}2-1}\\0\end{matrix}\right], \quad \text{and}\quad 
\chi=\left[\begin{matrix}0\\T_2^{\frac{d_2}2-1}\\0\\\vdots\\0\\(-T_1)^{\frac{d_2}2-1}\end{matrix}\right]$$ to compute 
$g_{\left(1,\frac{d_2+3}2\right)}$ (when $d_2$ is odd), and $g_{\left(1,\frac{d_2+2}2\right)}$ and 
$g'_{\left(1,\frac{d_2+2}2\right)}$ (when $d_2$ is even), respectively.
\QED
\end{proof}

\section{The case of $d_1=d_2$.}\label{d1=d2}
The $S$-module structure of $\mathcal A_{\ge d_2-1}$ is completely described in Corollary \ref{XXX} for all choices of $d_1\le d_2$ in Data \ref{data1}. If $d_1<d_2$ and the first column of $\varphi$ contains a generalized zero, then the $S$-module structure of $\mathcal A_{\ge d_1-1}$ is completely described in Theorem \ref{claudia}; see also Table \ref{table}. In Theorem \ref{andy} we assume that $d_1=d_2$ and we describe
$\mathcal A_{d_1-2}$; hence, in this case, the $S$-module structure of $\mathcal A_{\ge d_1-2}$ is completely described by combining Theorem \ref{andy} and Corollary \ref{XXX}. The geometric significance of Theorem \ref{andy} is explained in  Remark \ref{signif}. A preliminary version of Theorem \ref{andy} initiated the investigation    that culminated in \cite{CKPU}. Most of the calculations that are used in the proof of Theorem \ref{andy} have already been incorporated in \cite{CKPU}. The geometric applications of these calculations are emphasized in \cite{CKPU}. In the present work we focus on the application of these calculations to Rees algebras.

\begin{data}\label{33data} Adopt Data {\rm \ref{data1}} with $d_1=d_2$. Let $$C= \left[\begin{matrix} c_{0,1}&c_{0,2}\\\vdots&\vdots \\c_{d_1,1}&c_{d_1,2}\end{matrix}\right]$$ be the matrix $[\Upsilon_{1,1}|\Upsilon_{1,2}]$ of Definition \ref{upsilon}. Notice that
\begin{equation}\label{3.11} \left[\begin{matrix}T_1&T_2&T_3\end{matrix}\right]\varphi = \left[\begin{matrix} y^{d_1}&xy^{d_1-1}&\cdots&x^{d_1}\end{matrix}\right]C.\end{equation}  \end{data}

\begin{theorem} \label{andy}
If Data {\rm \ref{33data}} is adopted, then the following statements hold.
\begin{itemize}
\item[{\rm (1)}] There are eight possible values for the pair of integers $(\mu(I_1(\varphi)),\mu(I_2(C)))$. Indeed,  the 
following chart gives the possible values for $(\mu(I_1(\varphi)),\mu(I_2(C)))$ as a function of $d_1${\rm:}

\bigskip
\begin{center}
\begin{tabular}{c|l}
$d_1$&\phantom{$(3,3)$ o } possible values for $(\mu(I_1(\varphi)),\mu(I_2(C)))$\\ \hline
 $5$ or more  &$(6,6)$, $(5,6)$, $(5,5)$, $(4,6)$, $(4,5)$, $(4,4)$, $(3,3)$, or $(2,1)$\\
$4$&\phantom{$(6,6)$, }$(5,6)$, $(5,5)$, $(4,6)$, $(4,5)$, $(4,4)$, $(3,3)$, or $(2,1)$\\
$3$&\phantom{$(6,6)$, $(5,6)$, $(5,5)$, }$(4,6)$, $(4,5)$, $(4,4)$, $(3,3)$, or $(2,1)$\\
$2$&\phantom{$(6,6)$, $(5,6)$, $(5,5)$, $(4,6)$, $(4,5)$, $(4,4)$, }$(3,3)$, or $(2,1)$\\
$1$&\phantom{$(6,6)$, $(5,6)$, $(5,5)$, $(4,6)$, $(4,5)$, $(4,4)$, $(3,3)$, or }$(2,1)\, $.\\
\end{tabular}
\end{center}
 
\item[{\rm (2)}] The $S$-module $\mathcal A_{d_1-2}$ is  resolved by
$$\begin{array} {ccccllc}
 0&\to& S(-4)^2&\to& S(-2)^{d_1-5}\oplus S(-3)^6&\text{if $(\mu(I_1(\varphi)),\mu(I_2(C)))=(6,6)$}\\
0&\to&
S(-5) &\to  &
S(-2)^{d_1-4}
 \oplus
S(-3)^3
\oplus S(-4)
&\text{if $(\mu(I_1(\varphi)),\mu(I_2(C)))=(5,6)$}\\
0&\to&
S(-4)& \to&
S(-2)^{d_1-4}
 \oplus
S(-3)^4&\text{if $(\mu(I_1(\varphi)),\mu(I_2(C)))=(5,5)$}\\
0&\to &
S(-5)^2& \to&
S(-2)^{d_1-3}
  \phantom{\oplus
S(-3)^3}\oplus
S(-4)^4&\text{if $(\mu(I_1(\varphi)),\mu(I_2(C)))=(4,6)$}\\
0&\to &
S(-5)& \to&
S(-2)^{d_1-3}
 \oplus
S(-3)\oplus S(-4)^2&\text{if $(\mu(I_1(\varphi)),\mu(I_2(C)))=(4,5)$}\\
&&0&\to &
S(-2)^{d_1-3}
 \oplus
S(-3)^2 &\text{if $(\mu(I_1(\varphi)),\mu(I_2(C)))=(4,4)$}\\
&&0&\to &
S(-2)^{d_1-2}
\phantom{\oplus
S(-3)^3} \oplus
S(-4)
&\text{if $(\mu(I_1(\varphi)),\mu(I_2(C)))=(3,3)$}\\
&&0&\to &
S(-2)^{d_1-1}
&\text{if $(\mu(I_1(\varphi)),\mu(I_2(C)))=(2,1)\, $}.\end{array}$$
\item[{\rm (3)}] The $S$-module $\mathcal A_{d_1-2}$ is free if and only if $\mu(I_2(C))\le 4$.
\end{itemize}
\end{theorem}

\begin{proof} We know from Theorem \ref{chart}, part (1), with $i=d_1-2$, that
$$0\to \mathcal A_{d_1-2} \lto S(-2)^{d_1+1} \stackrel{C\,^{\text{\rm T}}}{\longrightarrow}
 S(-1)^2 $$is an exact sequence of homogeneous $S$-module homomorphisms; thus,  $\mathcal A_{d_1-2}\simeq \ker C\,^{\rm T}$. We apply the results of  Section 4 in \cite{CKPU}.

The matrix $\varphi$ is an element of the space of matrices ``$\BalH_d$'' from Definition 4.3 in \cite{CKPU}.
 The group $G=\GL_3(k)\times \GL_2(k)$ acts on $\BalH_d$; and, in Theorem 4.9 of \cite{CKPU}, $\BalH_d$ is decomposed  into $11$ disjoint orbits under the action of $G$: $\BalH_d=\bigcup_{\#\in {\rm ECP}} \DO^{\Bal}_{\#}$. 
These orbits are parameterized by a poset called the Extended Configuartion Poset (ECP).
The value of $(\mu(I_1(\varphi)),\mu(I_2(C)))$, as a function of $\#$ with $\varphi\in \DO^{\Bal}_\#$, is given, in part (2) of Lemma 4.10 in \cite{CKPU}, to be 
$$\begin{array}{cclll}
(\mu(I_1(\varphi)),\mu(I_2(C)))=(6,6)&\iff&\#=(\emptyset,\mu_6)\\
(\mu(I_1(\varphi)),\mu(I_2(C)))=(5,6)&\iff&\#=(\emptyset,\mu_5)\\
(\mu(I_1(\varphi)),\mu(I_2(C)))=(5,5)&\iff&\#=(c,\mu_5)\\
(\mu(I_1(\varphi)),\mu(I_2(C)))=(4,6)&\iff&\#=(\emptyset,\mu_4) \\
(\mu(I_1(\varphi)),\mu(I_2(C)))=(4,5)&\iff&\#=(c,\mu_4)\\
(\mu(I_1(\varphi)),\mu(I_2(C)))=(4,4)&\iff&\#=(c,c),\ &\text{or}\  (c:c)\\
(\mu(I_1(\varphi)),\mu(I_2(C)))=(3,3)&\iff&\#=(c,c,c),\ &\text{or}\ (c:c,c),\ &\text{or}\ (c:c:c)\\
(\mu(I_1(\varphi)),\mu(I_2(C)))=(2,1)&\iff&\#=\mu_2\, .\\
\end{array}$$ We notice that when $d_1$ is small, then it is not possible for $(\mu(I_1(\varphi)),\mu(I_2(C)))$ to take on all of the values listed so far. Indeed, the entries of $\varphi$ are elements of the vector space of homogeneous forms of degree $d_1$ in $2$ variables; consequently, $\mu(I_1(\varphi))\le d_1+1$. On the other hand, this is the only constraint that a small value for $d_1$  imposes  on the pair $(\mu(I_1(\varphi)),\mu(I_2(C)))$, as is shown in Proposition~4.21 of \cite{CKPU}. This completes the proof of (1).

Now we prove (2).
For each $\#$ in ECP, there is a canonical matrix
$C_{\#}$ with the property that if $\varphi'\in \DO_{\#}^{\Bal}$ and $C'$ is the partner of $\varphi'$ in the sense of (\ref{3.11}), then $C_{\#}$ may be obtained from $C'$ by a linear change of variables in $S$, elementary row and column operations, and the suppression of zero rows; furthermore, $I_\ell(C')=I_{\ell}(C_{\#})$, and if
$X_{\#}=\left[\begin{smallmatrix} C_\#\\0\end{smallmatrix}\right]$ is the matrix with the same number of rows as $C'$, then
$\ker C'\,^{\rm T}\simeq \ker X_\#^{\rm T}$. We now record the matrices $C_{\#}$ as given in   \cite[Lemma 4.10, part (1)]{CKPU}:
$$\begin{array}{rcl rcl rcl rcl}  C_{(\emptyset,\mu_6)}&=&\left [\begin{smallmatrix} T_1&0\\T_2&0\\T_3&0\\0&T_1\\0&T_2\\0&T_3\end{smallmatrix}\right ],\ \  & C_{(\emptyset,\mu_5)}&=&\left [\begin{smallmatrix} T_1&T_3\\T_2&0\\T_3&0\\0&T_1\\0&T_2\end{smallmatrix}\right ],\ \
&
C_{(c,\mu_5)}&=&\left [\begin{smallmatrix} T_1&0\\T_2&0\\0&T_1\\0&T_2\\0&T_3\end{smallmatrix}\right ],\ \    &
C_{(\emptyset,\mu_4)}&=&\left [\begin{smallmatrix} T_1&0\\T_2&T_1\\T_3&T_2\\0&T_3\end{smallmatrix}\right ],\ \
\\
C_{(c,\mu_4)}&=&\left [\begin{smallmatrix} T_1&T_2\\T_2&0\\0&T_1\\0&T_3\end{smallmatrix}\right ],\ \  &
C_{c,c}&=&\left [\begin{smallmatrix} T_1&0\\0&T_1\\T_2&T_2\\0&T_3\end{smallmatrix}\right ],\ \  &
 C_{c:c}&=&\left [\begin{smallmatrix} T_1&0\\T_2&T_3\\0&T_1\\0&T_2\end{smallmatrix}\right ],\ \   &
 C_{c,c,c}&=&\left [\begin{smallmatrix} T_1&T_1\\T_2&0\\0&T_3\end{smallmatrix}\right ],\\
C_{c:c,c}&=&\left [\begin{smallmatrix} T_1&0\\T_2&T_3\\0&T_2\end{smallmatrix}\right ],\ \  &
C_{c:c:c}&=&\left [\begin{smallmatrix} T_1&T_2\\T_2&T_3\\0&T_1\end{smallmatrix}\right ],\ \   &
C_{\mu_2}&=&\left [\begin{smallmatrix} T_1&T_2\\T_2&T_3\end{smallmatrix}\right ].
\end{array}$$
One   readily checks that the kernel of $X_{\#}^{\rm T}$ is as claimed for each choice of $\#$. This completes the proof of (2). Assertion (3) follows immediately from (2).
\QED
\end{proof}

\begin{remark}\label{signif} We explain the names of the elements of ECP.
Adopt Data \ref{33data} with   $k$   algebraically closed. Let $\mathcal C_{\varphi}$ be the curve and
$\eta_{\varphi}: \mathbb P^1_k\to \mathcal C_{\varphi}$ be the parameterization
described in Remark \ref{curve}.

\begin{itemize}

\item[{\rm (1)}]  If $\varphi\in \DO_{\mu_2}^{\Bal}$, then $\eta_{\varphi}$ is a birational parameterization of $\mathcal C_{\varphi}$ if and only if $d_1=1$; in this case, $\mathcal C_{\phi}$ is nonsingular.

\item[{\rm (2)}] If $d_1$ is a prime integer, then $\eta_{\varphi}$ is a birational parameterization of $\mathcal C_{\varphi}$ if and only if $\mu(I_1(\varphi))$ is at least $3$; see \cite[0.11]{CKPU}.

\end{itemize}

\noindent For the rest of our remarks, we assume that  $\eta_{\varphi}$ is a birational parameterization of $\mathcal C_{\varphi}$. (If one starts with a non-birational parameterization, then one can always re-parameterize in order to obtain a birational parameterization.)

\begin{itemize}
\item[{\rm (3)}]If $\varphi\in \DO^{\Bal}_\#$, for some $\#\in {\rm ECP}$, and $\eta_{\varphi}$ is a birational parameterization of $\mathcal C_{\varphi}$, then  the number of $c$'s that appear in the name of $\#$ is equal to the number of singularities of multiplicity $d_1$ on, or infinitely near, $\mathcal C$. (In \cite{CKPU}, the degree of $\mathcal C$ is $d=2c$; in the present language, the degree of $\mathcal C$ is $d=2d_1$.)
In particular,
\begin{itemize}
\item[{\rm(a)}] The $S$-module $\mathcal A_{d_1-2}$ is free and isomorphic to $S(-2)^{d_1-2}  \oplus
S(-4)$ if and only if there are exactly $3$ singularities of multiplicity $d_1$ on, or infinitely near, $\mathcal C$.
\item[{\rm(b)}] The $S$-module $\mathcal A_{d_1-2}$ is free and isomorphic to $S(-2)^{d_1-3}
 \oplus
S(-3)^2 $ if and only if there are exactly $2$ singularities of multiplicity $d_1$ on, or infinitely near, $\mathcal C$.
\item[{\rm(c)}] The $S$-module $\mathcal A_{d_1-2}$ is free if and only if there are at least $2$ singularities of multiplicity $d_1$ on, or infinitely near, $\mathcal C$.
\end{itemize}
\item[{\rm (4)}] If $\#$ is an element of ECP, then the punctuation describes the configuration of multiplicity $d_1$ singularities on $\mathcal C$. A colon indicates an infinitely near singularity, a comma indicates a different singularity on the curve, and $\emptyset$ indicates that there are no  multiplicity $d_1$ singularities on $\mathcal C_{\varphi}$.
\end{itemize}
\end{remark}

\begin{remark} Again, we compare our results with the results of \cite{BD}. In \cite{BD} the ambient hypothesis forces $d_2$ to equal either $d_1$ or $d_1+1$ (see \cite[Cor.~4.4]{BD}); and therefore, the hypothesis $d_1=d_2$ of the present section  agrees with
one of the hypotheses of \cite{BD}. The style of answer; however, is completely different. We describe {\bf all} of the possible bi-degrees for all possible   bi-homogeneous minimal generating sets for {\bf exactly one} $\mathcal A_i$; namely, $\mathcal A_{d_1-2}$, and we explain  how these bi-degrees reflect the configuration of singularities on the curve. On the other hand,   \cite{BD}  considers {\bf exactly one} configuration of singularities and for this configuration \cite{BD} gives a {\bf complete set} of minimal bi-homogeneous  generators (notice generators rather than just degrees) for the {\bf entire ideal} $\mathcal A$.   The singularities of the curve of \cite{BD} can be read from \cite[Thm.~3.11]{BD}, together with \cite[Cor.~1.9(1)]{CKPU}. The curve of \cite{BD} has four singularities: all of these singularities are on the curve, three of these singularities have multiplicity $d_1$, and the fourth singularity  has multiplicity $d_1-1$. In the language of the present section, the singularities of the curve of \cite{BD} correspond to the element $\#=(c,c,c)$ of the Extended Configuration Poset and the Data  \ref{33data} for the curves of \cite{BD} satisfies  $(\mu(I_1(\varphi)),\mu(I_2(C))$ is equal to $(3,3)$. \end{remark}

\section{An Application: Sextic curves.}\label{sextic}

The results outlined in the previous  sections suffice to provide significant information about the defining
equations for $\mathcal R$ if $d=d_1 +d_2 \leq 6$, since then $d_1 \leq 2$ (see Section \ref{2=d1}) or $d_1=d_2$ (see Section \ref{d1=d2}). We
focus on the case $d=6$, the case of a sextic curve.
The following data is in effect throughout this section.

\begin{data}\label{data7} Adopt Data \ref{data1} with $d=6$ and $k$ an algebraically closed field. Let $\eta:\mathbb P_k^1\to \mathbb P^2_k$ be the morphism determined by $\varphi$ and  $\mathcal C$ be the rational plane curve 
parameterized by $\eta$,
  as described in Remark \ref{curve}. Assume that $\mathcal C$ has degree $6$, or equivalently, that the morphism $\eta$ is birational onto its image $\mathcal C$. Let $\mathcal J$ be the ideal in $B$ which is the kernel of the composition 
$$B\twoheadrightarrow \Sym(I) \twoheadrightarrow \mathcal R\, ,$$ as described in Data \ref{data1}. 
 \end{data}

As it turns out, there is, essentially, a one-to-one correspondence between the bi-degrees of the
defining equations of $\mathcal R$ on the one hand and the types of the singularities on or
infinitely near the curve $\mathcal C$ on the other hand. Here one says that a singularity is
infinitely near $\mathcal C$
if it is obtained from a singularity on $\mathcal C$ by a sequence of quadratic transformations.
This  correspondence can be justified using the results of \cite{CKPU}. The information in Theorem \ref{sextic-table} about the bi-degrees of the defining equations of $\mathcal R$ is a  compilation of results from many places as described below.

\begin{theorem}\label{sextic-table}Adopt Data {\rm\ref{data7}}. The correspondence between the bi-degrees of the
defining equations of $\mathcal R$   and the types of the singularities on or
infinitely near the curve $\mathcal C$   is summarized in Table {\rm\ref{deg6}}. The first column gives the
possible values of $d_1,d_2$, namely
$1,5$ or $2,4$ or $3,3${\rm;} the second column lists the corresponding bi-degrees of minimal
homogeneous generators of $\mathcal J$ together with the multiplicities by which they appear,
suppressing the obvious bi-degrees $(d_1,1),(d_2,1)$
$($of the equations defining ${\rm Sym}(I)$$)$ and $(0,6)$ $($of the implicit equation of $\mathcal C$$)${\rm;} and
   the third column 
gives the multiplicities of the singularities on or infinitely near $\mathcal C$.\end{theorem}

\begin{table}
\begin{center}
\begin{tabular}{cc|c|c}
$d_1$ & $d_2$     &        equations of $\mathcal R$           &            singularities of $\mathcal C$ \\
\hline \hline
1  &   5          &      (1,5)  \  \ (2,4) \ \  (3,3)     \  \ (4,2)     &           1 of multiplicity 5 on $\mathcal C$ \\
\hline
2   &  4          &    (1,3):2   \ \      (2,2)                    &            1 of multiplicity 4 on $\mathcal C$ \\
    &             &                                            &            4 double points on or near $\mathcal C$ \\ \cline{3-4}
    &             &   (1,4):4     \ \ (2,3):3 \ \   (3,2)              &           10 double points on or near $\mathcal C$ \\
\hline
3   &  3          &       (1,4):4  \ \   (2,2):3                    &            10 double points on or near $\mathcal C$ \\
    \cline{3-4}
    &             &       (1,3) \ \ (1,4):2 \ \   (2,2):3              &             1 of multiplicity 3 on $\mathcal C$ \\
    &             &                                            &             7 double points on or near $\mathcal C$\\ \cline{3-4}
    &             &        (1,3):2  \ \   (2,2):3                   &             2 of multiplicity 3 and \\
    &             &                                            &             4 double points on or near $\mathcal C$ \\
    \cline{3-4}
    &             &        (1,2) \ \ (1,4) \ \   (2,2)               &           3 of multiplicity 3 and \\
    &             &                                            &            1 double point on or near $\mathcal C$
\end{tabular}
\medskip

\caption{{\bf The correspondence between the Rees algebra and the singularities of a parameterized plane sextic.}}\label{deg6}
\end{center}\end{table}

\begin{remark}
Notice that in Table \ref{deg6}, the constellation of $10$ double points on or infinitely
near the curve corresponds to two distinct numerical types of Rees algebras. Thus, the Rees
algebra provides a finer distinction. We would like to know a geometric interpretation of this algebraic distinction.\end{remark}

Before proving Theorem \ref{sextic-table},   we recall a few of the ingredients that are used. We make repeated  use of  Max Noether's formula for the geometric genus of an irreducible plane curve. In the special case of a rational curve $\mathcal C$ of degree $6$ it says that

\begin{equation}\label{Max} 10=\sum_{q} \binom{m_q}{2},
\end{equation}
where $q$ ranges over all singularities  on or  infinitely near   $\mathcal C$,  and $m_q$ is the multiplicity at $q$. We also make repeated use of the General Lemma of \cite{CKPU} (see \cite[1.7, 1.8, and 1.9]{CKPU}, or 
  \cite[1.1]{A86} and \cite[Lems.~1.3 and 1.5  and Prop.~1.5]{A88},  or \cite[Thm.~3]{SCG}), which we again state under the special hypotheses of this section.

\begin{theorem}\label{GeneralLemma} Adopt Data {\rm \ref{data7}}.
\begin{itemize}
\item[{\rm(1)}] If $p$ is a point on $\mathcal C$, then the multiplicity $m_p$ of $\mathcal C$ at $p$ satisfies either
$m_p=d_2$  or    $m_p\le d_1$. 
\item[{\rm(2)}] If $d_1<d_2$, then  the first column of $\varphi$ has a generalized zero if and only if $\mathcal C$ has a singularity of multiplicity $d_2$.
\item[{\rm(3)}] If $d_1=d_2=3$ and $C$ is the matrix of Data {\rm\ref{33data}}, then  the  number of singular points of
multiplicity $3$ that are either on $\mathcal C$ or infinitely near $\mathcal C$ is $6-\mu(I_2( C))$.
\item[{\rm(4)}] The infinitely near singularities of $\mathcal C$ have multiplicity at most $d_1$. 
\end{itemize}\end{theorem}

\begin{proof} Item (1) follows from parts (1) and (2) of Corollary~{1.9} of  \cite {CKPU}, and item (2) is a consequence of  part (4) of the same corollary. Item (3) is established by combining parts (1) and (4) of   \cite[3.22]{CKPU}. Notice that the equality $\deg\operatorname{gcd} I_3(A)=6-\mu(I_2(C))$ from (1) of \cite[3.22]{CKPU} holds even when $\operatorname{gcd} I_3(A)$ is a unit; see, for example, item (2) of \cite[4.10]{CKPU}. Item (4) is \cite[2.3]{CKPU}.\QED\end{proof}

\begin{proof2}We now justify the Table \ref{deg6} in detail. When $(d_1,d_2)=(1,5)$, the numerical information about the Rees ring follows from 
\cite[4.3]{HSV}, \cite[2.3]{CHW}, or \cite[3.6]{KPU1}. Moreover, there is a generalized zero in the first column of $\varphi$, hence
the curve $\mathcal C$ has a singularity of multiplicity 5, as can be seen from item (2) of  Theorem \ref{GeneralLemma}. There are no further singularities  or  infinitely near singularities of $\mathcal C$ by Max Noether's formula (\ref{Max}).

Next assume that $(d_1,d_2)=(2,4)$ and that the first column of $\varphi$ has a generalized zero. The numerical information about  the Rees algebra is contained in Table~\ref{tbl3}. Item (1) of Theorem \ref{GeneralLemma} shows that all of  the singularities of $\mathcal C$ have multiplicity $2$ or $4$, and, according to (2), $\mathcal C$ has a singularity of   multiplicity $4$. All of the infinitely near singularities of $\mathcal C$ have multiplicity $2$ by (4). The rest of the description of the singularities of $\mathcal C$ follows from (\ref{Max}).

 If $(d_1,d_2)=(2,4)$ and the first column of $\varphi$ does not have a generalized zero, then 
 the numerical information about the Rees ring is given by \cite[3.2]{bu}. Since the first column of $\varphi$ does not have a generalized zero, there is no point of multiplicity $4$ on the curve. Hence all singularities and infinitely near singularities have multiplicity $2$ and there are 10 of them by (\ref{Max}). 

Finally we consider the case $(d_1,d_2)=(3,3)$. Notice that a quadratic transformation cannot increase the multiplicity of a singularity. By item (1) of Theorem \ref{GeneralLemma}, all singularities on or infinitely near $\mathcal C$ have multiplicity either 2 or 3. Therefore, we employ (\ref{Max}), once again, to see that the last column of Table \ref{deg6} lists all possible configurations of singularities on or infinitely near $\mathcal C$. We use Theorem \ref{andy} to connect the configuration of singularities on or infinitely near $\mathcal C$ to the degrees of the defining equations of $\mathcal R$. Part (3) of Theorem \ref{GeneralLemma} shows that  the number of points of multiplicity $3$ on or infinitely near $\mathcal C$ is $6-\mu(I_2(C))$. The other invariant that is used in Theorem \ref{andy} is $\mu(I_1(\varphi))$. 
 Observe that $3\le \mu(I_1(\varphi))\le 4$. Indeed,   the first inequality holds by part (2) of  Remark \ref{signif}  because the ambient hypothesis of Data \ref{data7} guarantees that the morphism $\eta$, which  parameterizes the curve $\mathcal C$, is birational onto its image;  and the second inequality holds because the entries of $\varphi$ are cubics in two variables. We read the $T$-degrees of a minimal $S$-module generating set for $\mathcal A_1$ from part (2) of Theorem \ref{andy}. As seen above, we need only look at the rows with $3\le \mu(I_1(\varphi))\le 4$. We see that if there are no multiplicity $3$ singularities on or infinitely near $\mathcal C$, then $\mu(I_2(C))=6$ and there are $4$ minimal generators of the $S$-module $\mathcal A_1$  and each of these has  $T$-degree $4$. We record these generator degrees as $(1,4):4$  since every element of $\mathcal A_1$ has $xy$-degree $1$. Similarly, if there is exactly one multiplicity $3$ singularity on or infinitely near $\mathcal C$, then the  minimal generators of the $S$-module $\mathcal A_1$ have bi-degree $(1,3)$ and $(1,4):2$. If there are $2$ such singularities, then the generators have bi-degree $(1,3):2$, and if there are $3$ such singularities, then the generators have bi-degree $(1,2)$ and $(1,4)$.  Corollary \ref{XXX} shows that the $S$-module $\mathcal A_{\ge 2}$ is minimally generated by $3$ elements of bi-degree $(2,2)$. Degree considerations  show that in three of the cases the minimal generators of the $B$-module $\mathcal J$ have been identified. When there are $3$ singularities of multiplicity $3$ on or infinitely near $\mathcal C$, then a further calculation is necessary. As $\Sym(I)_{\le 2}\simeq B_{\le 2}$, multiplying the generator of $\mathcal A$ of bi-degree $(1,2)$ yields two linearly independent elements of $\mathcal A$ of bi-degree $(2,2)$; and therefore $\dim_k[\mathcal A/B \mathcal A_{(1,2)}]_{(2,2)}=1$.   
\QED \end{proof2}

\bigskip\noindent{\bf Acknolwedgment.} Part of this paper was written at
 Centre International de Rencontres Math\'ematiques (CIRM) in Luminy, France, while the authors participated in a Petit Groupe de Travail. The rest of the paper was written at the Mathematical Sciences Research Institute (MSRI) in  Berkeley, California, while the authors participated in the Special Year in Commutative Algebra. The authors are very appreciative of the hospitality offered by the Soci\'et\'e Math\'ematique de France and MSRI.

\end{document}